\documentclass[a4paper,11pt]{article}
\usepackage[utf8x]{inputenc}
\usepackage[english,french]{babel} 
\usepackage{ucs}
\usepackage{textcomp}
\usepackage{dsfont}
\usepackage[T1]{fontenc}
\usepackage{amsmath, amsfonts, amsthm, amssymb}
\usepackage{framed} 
\usepackage{array}
\usepackage{hyperref} 
\usepackage{siunitx}
\usepackage{enumitem}
\usepackage{cancel}
\usepackage{fullpage}
\usepackage{graphicx}
\usepackage{subfig}
\usepackage{xcolor}
\usepackage{soul}

\newtheorem{theo}{Theorem}[section]

\newtheorem{prop}[theo]{Proposition}
\newtheorem{coro}[theo]{Corollary}
\newtheorem{lemme}[theo]{Lemma}
\newtheorem{rema}[theo]{Remark}

\newcounter{defcounter}
\newenvironment{myequation}{%
\addtocounter{equation}{-1}
\refstepcounter{defcounter}

\begin{equation}}
{\end{equation}}

\newcommand{\R}{\mathbb{R}}
\newcommand{\N}{\mathbb{N}}
\newcommand{\Rd}{\mathbb{R}^d}

\newcommand{\dx}{\textrm{d}}

\newcommand{\teps}{t_{\varepsilon}}
\newcommand{\seps}{s_{\varepsilon}}

\newcommand{\neps}{n_{\varepsilon}}
\newcommand{\nepsk}{n_{\varepsilon_k}}
\newcommand{\Ieps}{I_{\varepsilon}}
\newcommand{\Iepsk}{I_{\varepsilon_k}}
\newcommand{\Iepsun}{I_{\varepsilon,1}}
\newcommand{\Iepsdeux}{I_{\varepsilon,2}}
\newcommand{\Iepstrois}{I_{\varepsilon,3}}

\newcommand{\Jeps}{J_{\varepsilon}}

\newcommand{\eps}{\varepsilon}
\newcommand{\ueps}{u_{\varepsilon}}
\newcommand{\veps}{v_{\varepsilon}}

\newcommand{\phieps}{\Phi_{\varepsilon}}

\newcommand{\rhoeps}{\rho_{\varepsilon}}
\newcommand{\intRd}{\int_{\Rd}}

\title{Survival criterion for a population subject to selection and mutations ;\\
Application to temporally piecewise constant environments}
\author{Manon Costa \thanks{Institut de Math\'{e}matiques de Toulouse ; UMR 5219, Universit\'{e} de Toulouse ; CNRS, UPS IMT, F-31062 Toulouse Cedex 9, France ({\tt manon.costa@math.univ-toulouse.fr})}, Christ\`{e}le Etchegaray \thanks{Team MONC, INRIA Bordeaux-Sud-Ouest, Institut de Math\'{e}matiques de Bordeaux, CNRS UMR 5251 $\&$ Universit\'{e} de
Bordeaux, 351 cours de la Lib\'{e}ration, 33405 Talence Cedex, France ({\tt christele.etchegaray@inria.fr})}, Sepideh Mirrahimi\thanks{Institut de Math\'{e}matiques de Toulouse ; UMR 5219, Universit\'{e} de Toulouse ; CNRS, UPS IMT, F-31062 Toulouse Cedex 9, France ({\tt Sepideh.Mirrahimi@math.univ-toulouse.fr})}}

\begin{document}

\maketitle
\selectlanguage{english}
\begin{abstract}
We study a parabolic Lotka-Volterra type equation that describes the evolution of a population structured by a phenotypic trait, under the effects of mutations and competition for resources modelled by a nonlocal feedback. The limit of small mutations is characterized by a Hamilton-Jacobi equation with constraint that describes the concentration of the population on some traits. This result was already established in \cite{Perthame2008},\cite{Barles2009},\cite{Lorz2011} in a time-homogenous environment, when the asymptotic persistence of the population was ensured by assumptions on either the growth rate or the initial data. Here, we relax these assumptions to extend the study to situations where the population may go extinct at the limit. For that purpose, we provide conditions on the initial data for the asymptotic fate of the population. Finally, we show how this study for a time-homogenous environment allows to consider temporally piecewise constant environments. 
\end{abstract}




\section{Introduction}
\subsection{Model and motivations}

In this paper, we study the asymptotic behaviour of solutions to parabolic Lotka-Volterra type equations used to model the evolutionary dynamics of a population where individuals are characterized by a phenotypic trait $x\in \Rd$. The population density $(t,x)\mapsto n(t,x)$ satisfies the integro-differential problem
\begin{equation}\label{eq:selection_mutation}
\begin{cases}
\partial_t n(t,x) - \sigma \Delta n(t,x) = n(t,x) R(x,I(t)), \, x \in \R^d, \, t> 0 , \\
I(t) = \intRd \psi(x) n(t,x) \dx x\,,\\
n (0,x) =n^0 \in L^1(\R^d),\, n^0\geq 0\,.
\end{cases}
\end{equation}
The population follows a selection-mutation dynamics, describing the interplay between ecology and evolution via the competition for resources. 
The term $R(x,I)$ models the growth rate of individuals with trait $x$ depending on the nonlocal interaction term $I(t)$. This interaction term $I$ represents the total consumption of a resource, with $\psi(x)$ being the trait-dependent consumption rate. The growth rate $R$ is then naturally assumed decreasing in $I$. Mutations are described by the Laplace term and arise with rate $\sigma$. Such macroscopic selection-mutation models can in fact be obtained from stochastic individual-based population models in a large population limit (see \cite{fournier2004},\cite{CFM08} and subsequent works).

The qualitative behavior of integro-differential selection-mutation models have been widely studied (see for instance \cite{AC.SC:04},\cite{DIEKMANN2005257},\cite{LD.PJ.SM.GR:08},\cite{Perthame2008},\cite{PJ.GR:09},\cite{Raoul2009}). These works mainly investigate the long time behavior of the solutions \cite{LD.PJ.SM.GR:08},\cite{PJ.GR:09}, the stability of stationary solutions \cite{AC.SC:04},\cite{Raoul2009} and the asymptotic behavior of the solutions for instance in the regime of small mutations \cite{DIEKMANN2005257},\cite{Perthame2008}. Here, we are mainly interested in an approach based on Hamilton-Jacobi equations with constraint that allows  to study the asymptotic solutions in the regime of small mutations and in long time. This approach, which has been developed during the last decade to study models from evolutionary biology, was first suggested in \cite{DIEKMANN2005257}. The approach was rigorously justified in  \cite{Perthame2008},\cite{Barles2009},\cite{Lorz2011} in the case of homogenous environments and then was extended to study more complex models with possible heterogeneity. Here, we will first focus on the case of homogeneous environments but extend the previous results in \cite{Perthame2008},\cite{Barles2009},\cite{Lorz2011}  to consider  general initial conditions. We will next show how this result would allow us to treat the case with a temporally piecewise constant environment.

We assume that mutations have a small effect, and we change the time scale to study the effect of mutations on the evolution of the population. More precisely, taking $\sigma=\varepsilon^2$ and making the change of variable $t\mapsto t/\varepsilon$, one obtains the rescaled problem
\begin{equation}\label{eq:neps}
\begin{cases}
\partial_t \neps - \varepsilon \Delta \neps = \frac{1}{\varepsilon} \neps R(x,\Ieps(t)), \, x \in \R^d, \, t> 0 , \\
\neps (t=0) =\neps^0 \in L^1(\R^d),\, \neps^0\geq 0\, ,
\end{cases}
\end{equation}
\begin{equation}\label{eq:Ieps}
\Ieps(t) = \int_{\R^d} \psi(x) \neps(t,x) \dx x\, .
\end{equation}
The study of the asymptotic solutions as $\varepsilon \rightarrow 0$ has been carried out using a Hamilton-Jacobi approach   \cite{Perthame2008},\cite{Barles2009},\cite{Lorz2011}. With this scaling, the selection is fast compared to the diversification of traits arising from mutations. As a consequence, we expect that as $\varepsilon \rightarrow 0$, $\neps(t,\cdot)$ concentrates as a Dirac mass which evolves in time. A classical method to study such asymptotic solutions consists in making the Hopf-Cole transformation
\begin{equation*}
\neps(t,x) = e^{\frac{\ueps(t,x)}{\varepsilon}}\,.
\end{equation*}
The problem \eqref{eq:neps} then rewrites on $\ueps$ as
\begin{equation}\label{eq:ueps}
\begin{cases}
\partial_t \ueps - \varepsilon \Delta \ueps = | \nabla \ueps |^2 + R(x,\Ieps(t)),\quad x\in \Rd,\, t> 0,\\
\ueps(t=0) = \varepsilon \ln \neps^0\,.
\end{cases}
\end{equation}
In \cite{Perthame2008},\cite{Barles2009},\cite{Lorz2011}, the authors establish the convergence, up to a subsequence, of $(\ueps)_{\varepsilon}$ towards a function $u$ which is solution of a Hamilton-Jacobi equation in the viscosity sense \cite{crandall1992user},\cite{barles1994solutions}, for different sorts of growth rates. In those earlier works, the assumptions ensured the persistence of the population at the limit: in \cite{Perthame2008},\cite{Barles2009}, the growth rate was bounded and everywhere positive for a non-zero small enough total population size. In \cite{Lorz2011}, the growth rate was assumed concave, and the initial condition was taken such that the population was viable.\par 
In this paper, we relax these assumptions to take into account more general growth functions and no strong constraint on the initial condition. In particular, the population may not be viable at initial time and may become extinct in the limit of small mutations. We provide conditions on the initial state for the asymptotic fate of the population in a constant environment. 
Moreover, our result allows us to treat the case where the environment is piecewise constant in time, extending in this way the Hamilton-Jacobi framework to models involving sudden variations of the environment. 
Let us introduce the problem under study for a temporally piecewise constant environment. 

For $\mathcal{E}$ a discrete space, consider $e:\R_+ \rightarrow \mathcal{E}$ a piecewise constant function describing the environment. It is equivalently defined by the increasing sequence $(T_i)_{i \in \N}$ in $\R_+$ and the sequence $(e_i)_{i \in \N}$ in $\mathcal{E}$, such that for all $t\in \R_+$, there exists $j\in \N$ such that $T_j \leq t < T_j +1$ and $
e(t) = e_j$. Now, while still assuming that mutations have a small effect by taking $\sigma = \varepsilon^2$ in \eqref{eq:selection_mutation}, we also consider that the environment varies slowly compared to the birth and death events. As a consequence, the growth rate now writes $R(x,e(\varepsilon t),I(t))$. Therefore, using the same change of variables as before, $t\mapsto t/ \varepsilon$, one obtains the rescaled problem in a temporally piecewise constant environment, that writes

\begin{equation}\label{eq:neps_env}
\begin{cases}
\partial_t \neps - \varepsilon \Delta \neps = \frac{1}{\varepsilon} \neps R(x,e(t),\Ieps(t)), \, x \in \R^d, \, t> 0 , \\
\Ieps(t) = \int_{\R^d} \psi(x) \neps(t,x) \dx x\,,\\
\neps (t=0) =\neps^0 \in L^1(\R^d),\, \neps^0\geq 0\, ,
\end{cases}
\end{equation}
and from the Hopf-Cole transformation $\neps(t,x) = e^{\frac{\ueps(t,x)}{\varepsilon}}$, we can write the corresponding problem on $\ueps$: 
\begin{equation}\label{eq:ueps_env}
\begin{cases}
\partial_t \ueps - \varepsilon \Delta \ueps = | \nabla \ueps |^2 + R(x,e(t),\Ieps(t)),\quad x\in \Rd,\, t> 0,\\
\Ieps(t) = \int_{\R^d} \psi(x) \neps(t,x) \dx x\,,\\
\ueps(t=0) = \varepsilon \ln \neps^0\,.
\end{cases}
\end{equation}

In the past years, several articles have treated the evolutionary dynamics of  populations in time-varying environments. In \cite{LORENZI2015166},\cite{LA.PB.GF:19}, the authors study closely related selection-mutation models with time varying environments, but considering particular forms of growth rates. They prove that semi-explicit solutions or explicit solutions exist for this type of problems.  
 In \cite{LA.PB.GF:19}  the variation of the environment results indeed from a time dependent administration of cancer drugs. In this work, the authors also seek for numerical optimal controls corresponding to the most efficient drug delivery schedules.  
In \cite{MIRRAHIMI201541},  a similar model with the same scaling as in \eqref{eq:neps} has been studied but considering an environment which oscillates in time with a rescaled period $1/\varepsilon$, which means that  in the original time scale the period of the oscillations is of order $1$. The  Hamilton-Jacobi limit is then rigorously justified, where the limiting growth rate in the concave case derives from a homogenization process. In the long-time asymptotics, the dominant trait maximizes this homogenized growth rate. In \cite{SF_SM}, the authors first study the long time asymptotic of the selection-mutation problem in a time-periodic environment, before studying the small mutations scaling. In this framework, the solutions converge towards a Dirac mass while the population size varies periodically in time.
The limiting problem then describes the adaptation of the trait to the averaged environment over a period of fluctuations. 
In a very recent article \cite{Carrere_Nadin_pprint} the authors study  a selection-mutation model  with a time periodic environment and they investigate the impact of the different parameters of the model on the final population size. In another recent article \cite{SF.SM:19} a similar selection-mutation model has been studied where the optimal trait varies with a linear trend but in an oscillatory manner. Using the Hamilton-Jacobi approach it is then proved that the population concentrates around a dominant trait which follows the variation of the optimal trait with the same speed but with a constant lag.

In our work, we follow the approach of \cite{MIRRAHIMI201541} and consider the rescaled evolution problem. However, in our scaling the environment varies at a slower pace than time, so that we no longer face a homogeneization problem, and the population adapts to instantaneous environmental variations. 
This work therefore provides a new perspective in the study of evolutionary dynamics in time-fluctuating environments.

\subsection{Assumptions}
We first give the general set of assumptions that we consider. In these assumptions, the main novelty is that $R$ is neither concave nor bounded (Assumption \eqref{hyp:supR}). We also consider a weaker regularity assumption on $R$ (Assumption \eqref{hyp:R_Wloc}), comparing to the works in \cite{Barles2009},\cite{Lorz2011}. Finally, we do not assume anything on the viability of the population at initial time.

\subsubsection{General assumptions}\label{hypo_normal}

\begin{description}
\item[Assumptions on $\psi$ and $R$:] there exist strictly positive constants $\psi_m,\, \psi_M$ such that 
\end{description}
\begin{myequation}\label{hyp:psi}
0<\psi_m<\psi<\psi_M<+\infty ,\, \quad \psi \in W^{2,\infty}(\R^d).
\end{myequation}
We also assume that there exists $I_M>0$ such that
\begin{myequation}\label{hyp:maxR}
\underset{x\in \Rd}{\max}\, R(x,I_M) =0 \,,
\end{myequation}
so that the population size can not grow too much. There also exist strictly positive constants $K_i$ describing the negative effect of competition on reproduction.  We have $\forall x\in \Rd,\, I \in \R_+,$
\begin{myequation}\label{hyp:dRI}
-K_1 \leq \frac{\partial R}{\partial I}(x,I) \leq -K_1^{-1}<0,
\end{myequation}
\begin{myequation}\label{hyp:supR}
-K_2- K_3 |x|^2 \leq R(x,I) \leq K_0 \quad \forall 0\leq I \leq 2I_M,
\end{myequation}
\begin{myequation}\label{hyp:DeltaPsiR}
\Delta (\psi R)  \geq -K_3\quad \forall x\in \Rd,\, I \in [0,2I_M]\,,
\end{myequation}
and furthermore,\begin{myequation}\label{hyp:R_Wloc}
\forall 0\leq I \leq 2 I_M,\,  R(\cdot,I) \in W_{\text{loc}}^{2,\infty}(\Rd)\,.
\end{myequation}
The three previous assumptions are more general than the ones used in previous works \cite{Barles2009},\cite{Lorz2011}. In those works $R$ was assumed to be either concave or bounded. Finally, we assume that the space of traits having a positive growth rate in the absence of competition is compact:
\begin{myequation}\label{hyp:R+compact}
\{x\in \Rd,\, R(x,0) \geq 0 \} \text{ is compact in } \Rd\,.
\end{myequation}

\begin{description}
\item[Assumptions on the initial condition:] there exists $I_m$ such that
\end{description}
\begin{myequation}\label{hyp:I0}
0< I_m \leq \Ieps(0) \leq I_M, 
\end{myequation}
and there exist strictly positive constants $A_{i},\,B_{i}>0$, with $i\in \{1,2\}$, such that 
\begin{myequation}\label{hyp:n0}
\neps^0(x) = e^{\frac{\ueps^0(x)}{\varepsilon}},\, \text{ with } 
-A_1-B_1|x|^2\leq \ueps^0(x) \leq A_2 -B_2|x|\,.
\end{myequation}
Note that the right-hand side  is meant to control the initial population density when $|x|$ is large. The left-hand side inequality may be relaxed (see Remark \ref{rema:noyau_chaleur}). Finally, we assume that
\begin{myequation}\label{as:u0-Lip}
(\ueps^0)_{\varepsilon} \text{ is a sequence of locally uniformly Lipschitz functions converging locally uniformly to $u^0$}.
\end{myequation}


\subsubsection{Assumptions in a concave setting}\label{hypo:concave}
We also give a convergence result in a concave setting, that provides enough regularity to fully characterize the dynamics. This framework relies on the uniform concavity of $\ueps^0$ as well as on the concavity of $R$. In particular, it is possible under additional assumptions to derive the so-called \textit{canonical equation} that describes the dynamics of the unique concentration point of the population over time. 

\begin{description}
\item[Assumptions on $\psi$ and $R$:] we assume that $\psi$ is smooth and strictly positive ; $R$ is smooth
\end{description}
 and satisfies \eqref{hyp:maxR} and \eqref{hyp:dRI}. Assumption \eqref{hyp:supR} is refined to: there exist positive constants such that 
\begin{myequation}\label{hyp:supR_concave}
-K_2 - K_3 |x|^2 \leq R(x,I) \leq K_0 - K_1 |x|^2 \, \text{ for } 0\leq I \leq 2 I_M\,.
\end{myequation}
Moreover, we assume that
\begin{myequation}\label{hyp:conc_D2R}
-2\underline{K}_2 \leq D^2 R(x,I) \leq -2 \overline{K}_2 <0\,,
\end{myequation}
and
\begin{myequation}\label{hyp:DIDxR}
\left| \frac{\partial^2 R}{\partial I \partial x_i} (x,I) \right| + \left| \frac{\partial^3 R}{\partial I \partial x_i \partial x_j}(x,I) \right| \leq K_3\,.
\end{myequation}
\begin{description}
\item[Assumptions on the initial condition:] $\ueps^0$ is strictly uniformly concave: for some positive 
\end{description}
constants,
\begin{myequation}\label{hyp:conc_D2u0}
\begin{aligned}
-\underline{L}_0-\underline{L}_1|x|^2 &\leq \ueps^0(x) \leq \overline{L}_0 - \overline{L}_1 |x|^2\,,\\
-2 \underline{L}_1 &\leq D^2 \ueps^0 \leq -2 \overline{L}_1\,.
\end{aligned}
\end{myequation}
We also have that $(\Ieps^0)_{\varepsilon}$ converges to $I_0$ with
\begin{myequation}\label{hyp:conc_I0}
0< I_m \leq I_0 \leq I_M\,.
\end{myequation}
\begin{description}
\item[Assumptions for the canonical equation:] we also assume that 
\end{description}
\begin{myequation}\label{hyp:canoniqueR}
\parallel D^3  R(\cdot,I) \parallel_{L^{\infty}(\Rd)} \leq K_4
\end{myequation}
and
\begin{myequation}\label{hyp:canoniqueU}
\parallel D^3  u_\eps^0(\cdot) \parallel_{L^{\infty}(\Rd)} \leq L_2\,.
\end{myequation}

\subsection{Main results and plan of the paper}
We first prove the local uniform convergence (up to a subsequence) of $(\ueps)_\eps$ towards a continuous function $u$, and the weak convergence in the sense of measures of $(\neps)_{\varepsilon}$. Assuming convergence of $(\Ieps)_\eps$, we are also able to identify this limit as the viscosity solution of a Hamilton-Jacobi equation that may carry a constraint depending on the limiting function $I$. First, let 
\begin{equation}\label{eq:def_gammat}
\Gamma_t := \{x\in \Rd,\, u(t,x)=0 \}
\end{equation}
be the space of zeros of the limiting function $u$.

\begin{theo}\label{theo:CV_ueps}
Under Assumptions \eqref{hyp:psi}-\eqref{hyp:n0}, 
\begin{enumerate}
\item[(i)] after extraction of a subsequence, $(\ueps)_{\varepsilon}$ converges locally uniformly to a function $u\in \mathcal{C}((0,\infty)\times \Rd)$ with $u \leq 0$. Moreover, up to a subsequence, $(\neps)_{\varepsilon}$ converges weakly in the sense of measures towards $n\in L^{\infty}(\R_+;\mathcal{M}^1(\Rd))$, with $\mathcal{M}^1(\Rd)$ the space of Radon measures in $\Rd$. Finally, a.e in $t$, Supp $n(t,\cdot) \subset \Gamma_t$.
\item[(ii)] Moreover, if  we assume aditionally \eqref{as:u0-Lip},
then $u\in \mathcal{C}([0,T) \times \Rd)$ and $u(0,x)=u^0(x)$.
\item[(iii)] If $(\Ieps)_{\varepsilon}$ converges to $I$ on some time interval $[0,T)$ with $T \in (0,+\infty]$, then $u$ is a viscosity solution of the following Hamilton-Jacobi equation on $[0,T)\times \Rd$:
\begin{equation}\label{eq:HJ_theo}
\partial_t u = |\nabla u |^2 + R(x,I(t))
\,.
\end{equation}
While $I$ is lower bounded by a strictly positive constant, for a.e $t\in [0,T)$, Equation \eqref{eq:HJ_theo} is complemented with the constraint 
\begin{equation}\label{eq:HJ_theo_constaint}
\max_{x\in \Rd} u(t,x) = 0\,,
\end{equation}
and for any Lebesgue continuity point of $I$, we have that 
\begin{equation*}
\Gamma_t \subseteq \{x\in\Rd,\, R(x,I(t))=0\}\,.
\end{equation*} 
\end{enumerate}
\end{theo}

\begin{rema}\label{rema:noyau_chaleur}
 Assuming \eqref{hyp:supR},  the lower bound on $\ueps^0$ in \eqref{hyp:n0} may be relaxed to: $\exists x_0\in \Rd$, $L_0,\,M_0>0$ such that 
\begin{equation}\label{hyp2_ueps0}
\forall x \in B_{L_0}(x_0),\, \ueps^0(x) \geq -M_0\,.
\end{equation}
Then thanks to the fundamental solution of the operator $\partial_t n-\sigma \Delta n+K_3 \,x^2 n$ ,  for all $t_1>0$, there exist positive constants $D_{i}$, $i\in \{1,2,3 \}$,   such that for $\varepsilon\leq 1$,
\begin{equation}\label{eq:resultat_chaleur}
\ueps(t,x) \geq - D_1 |x|^2 - D_2 - D_3 t \quad \text{ for } (t,x) \in (t_1,+\infty)\times \Rd\,.
\end{equation}
See \cite{Barles2009} where a similar result is proved in the case where $R$ is bounded. See also \cite{alfaro2017quadratic} for the study of a closely related model with quadratic growth rate.
\end{rema}
Theorem \ref{theo:CV_ueps} shows that it is necessary to understand the asymptotic behaviour of $(\Ieps)_{\varepsilon}$. Our main result is the following theorem, that uses the convergence of $(\ueps)_{\varepsilon}$ and the one of $(\neps)_{\varepsilon}$, both obtained up to a subsequence, to provide conditions at initial time for the asymptotic extinction or persistence of the population.

\begin{theo}\label{theo:critere_asympt}
Assume \eqref{hyp:psi}-\eqref{hyp:n0}, and that $(n_{\eps}^0)_{\eps}$ converges weakly in the sense of measures to $n^0$.
\begin{enumerate}
\item[i)] If
\begin{equation}\label{eq:cond_pers}
\text{Supp } n^0 \cap \{x \in \Rd,\, R(x,0) >0 \} \neq \emptyset\,,
\end{equation}
then for every $T>0$, there exists $\underline{I}(T)>0$ and $\varepsilon(T)>0$ such that 
\begin{equation}\label{eq:mino_I}
I_{\varepsilon}(t)\ge \underline{I}(T),\quad\quad \forall t\in[0,T), \forall \varepsilon\le\varepsilon(T)\,,
\end{equation}
and $(\Ieps)_{\varepsilon}$ converges up to a subsequence a.e in $t\in \R_+$ towards a BV function $I$.
\end{enumerate}
Assume additionally \eqref{as:u0-Lip}.
\begin{enumerate}
\item[ii)] If 
\begin{equation}\label{eq:cond_ext}
\exists C>0 \text{ such that } \Gamma_{0} \subseteq \{ x\in \Rd,\, R(x,0) \leq-C\}\,,
\end{equation}
then there exists a finite and positive constant $0<T_0<+\infty$, such that $\, \lim_{\varepsilon \rightarrow 0} \Ieps(t)\vert_{(0,T_0)} =0$.
\item[iii)] If 
\begin{equation}\label{eq:cond_ext_ponct}
\Gamma_0 \subseteq \{ x \in \Rd,\, R(x,0)\leq 0\}\,,
\end{equation}
then $\forall \delta>0,\, \exists \varepsilon_{\delta} >0,\, \forall \varepsilon < \varepsilon_{\delta},\, \exists t_{\varepsilon}\in [0,T],$ such that $$ \Ieps(t_{\varepsilon}) < \delta.$$
\end{enumerate}
\end{theo}

Note that these conditions are based on the viability of the traits initially present in the population, and not on the growth rate of the individuals at initial time, since we look at $R(\cdot,0)$ rather than at $R(\cdot,\Ieps(0))$. Numerical simulations in Section \ref{sec:numerics} will illustrate these situations. \par 

The first assertion corresponds to the situation where at least a part of the initial population is viable for some strictly positive competition level. In this case we can ensure that the population size stays uniformly strictly positive, and the previously developped tools apply to prove the convergence. We emphasize that the population size may initially decrease for $\varepsilon$ small, so that the competition also decreases, up until the growth rate reaches positive values for some traits in the population (see Figure \ref{fig1:constant_env} (\ref{sub12pers:sol}, \ref{sub12pers:sol_ext}, \ref{sub12pers:Rho})). \par 
The two other assertions concern the case of a population where all individuals have a negative growth rate for any strictly positive competition level. In this case, we show that at the limit, the population size reaches zero, either punctually or during a time interval.

\begin{rema}\label{rema_concave_critique}
If $\Gamma_0 \subseteq \{ x\in \Rd,\, R(x,0)=0\}$, then the dynamics can not be identified: depending on the shape of $R$, both extinction or persistence can occur. 
\end{rema}

\begin{rema}\label{rema:cas_non_decrit}
Note that in the condition for asymptotic persistence in \textit{i)} the set $ \text{Supp } n^0$ is involved, which is a smaller set than $\Gamma_0$, which is the one involved in the condition for asymptotic extinction in \textit{ii), iii)}. Therefore, one situation is not described, namely when 
\begin{equation*}
\text{Supp } n^0 \subseteq \{x \in \Rd,\, R(x,0) \leq 0 \} \text{ and } \Gamma_0 \cap \{x \in \Rd,\, R(x,0) > 0 \} \neq \emptyset\,.
\end{equation*}
We illustrate this situation numerically in Section \ref{sec:numerics} (see Figure \ref{fig4}).
\end{rema}

Finally, we give additional results in a concave setting that was first studied in \cite{Lorz2011} in the case where the population persists at the limit. It was proved that if $R$ is concave and $(\ueps^0)_\eps$ is uniformly strictly concave, then $u$ is strictly concave. This property provides enough regularity to better understand the dynamics at play. In particular in this case, $u(t,\cdot)$ has a unique maximum point, that is there exists $\overline{x}(t)\in \R^d$ such that
$ \Gamma_t=\{\overline{x}(t) \}$. As a consequence, $\text{Supp }n(t,\cdot) =\{\overline{x}(t) \}$ which means that $n$ is a Dirac mass concentrated on the dominant trait $\overline{x}(t)$. The time evolution of $\overline{x}$ can be described by the so-called \textit{canonical equation}. Moreover, in \cite{mirrahimi2016class}, the authors show uniqueness for $u$, allowing to have a strong convergence. We combine these results with Theorem \ref{theo:critere_asympt} to extend the study to situations where the initial population is not necessarily viable, and provide an estimate of the maximal duration of extinction, after which the population grows again. 

\begin{theo}\label{theo:CV_concave}
Assume \eqref{hyp:psi}--\eqref{hyp:dRI}, \eqref{as:u0-Lip} together with \eqref{hyp:supR_concave}--\eqref{hyp:conc_I0}. 
\begin{enumerate}
\item[i)] After extraction of a subsequence, $(\ueps)_{\varepsilon}$ converges locally uniformly to a nonpositive and strictly concave function 
$u\in L_{\text{loc}}^{\infty} (\R_+;W_{\text{loc}}^{2,\infty}(\Rd)) \cap W_{\text{loc}}^{1,\infty} (\R_+; L_{\text{loc}}^{\infty}(\Rd))$. Consequently, $u(t,\cdot)$ has a unique  maximum point $\overline x(t)$ . 
\item[ii)]  If $\Gamma_0 \subseteq \{ R(\cdot,0)>0\}$, then  $(\Ieps)_{\varepsilon}$ converges to $I\in  W^{1,\infty}(\R_+)$, with $I>0$. Moreover, $(u,I)$ is the unique
 viscosity solution of \eqref{eq:HJ_theo}-\eqref{eq:HJ_theo_constaint}, combined with $u(0,\cdot)=u^0$, which is indeed a classical solution.  Furthermore, weakly in the sense of measures, for $(t,x) \in [0,T] \times \Rd$,
\begin{equation*}
\neps(t,x) \underset{{\varepsilon \rightarrow 0}}{\rightharpoonup} \overline{\rho}(t) \delta(x - \overline{x}(t))
\end{equation*}
with $\overline{\rho}(t) := \frac{I(t)}{\psi(\overline{x}(t))}$, where $I$ is implicitly defined by
\begin{equation*}
R(\overline{x}(t),I(t))=0\,.
\end{equation*}
Finally, assuming additionally \eqref{hyp:canoniqueR}-\eqref{hyp:canoniqueU}, we have that $\overline{x}\in W^{1,\infty}(\R_+;\Rd)$ and satisfies the canonical equation
\begin{equation}\label{eq:canonique_theo}
\overset{\mathbf{\cdot}}{\overline{x}} (t) = (-D^2u(t,\overline{x}(t)))^{-1}\cdot \nabla_x R(\overline{x}(t),I(t)) \,,\quad \overline{x}(0) = x_0\,.
\end{equation}
\item[iii)] If $\Gamma_0 \subseteq \{ R(\cdot,0) <0\}$, then   on $(0,\overline{T})$, with $\overline{T} := \sup \{t>0,\, R(\overline{x}(t),0)<0 \}$, $(\Ieps)_{\varepsilon}$ converges to $0$ and $u$ is the unique viscosity solution of 
\begin{equation}\label{eq:HJ_conc_ext}
\begin{cases}
\partial_t u(t,x) = |\nabla u(t,x) |^2 + R(x,0)\, \qquad x \in \R^d, \, t\in (0, \overline T),\\
u(0,x)=u^0(x)\,,
\end{cases}
\end{equation}
which is indeed a classical solution.
Moreover, $R(\overline{x}(\overline{T}),0)=0$, and assuming additionally \eqref{hyp:canoniqueR}-\eqref{hyp:canoniqueU}, we have that $\overline{x}\in W^{1,\infty}([0,\overline{T});\Rd)$ and satisfies the canonical equation
\begin{equation}\label{eq:canonique_conc_ext}
\overset{\mathbf{\cdot}}{\overline{x}} (t) = (-D^2u(t,\overline{x}(t)))^{-1}\cdot \nabla_x R(\overline{x}(t),0)\, \quad \forall t\in (0,\overline T),\quad \overline{x}(0) = x_0\,.
\end{equation}
The function $h:t\mapsto R(\overline{x}(t),0)$ is increasing on $(0,\overline{T})$ and for some constants $A_{1},A_2>0$ related to the concavity assumptions on $R$ and   $\ueps^0$, one has the following bounds
\begin{equation*}
\begin{aligned}
 A_1
 \frac{-R(\overline{x}(0),0)}{|\nabla_x R(\overline{x}(0),0)|^2}\leq  \overline{T} \leq A_2 \frac{-R(\overline{x}(0),0)}{|\nabla_x R(\overline{x}(\overline{T}),0)|^2}\,.
 \end{aligned}
\end{equation*}
\end{enumerate}
\end{theo}

Note that the final assertion states that $\overline x(\cdot)$ reaches at a finite time $\overline T$ the zone of viable traits, that is the domain where $R(\cdot,0)$ takes positive values. We expect indeed that for $t\geq \overline T$, $(I_\eps(t))_\eps$ converges to $I(t)>0$ (see Figure \ref{fig1:constant_env}). This means that, when the effect of mutations is small but nonzero, even if initially the population is maladapted and non-viable, while the population size drops drastically, still a very small population evolves gradually  towards better traits and after some time the population becomes viable and may grow again.

We now apply Theorem \ref{theo:critere_asympt} and Theorem \ref{theo:CV_concave} to the asymptotic study of \eqref{eq:ueps_env} for the evolution of a population in a temporally piecewise constant concave environment. More precisely, we focus on the case where the growth rate in \eqref{eq:neps_env} is defined by 
\begin{equation*}
R(x,e(t),\Ieps(t)) = R_1(x, \Ieps(t))\mathds{1}_{[0,T_1)} + R_2(x, \Ieps(t))\mathds{1}_{[T_1,T_2)}\,,
\end{equation*}
for some $T_{1},T_2>0$, and $R_{1},R_2$ smooth and concave. In this situation, we can naturally define $\ueps(T_{1},x)$ and $\Ieps(T_1)$ by
\begin{equation*}
\ueps(T_{1},x) := \lim_{t\rightarrow T_{1}^-} \ueps(t,x)\,, \quad \Ieps(T_1) := \lim_{t\rightarrow T_{1}^-} \Ieps(t)\,.
\end{equation*}
 The following result follows from Theorem \ref{theo:CV_concave} (and Remark \ref{rema_concave_critique}) to deal with the discontinuity at time $T_1$ and determine the asymptotic fate of the population in the new environment. 

\begin{coro}\label{theo:CV_concave_variable}
Assume that $R_{i}$, $i\in \{1,2 \}$, satisfy \eqref{hyp:maxR}-\eqref{hyp:dRI} together with \eqref{hyp:supR_concave}--\eqref{hyp:DIDxR}, and that the initial condition verifies \eqref{hyp:conc_D2u0}-\eqref{hyp:conc_I0}. Assume additionally that  $(n_{\eps}^0)_{\eps}$ converges weakly in the sense of measures to $n^0$, with $n^0$ satisfying \eqref{eq:cond_pers}. Then in $[0,T_2)$, $(\ueps)_{\varepsilon}$ converges towards a continuous function $u$, and $(\neps)_{\varepsilon}$ converges weakly in the sense of measures towards $n$, such that
\begin{enumerate}
\item[i)] on $[0,T_1)$, $(\Ieps)_{\varepsilon}$ converges to $I>0$ with 
\begin{equation*}
(u,I) \in L_{\text{loc}}^{\infty} ([0,T_{1});W_{\text{loc}}^{2,\infty}(\Rd)) \cap W_{\text{loc}}^{1,\infty} ([0,T_{1}); L_{\text{loc}}^{\infty}(\Rd)) \times W^{1,\infty}([0,T_{1}))
\end{equation*}
where $u$ is strictly concave and is the solution of the constrained Hamilton-Jacobi equation \eqref{eq:HJ_theo} associated with the growth rate $R_1$, $I$ is defined implicitly by $R(\overline{x}(t),I(t))=0$, and one has
$n(t,x) = \overline{\rho}(t) \delta_{(x-\overline{x}(t))}$ with $\overline{\rho}(t)=\frac{I(t)}{\psi(\overline{x}(t))}$. Moreover, assuming additionally that $R_1$ and $(\ueps^0)_{\varepsilon}$ satisfy Assumptions \eqref{hyp:canoniqueR}-\eqref{hyp:canoniqueU}, we have that $\overline{x}\in W^{1,\infty}([0,T_{1});\Rd)$ and satisfies the canonical equation \eqref{eq:canonique_theo} in $(0,T_{1})$ starting from $\overline{x}(0)$.
\item[ii)] If $R_2(\overline{x}(T_1),0) \leq 0$, then there is asymptotic extinction of the population, either at a single time point or on a time interval.
\item[iii)] If $R_2(\overline{x}(T_1),0) > 0$, then $(\Ieps)_{\varepsilon}$ converges towards $I>0$ on $[T_1,T_2)$ and $u$ is strictly concave with 
\begin{equation*}
(u,I) \in L_{\text{loc}}^{\infty} ([T_{1},T_2);W_{\text{loc}}^{2,\infty}(\Rd)) \cap W_{\text{loc}}^{1,\infty} ((T_{1},T_2); L_{\text{loc}}^{\infty}(\Rd)) \times W^{1,\infty}((T_{1},T_2))\,.
\end{equation*}
Moreover, if $R_2$ and $\ueps(T_1,x)$ verify \eqref{hyp:canoniqueR}-\eqref{hyp:canoniqueU}, then $\overline{x}\in W^{1,\infty}((T_{1},T_2);\Rd)$ and satisfies the canonical equation \eqref{eq:canonique_theo} in $(T_{1},T_2)$ starting from $\overline{x}(T_1)$.
\end{enumerate}
\end{coro}

We illustrate numerically this situation in Section \ref{sec:numerics}. The plan of the paper is as follows. In Section \ref{sec:cv_ueps}, we prove Theorem \ref{theo:CV_ueps} that shows the convergence up to a subsequence of $(\ueps)_{\varepsilon}$ and $(\neps)_{\varepsilon}$. In Section \ref{sec:critere}, we prove Theorem \ref{theo:critere_asympt}, the main contribution of the paper, that gives criteria for the asymptotic fate of the population. The proof of Theorem \ref{theo:CV_concave} is given in Section \ref{sec:concave}, and mainly focuses on the description of the situation of extinction. Finally, in Section \ref{sec:numerics}, we perform numerical simulations of the model, for time-homogenous and piecewise constant environments.

\section{Proof of Theorem \ref{theo:CV_ueps}}\label{sec:cv_ueps}
In this section, we give the proof of Theorem \ref{theo:CV_ueps} on the convergence up to a subsequence of $(\ueps)_{\varepsilon}$ and $(\neps)_{\varepsilon}$. In the following, we detail important bounds and regularity results that allow to pass to the limit and whose proofs differ from previous works due to our weaker assumptions. The proofs of the convergence of $(\ueps)_{\varepsilon}$ and $(\neps)_{\varepsilon}$, and of the identification of the limiting problem follow classical steps and are detailed in Appendix \ref{app:cv_u}.

\subsection{Preliminary estimates}
We first establish estimates on $\Ieps$ and $\rhoeps$.

\begin{prop}\label{prop:estimates_ieps}
Assume \eqref{hyp:psi}-\eqref{hyp:R_Wloc}, and that $0 \leq \Ieps(0) \leq I_M + C \varepsilon^2$. Then, $\exists \varepsilon_0>0$ such that $\forall \varepsilon <\varepsilon_0$, $\forall t \in \R_+$,
\begin{equation}\label{estimate_Ieps}
 0 \leq \Ieps(t) \leq 2I_M\,.
\end{equation}
Moreover, the solution $\neps$ is nonnegative for all time and there exists $\rho_M>0$ such that $\forall \varepsilon < \varepsilon_0$, $\forall t \in R_+$,
\begin{equation}\label{estimate_rhoeps}
0\leq \rho_{\varepsilon}(t)= \intRd \neps(t,x)\dx x \leq \rho_M\,.
\end{equation}
\end{prop}

The proof derives from \cite{Perthame2008}.


\subsection{Regularity of \texorpdfstring{$(\ueps)_{\varepsilon}$}{(u)eps}}
We prove now a regularity result for $\ueps$. For $T>0$, let us define $D(T)=\sqrt{A_2 + (B_2^2+K_0) T}$, and the additional sequence $(\veps := \sqrt{2D(T)^2 - \ueps})_{\varepsilon}$.

\begin{prop}\label{prop:regularizing_effect}
Assume \eqref{hyp:psi}-\eqref{hyp:n0}. Then, 
\begin{enumerate}[label=\bfseries \roman*)]
\item the sequences $(\ueps)_{\varepsilon}$ as well as $(\veps)_{\varepsilon}$ are locally uniformly bounded: there exist positive constants $F_i$, $i\in \{ 1,...,4 \}$, such that for $(t,x) \in [0,T] \times \Rd$, 
\begin{equation}\label{eq:bounds_ueps}
-F_1 t - F_2 |x|^2 \leq \ueps(t,x) \leq F_3 t -F_4 |x|\,.
\end{equation} 
Moreover, $v_\eps$ is well-defined and it satisfies, for $(t,x) \in [0,T] \times \Rd$,
\begin{equation}
\label{bound-ve}
D(T)\leq v_\eps(t,x) \leq C|x|+B(T),
\end{equation} 
with $B$ and $C$ positive constants.
\item for all $t_0>0$, $(\ueps)_{\varepsilon}$ is locally equicontinuous in time and $(\veps)_{\varepsilon}$ is locally uniformly Lipschitz in $[t_0,T]\times \Rd$. In particular, for any $L>1$ and  $(t,x) \in [t_0,T]\times B_L(0)$,
\begin{equation}\label{eqtheo:nablav}
|\nabla \veps | \leq C(T,L) \left( 1 + \frac{1}{\sqrt{t_0}} \right)\,.
\end{equation} 
\item if $(\nabla \ueps^0)_{\varepsilon}$ is locally uniformly bounded, then $(\ueps)_{\varepsilon}$ is locally equicontinuous in time and $(\veps)_{\varepsilon}$ is locally uniformly Lipschitz on $[0,T]\times \Rd$. In particular, for  any $L>1$, there exists $a(L)\leq 1$ such that for any $(t,x) \in [0,T]\times B_L(0)$,
\begin{equation*}
|\nabla \veps|(t,x) \leq C(T,L) \left(1+\frac{1}{\sqrt{t+a(L)}} \right)\,.
\end{equation*}
\end{enumerate}
\end{prop}

\begin{proof}
The proof follows arguments used in \cite{Barles2009}, see also \cite{Lorz2011},\cite{SF_SM}. The novelty here lies in the weaker assumptions on $R$, that is neither bounded nor concave. As a consequence, the Lipschitz bounds are more difficult to obtain. We detail here the proof of these bounds, while the rest of the proof is postponed to Appendix \ref{appendix:prop_regularizing}.

\paragraph{Regularizing effect in space.} Neglecting the subscript $\varepsilon$ here for the simplification of notations, denote $v = \sqrt{2D(T)^2 - u}$, that satisfies on $[0,T]\times \Rd$
\begin{equation}\label{eq:reg_v}
\partial_t v - \varepsilon \Delta v - \left[\varepsilon \frac{1}{v} -2v \right] |\nabla v |^2 = - \frac{R(x,I)}{2v}\,.
\end{equation}
Write $p=\nabla v$. Differentiating \eqref{eq:reg_v} with respect to $x$, multiplying by $\frac{p}{| p|}$, we obtain that

\begin{equation}\label{eq:reg_p2_bis}
\begin{split}
\partial_t |p| - \varepsilon \Delta |p| - 2\left[\varepsilon \frac{1}{v} -2v\right] p \cdot \nabla |p| + \left[ \varepsilon \frac{1}{v^2} +2 \right] |p|^3  -  \frac{1}{2v^2}R(x,I)|p| + \frac{1}{2v} \nabla_x R \cdot \frac{p}{|p|}\leq0\,.
\end{split}
\end{equation}
Now, we use Assumption \eqref{hyp:R_Wloc} to obtain a local lower-bound on $ \nabla_x R$: for any $L>0$, there exists $\overline{K}_L>0$ such that for any $0\leq I\leq 2 I_M$, $\parallel \nabla_x R(\cdot,I) \parallel_{L^\infty(B_L(0))} \leq \overline{K}_L$. We deduce thanks to \eqref{bound-ve} that,  on $[0,T]\times B_L(0)$,
$$ \frac{1}{2v} \nabla_x R \cdot \frac{p}{|p|} \geq -\frac{\overline{K}_{L}}{2v}> -\frac{\overline{K}_{L}}{2D(T)} .$$
Moreover, thanks to \eqref{hyp:supR} and to \eqref{bound-ve},  we have that
\begin{displaymath}
\begin{aligned}
&- \frac{1}{2v^2} R |p| > - \frac{K_0}{2v^2} |p| > - \frac{K_0}{2D(T)^2} |p|.
\end{aligned}
\end{displaymath}
As a consequence, for some constants $A_3>0$, $D_1(T) >0$, using the bounds \eqref{bound-ve} on $\veps$, and for $\theta(T,L)$ large enough, we have that $\forall (t,x) \in [0,T] \times B_L(0)$,
\begin{equation}\label{ineq:appendix_regularizing}
\partial_t |p| - \varepsilon \Delta |p| - \left[A_3 |x| + D_1(T) \right] \left| |p|\cdot \nabla |p| \right|+ 2 (|p| - \theta(T,L))^3 \leq 0.
\end{equation}
We are now going to find a strict supersolution for Equation \eqref{ineq:appendix_regularizing} to obtain an upper-bound for $|\veps|$ on $[0,T]\times B_L(0)$. For $A_4>0$ to be determined later, and $(t,x) \in (0,T]\times B_{L}(0)$, define 
\begin{equation*}
z(t,x) =\frac{1}{2 \sqrt{t}} + \frac{A_4 L^2}{(L^2 - |x|^2)} + \theta(T,L).
\end{equation*}
We prove that for $A_4$ large enough, $z$ is a strict supersolution of \eqref{ineq:appendix_regularizing} in $(0,T] \times B_L(0)$. Indeed, we compute 
\small
\begin{equation*}
\begin{aligned}
&
\partial_t z - \varepsilon \Delta z - [A_3 |x| + D_1(T)] z \nabla z + 2 (z - \theta(T,L))^3 =  - \frac{1}{4 t\sqrt{t} } - \varepsilon \left(\frac{2A_4 L^2 d}{(L^2 - |x|^2)^2} + \frac{8A_4 L^2 |x|^2}{(L^2 - |x|^2)^3} \right) \\
& \hspace{1.5 cm} - [A_3 |x| + D_1(T)] \left( \frac{1}{2 \sqrt{t}} + \frac{A_4 L^2}{L^2 - |x|^2} + \theta(T,L) \right) \frac{2A_4 L^2 x}{(L^2 - |x|^2)^2} + 2 \left( \frac{1}{2 \sqrt{t}} + \frac{A_4 L^2}{L^2 - |x|^2} \right)^3 
\\
&  
\geq- \varepsilon \left(\frac{2A_4 L^2 d}{(L^2 - |x|^2)^2} + \frac{8A_4 L^4}{(L^2 - |x|^2)^3} \right) 
- [A_3 L + D_1(T)]\left(  \frac{1}{2 \sqrt{t}} + \frac{A_4 L^2}{L^2 - |x|^2} + \theta(T,L) \right) \frac{2A_4 L^3}{(L^2 - |x|^2)^2}\\
& \hspace{10cm} + \frac{3}{ \sqrt{t}  } \frac{A_4^2 L^4}{ (L^2 - |x|^2)^2} +2 \frac{A_4^3 L^6}{ (L^2 - |x|^2)^3} 
  \,,
\end{aligned}
\end{equation*}
\normalsize
using that $|x| \leq L$.
It can be shown that for $L>1$, $\varepsilon \leq 1$  and $A_4 = A_4(T)$ large enough, the right-hand side of the inequality is strictly positive, so that $z$ is a strict supersolution of \eqref{ineq:appendix_regularizing} in $(0,T] \times B_L(0)$ and for $\varepsilon \leq 1$. We next prove that 
\begin{displaymath}
|p(t,x)| \leq z(t,x) \quad \text{in } (0,T] \times B_L(0)\,.
\end{displaymath}
First, note that for $|x| \rightarrow L$ or $t \rightarrow 0$, $z(t,x)$ goes to infinity, so that $|p(t,x)|- z(t,x)$ attains its maximum at an interior point of $(0,T] \times B_L(0)$. Define now $t_m \leq T$ such that 
\begin{displaymath}
\underset{(t,x) \in (0,t_m] \times B_L(0)}{\max} \{ |p(t,x)| - z(t,x) \} =0.
\end{displaymath}
If such $t_m$ does not exist, then the result is proved. Otherwise, take $x_m$ such that $\forall (t,x) \in (0,t_m)\times B_L(0)$, 
\begin{displaymath}
|p(t,x)| - z(t,x) \leq |p(t_m,x_m)| - z(t_m,x_m)=0.
\end{displaymath}
Then, we have at this point that 
\begin{displaymath}
\begin{aligned}
\partial_t ( |p(t_m,x_m) |- z(t_m,x_m)) \geq 0,&\quad - \Delta (|p(t_m,x_m)| - z(t_m,x_m)) \geq 0,\\  
|p(t_m,x_m)| \nabla |p(t_m,x_m)| &= z(t_m,x_m) \nabla z(t_m,x_m).
\end{aligned}
\end{displaymath}
As a consequence, since $|p|$ (resp. $z$) is a subsolution (resp. strict supersolution) of \eqref{ineq:appendix_regularizing}, we obtain from this that 
\begin{displaymath}
2(|p(t_m,x_m)|-\theta(T,L))^3 - 2(z(t_m,x_m)-\theta(T,L))^3 <0\,,
\end{displaymath}
and we deduce that
\begin{displaymath}
|p(t_m,x_m)| < z(t_m,x_m)\,,
\end{displaymath}
in contradiction with the definition of $(t_m,x_m)$. Therefore, in $(0,T]\times B_L(0)$, for $L>1$, we have that
\begin{displaymath}
|p(t,x)| \leq z(t,x) = \frac{1}{2\sqrt{t}} + \frac{A_4(T) L^2}{(L^2 - |x|^2)} + \theta(T,L).
\end{displaymath}
Finally, we deduce that, on $( t_0,T]\times B_{\frac{L}{2}}(0)$,  
\begin{equation*}
|p(t,x)| \leq C(T,L) \left(1+\frac{1}{\sqrt{t_0}} \right).
\end{equation*}

\paragraph{Additional assumption on the initial condition.} We consider now that $(\nabla \ueps^0)_{\varepsilon}$ is locally uniformly bounded such that for any $L>1$, there exists $a(L)$ small enough such that 
$$
|\nabla u_\eps^0| \leq \frac{1}{2\sqrt{a(L)}},\qquad \text{in $B_L(0)$}.
$$
 For $A_5>0$ a large constant, define
\begin{equation*}
z(t,x) =\frac{1}{2 \sqrt{t+a(L)}} + \frac{A_5 L^2}{(L^2 - |x|^2)} + \theta(T,L).
\end{equation*}
We can prove similarly to the previous proof that for $A_5$ large enough, $z$ is a strict supersolution of \eqref{ineq:appendix_regularizing} in $[0,T] \times B_L(0)$. Moreover, as $|x|\to L$, $z(t,x)$ goes to infinity and for all $x\in B_L(0)$, we have that
$$
|p(0,x)| \leq z(0,x). 
$$
One can then follow similar arguments as in the previous part to show that, in $[0,T]\times B_L(0)$,
$$
|p(t,x)|\leq z(t,x),
$$ 
and hence, in $[0,T]\times B_{\frac{L}{2}}(0)$,
\begin{equation}
|p(t,x)| \leq C(T,L) \left(1+\frac{1}{\sqrt{t+a(L)}} \right)\,,
\end{equation}
\end{proof}

\section{Proof of Theorem \ref{theo:critere_asympt}}\label{sec:critere}
In this part, we prove Theorem \ref{theo:critere_asympt}. We start by treating the asymptotic persistent case $i)$. Then, we deal with the asymptotic extinction of the population on a time interval $ii)$, before giving the proof in the case of the asymptotic extinction at a time point $iii)$.

\subsection{Asymptotic persistence}
We prove now the first statement of Theorem \ref{theo:critere_asympt}. Take $T<\infty$, and consider Assumptions \eqref{hyp:psi}-\eqref{hyp:n0}, as well as \eqref{eq:cond_pers}.

First, we show that in this case, there exists a strictly positive uniform lower bound for $(\Ieps)_{\varepsilon}$. Then, we use this property to show that $\Ieps$ has bounded variations, which leads to the convergence up to a subsequence of $(\Ieps)_{\varepsilon}$.

\subsubsection{Strictly positive uniform lower bound on \texorpdfstring{$(\Ieps)_{\varepsilon}$}{Ie}}

In this section we aim at proving \eqref{eq:mino_I}. We start by proving an inequality related to the time derivative of $ \Ieps(t)$. Let us remark that 
$$\frac{d}{dt}I_\varepsilon(t)=\varepsilon\int\Delta \psi(x)n_\varepsilon(t,x)dx +J_\varepsilon(t),$$
where
\begin{equation}
\label{def:Jeps}
J_\varepsilon(t)=\frac{1}{\varepsilon}\int \psi(x)n_\varepsilon(t,x)R(x,I_\varepsilon(t)) \dx x\,.
\end{equation}
We now consider the evolution of the negative part of $J_\varepsilon$, defined by $(\Jeps(t))_- := \max(0,- \Jeps(t))$.
\begin{lemme}\label{lemme:ineq_dJeps}
Under assumptions \eqref{hyp:psi}-\eqref{hyp:n0}, there exist two positive constants $G_1$ and $G_2$ such that, for $\eps \leq 1$,
\begin{equation}\label{ineq:Jeps_0}
\left(J_\varepsilon(0)\right)_-\le\frac{G_1}{\varepsilon}\,,
\end{equation}  
and
\begin{equation}\label{ineq:dJeps_neg}
\begin{aligned}
\frac{\dx}{\dx t } (\Jeps(t))_- 
& \leq  G_2 - \frac{K_1^{-1}}{\varepsilon} \Ieps(t) (\Jeps(t))_- \,.
\end{aligned}
\end{equation}
\end{lemme}

\begin{proof}
We start by deriving \eqref{ineq:Jeps_0}. From \eqref{hyp:supR},\eqref{hyp:I0} and \eqref{hyp:n0}, we have that
$$R(x,I_\varepsilon(0))\ge -K_2-K_3|x|^2, \quad I_\eps(0) \leq I_M, \quad \text{ and } n_\varepsilon(0,x)\le e^{\frac{A_2-B_2|x|}{\varepsilon}}.$$
For any positive constant $M$, we obtain
$$\begin{aligned}
\varepsilon J_{\varepsilon}(0)&=\int_{\{|x|<M\}} \psi(x)n_\varepsilon(x,0)R(x,I_\varepsilon(0))\dx x+\int_{\{|x|\ge M\}}  \psi(x)n_\varepsilon(x,0)R(x,I_\varepsilon(0)) \dx x\,,\\
&\ge-I_M(K_2+K_3M^2)-\psi_M \int_{\{|x|\ge M\}}  e^{\frac{A_2-B_2 |x|}{\varepsilon}}(K_2+K_3|x|^2)\dx x\,. \end{aligned}$$
The second term of the right-hand-side is small for $\eps\leq 1$ and $L$ large enough and we obtain \eqref{ineq:Jeps_0}. To obtain \eqref{ineq:dJeps_neg}, we compute the time derivative of $\Jeps$, to get 
\begin{equation*}
\begin{split}
\frac{\dx}{\dx t }\Jeps(t) 
= \underbrace{\frac{1}{\varepsilon^2} \intRd \psi R^2 \neps \dx x}_{\geq 0} + \intRd \psi R \Delta \neps \dx x + \left(\intRd \Delta \psi \neps \dx x \right) \left( \intRd \psi\neps \frac{\partial R}{\partial I}   \dx x\right) \\
+ \frac{\Jeps(t)}{\varepsilon} \intRd \psi\neps \frac{\partial R}{\partial I}  \dx x.
\end{split}
\end{equation*}
Now, we find that $\intRd \psi R \Delta \neps \dx x = \intRd \Delta (\psi R) \neps \dx x \geq - K_3 \rhoeps (t) \geq - K_3\rho_M$ using \eqref{hyp:DeltaPsiR} and Proposition \eqref{prop:estimates_ieps}. Moreover, from \eqref{hyp:psi} and \eqref{hyp:dRI},
\begin{displaymath}
\begin{aligned}
-K_1 \Ieps(t) \leq  \intRd \psi\neps \frac{\partial R}{\partial I} \dx x \leq -K_1^{-1} \Ieps(t)\quad \text{ and } \quad \left| \intRd \Delta \psi \neps \dx x \right| \leq \rhoeps(t) \parallel \Delta \psi \parallel_{\infty} ,
\end{aligned}
\end{displaymath}
so that Proposition \eqref{prop:estimates_ieps} leads to
\begin{displaymath}
\left(\intRd \Delta \psi \neps \dx x \right) \left( \intRd \psi\neps \frac{\partial R}{\partial I}   \dx x\right)  \geq -K_1 \rhoeps(t) \Ieps(t) \parallel \Delta \psi \parallel_{\infty}  \geq -2 K_1 \parallel \Delta \psi \parallel_{\infty}I_M\rho_M.
\end{displaymath}
Therefore, we obtain that

\begin{equation*}
\begin{aligned}
\frac{\dx}{\dx t }\Jeps(t) 
&\geq -G + \frac{\Jeps(t)}{\varepsilon} \intRd \psi\neps \frac{\partial R}{\partial I}  \dx x \quad \text{ for some } G\geq 0.
\end{aligned}
\end{equation*}
Finally, considering the negative part of $\Jeps$, and using \eqref{hyp:dRI}-\eqref{hyp:supR} permit to conclude.
\end{proof}

Let us now define for $I\in (0, I_M)$,
\begin{equation}\label{def:OmegaI}
\Omega_I:= \{ x \in \Rd,\, R(x,I) > 0 \}\,,
\end{equation}
and for $\varepsilon>0$ and $t\in (0,T)$, 
\begin{equation*}
\widetilde{\Omega}_{\varepsilon}(t): =  \{ x \in \Rd,\, R(x,\Ieps(t)) > 0 \}\,.
\end{equation*}
We first prove the following Lemma.

\begin{lemme}\label{lemm:pers31}
Assume \eqref{hyp:psi}-\eqref{hyp:n0}, and \eqref{eq:cond_pers}. Then, there exists $\varepsilon_0>0$, $I_0 \in (0,I_M)$ and $I_*\in (0, I_0)$ such that for all $\varepsilon<\varepsilon_0$,
\begin{equation}\label{eq:lemme31}
\Ieps(0) \geq \int_{\Omega_{I_*}} \psi(x) \neps^0(x) \dx x \geq I_0>0\,.
\end{equation}
\end{lemme}

\begin{proof}
The left-hand-side inequality always holds true. We prove that the assertion is true for some $I_*$ and $I_0$ in $(0,I_M)$. 
Consider $x_0 \in \text{Supp }n^0 \cap \{x,\, R(x,0)>0 \}$ which is non-empty thanks to \eqref{eq:cond_pers}. Then, we deduce that $R(x_0,0) >0$, and from \eqref{hyp:dRI} that there exists $I_*\in (0,I_M)$ such that $R(x_0,I_*) \geq\frac{R(x_0,0)}{2} >0$. As a consequence, 
\begin{equation*}
 \text{Supp }n^0 \cap \Omega_{I_*} \neq \emptyset\,,
\end{equation*}
and therefore $$\int_{\Omega_{I_*}} n^0(x)dx >0.$$ Since $\psi \geq \psi_m >0$, and $(\nepsk^0)_{k\rightarrow + \infty}$ converges weakly in the sense of measures towards $n^0$, we obtain that for every $\varepsilon$ small enough
$$\int_{\Omega_{I_*}}\psi(x)n_\varepsilon^0(x) \dx x \ge \frac{\psi_m}{2}\int_{\Omega_{I_*}}n^0(x) \dx x=I_0>0.$$
 Finally, we know from Assumption \eqref{hyp:dRI} that $\forall I_1<I_2$, we have $\Omega_{I_2} \subset \Omega_{I_1}$, leading to $$\int_{\Omega_{I_2}} \psi(x) \neps^0(x) \dx x < \int_{\Omega_{I_1}} \psi(x) \neps^0(x) \dx x.$$
As a consequence, taking a smaller $I_*$ does not change the inequality, and we can assume $I_*<I_0$.
\end{proof}

We now introduce two $\varepsilon$-dependent times that also depend on a small parameter $\delta< \frac{I_*}{2}$ that is fixed:
\begin{equation}
t_\varepsilon=\inf\{ t>0,I_\varepsilon(t)\le I_*-\delta\} \,,
\end{equation}
\begin{equation}
s_\varepsilon=\inf\{ t>t_\varepsilon,I_\varepsilon(t)\le I_*-2\delta\} \,.
\end{equation}
Lemma \ref{lemm:pers31} yields that these times are positive since $I_{*}<I_0$.\medskip\\

The proof of \eqref{eq:mino_I} in Theorem~\ref{theo:critere_asympt} i) is separated in different cases.
\paragraph{Case A:} if $s_\varepsilon\ge T$ for $\varepsilon\leq \varepsilon(T)$, with $\varepsilon(T)$ a small positive constant, then by definition, $\forall t\le T$, $\forall \varepsilon \leq \varepsilon(T)$, $I_\varepsilon(t)\ge I_*-2\delta$.
\paragraph{Case B:} we fix  $\frac{1}{2}<\beta<1$ and we consider two subcases. 
\subparagraph{Case B1:} if, up to a subsequence,  
\begin{equation}
\label{B1}
s_\varepsilon<T \text{ and }s_\varepsilon-t_\varepsilon\ge \varepsilon^\beta,
\end{equation}
then for all $t\in(0,s_\varepsilon)$, $I_\varepsilon(t)> I_*-2\delta$, and we deduce from Lemma \ref{lemme:ineq_dJeps} that
 $$\frac{\dx}{\dx t}\left(J_\varepsilon(t)\right)_-\le G_2-\frac{K_1^{-1}(I_*-2\delta)}{\varepsilon}\left(J_\varepsilon(t)\right)_-.$$
Let us denote $G=K_1^{-1}(I_*-2\delta) $ and remark that since $\beta<1$ and \eqref{ineq:Jeps_0}, we obtain from the Gronwall inequality that 
 \begin{equation}\label{eq:Jneg_petit}
 \left(J_\varepsilon(s_\varepsilon)\right)_-\le \left(J_\varepsilon(0)\right)_-e^{-Gs_\varepsilon/\varepsilon}+\varepsilon\frac{G_2}{G}\le H\varepsilon, \end{equation}
 where the last inequality holds for some $H>0$ and small enough $\varepsilon$.\\
 Now, let us show that $\Ieps$ is bounded by below on $(\seps,T)$.
For $t\in [s_\varepsilon,T]$, 
 $$\begin{aligned}
 I_\varepsilon(t)&=I_\varepsilon(s_\varepsilon)+\varepsilon\int_{s_\varepsilon}^t \int\Delta \psi(x)n_\varepsilon(x,u)\dx x \dx u +\int_{s_\varepsilon}^t J_\varepsilon(u)\dx u\,,\\
 &\ge I_\varepsilon(s_\varepsilon)-\rho_M||\Delta\psi||_\infty
 \varepsilon T-\int_{s_\varepsilon}^t \left(J_\varepsilon(u)\right)_- \dx u\,,\\
 &\ge \frac{I_*-2\delta}{2}-\int_{s_\varepsilon}^t \left(J_\varepsilon(u)\right)_- \dx u,
 \end{aligned}$$
 for $\varepsilon$ small enough, using \eqref{hyp:psi}.
 To conclude the proof it is sufficient to prove that 
 $$\int_{s_\varepsilon}^t \left(J_\varepsilon(u)\right)_-  \dx u< \frac{I_*-2\delta}{4}.$$
We proceed by contradiction and assume that up to a subsequence, there exists $T_\varepsilon<T$ such that 
  $$\int_{s_\varepsilon}^{T_\varepsilon} \left(J_\varepsilon(u)\right)_-  \dx u= \frac{I_*-2\delta}{4}.$$
Now using Lemma \ref{lemme:ineq_dJeps} again, $\forall t \in[s_\varepsilon,T_\varepsilon]$
 $$
 \begin{aligned}\frac{d}{dt}\left(J_\varepsilon(t)\right)_-
 &\le G_2-\frac{K_1^{-1}I_\varepsilon(t)}{\varepsilon}\left(J_\varepsilon(t)\right)_-\,,\\
  &\le G_2-\frac{K_1^{-1}\left[\frac{I_*-2\delta}{2}-\int_{s_\varepsilon}^t \left(J_\varepsilon(u)\right)_-du\right]}{\varepsilon}\left(J_\varepsilon(t)\right)_- \,,\\
  &\le G_2-\frac{K_1^{-1}(I_*-2\delta)}{4\varepsilon}\left(J_\varepsilon(t)\right)_- \,.
 \end{aligned}$$
As a consequence, $\forall s_{\varepsilon} \leq t \leq T_{\varepsilon}$, 
\begin{equation*}
(J_{\varepsilon}(t))_- \leq \left[ (J_{\varepsilon}(s_{\varepsilon}))_- - \frac{4 G_2  \varepsilon}{{K_1^{-1}(I_*-2\delta)}}\right] e^{- \frac{K_1^{-1}(I_*-2\delta) }{4  \varepsilon} (t-s_{\varepsilon})} + \frac{4 G_2  \varepsilon}{K_1^{-1}(I_*-2\delta)}.
\end{equation*}
Using \eqref{eq:Jneg_petit}, we deduce that there exists a positive constant $G_3$ such that $\forall t \in [s_{\varepsilon},T_{\varepsilon}]$, $0\leq (J_{\varepsilon}(t))_- \leq G_3\varepsilon$. 
As a consequence, we have that 
\begin{displaymath}
 0 \leq \int_{s_{\varepsilon}}^{T_{\varepsilon}} (J_{\varepsilon}(s))_- \dx s =  \frac{I_*-2\delta}{4}\leq G_3 \varepsilon T
\end{displaymath}
which leads to a contradiction as $\varepsilon\to0$. The result is then proved in this case.

\subparagraph{Case B2:} if, up to a subsequence,  
\begin{equation}
\label{B2}
s_\varepsilon<T\text{ and }s_\varepsilon-t_\varepsilon< \varepsilon^\beta\,,
\end{equation}
then we first prove that at some $\varepsilon-$dependent time, the resource consumption of individuals having a positive growth rate is uniformly bounded by below (see Lemma \ref{lemm:pers32}). Next, we prove that this is sufficient to conclude. Recalling the definition of $\Omega_{I}$ in \eqref{def:OmegaI}, we introduce a family of test functions $\varphi_{\varepsilon,I}$ such that for a given $C'>0$ and $ 2(1-\beta)<\alpha<1$,
$$
\left\{\begin{aligned} &\varphi_{\varepsilon,I}(x)=1\\
&\varphi_{\varepsilon,I}(x)=0\\
&\varphi_{\varepsilon,I}(x)\in(0,1)
\end{aligned}\right.\quad\quad 
\begin{aligned}&\text{for }x\in\Omega_{I+C'\varepsilon^{\alpha/2}}\\
&\text{for }x\in\Omega_{I}^c\\
&\text{for }x\in\Omega_{I}\setminus\Omega_{I+C'\varepsilon^{\alpha/2}}\,.\\
\end{aligned}$$
Moreover, we ask that 
$\varphi_{\varepsilon,I}\underset{\varepsilon\to0}{\longrightarrow}\mathds1_{\Omega_I}$ and that
\begin{equation}\label{eq:D2varphi}
||D^2 \varphi_{\varepsilon,I}||_{L^\infty(\Rd)} \leq \frac{C}{\varepsilon^\alpha},\end{equation}
for $C>0$. Such a sequence of functions exists since using the assumptions on $R$, for any $x\in\Omega_{I+C'\varepsilon^{\alpha/2}}$, then $d(x,\partial\Omega_I)\geq \tilde{C}\varepsilon^{\alpha/2}$ with $\tilde{C}>0$.\\
Finally, we define 
$$\Iepsun(t)=\intRd\psi(x)n_\varepsilon(t,x)\varphi_{\varepsilon,I_\varepsilon(t)}(x)\dx s.$$
We start by proving the following Lemma. 

\begin{lemme}\label{lemm:pers32}
Assume \eqref{hyp:psi}-\eqref{hyp:n0}, \eqref{eq:cond_pers} and \eqref{B2}, then there exists $\varepsilon_1(\delta)>0$ and $I_1>0$ such that for a sequence $(\tau_{\varepsilon})_{\varepsilon<\varepsilon_1(\delta)}$ with $\tau_{\varepsilon} \in (0,T]$, we have 
\begin{equation*}
\Iepsun(\tau_{\varepsilon}) \geq I_1\,,
\end{equation*}
and $\Ieps\geq I_1$ on $[0,\tau_{\varepsilon}]$.
\end{lemme}

\begin{proof}
Let us introduce for $t\in(0,T)$,
\begin{equation*}
\Iepsdeux (t) := \intRd \psi(x) \neps(t,x) (1- \varphi_{\varepsilon,I_*-2\delta} )  \dx x\,,
\end{equation*}
an approximation of the consumption rate at time $t$ of individuals that would have a negative growth rate for a competition level of $I_*-2\delta$. We obtain the result depending on the situation at time $t_{\varepsilon}$. 
\begin{enumerate}
\item[a)] First, assume that 
\end{enumerate}
\begin{equation*}
\Iepsdeux (t_{\varepsilon})  < \frac{I_*}{8}\,. 
\end{equation*}
In the following we prove that $\Iepsdeux (s_{\varepsilon})$ is still small and hence $\Iepsun (s_{\varepsilon})=\Ieps(s_\varepsilon)-\Iepsdeux (s_{\varepsilon})$ is bounded  by below by a positive constant.
To prove that, we derive an estimate on $\Iepsdeux $ for $t \in (t_{\varepsilon},s_{\varepsilon})$:
\begin{equation*}
\begin{aligned}
\frac{\dx}{\dx t } \Iepsdeux(t) &= \varepsilon \intRd \Delta(\psi (1- \varphi_{\varepsilon,I_*-2\delta})) \neps(t,x) \dx x + \frac{1}{\varepsilon} \intRd \psi \neps (1- \varphi_{\varepsilon,I_*-2\delta}) R(x,\Ieps(t)) \dx x \\
& < C_1 \varepsilon^{1-\alpha}+ \frac{1}{\varepsilon} \intRd \psi \neps (1- \varphi_{\varepsilon,I_*-2\delta}) R(x,\Ieps(t)) \dx x \,,
\end{aligned}
\end{equation*}
where we used \eqref{hyp:psi}, \eqref{estimate_Ieps} and \eqref{eq:D2varphi}. Now, $1- \varphi_{\varepsilon,I_*-2\delta}$ has its support in $\Omega_{I_*-2\delta + C'\varepsilon^{\alpha/2}}^c$. 
Moreover, in $[\teps,\seps]$, we have that $\Ieps(t) \geq I_*-2\delta$, so that \eqref{hyp:dRI}  yields
\begin{equation*}
\begin{aligned}
R(x,\Ieps(t))\leq R(x,I_*-2\delta)&< R(x,I_*-2\delta+C'\varepsilon^{\alpha/2})+ C'\varepsilon^{\alpha/2} K_1\\
&<C'K_1\varepsilon^{\alpha/2}, \quad \quad \quad \forall x\in\Omega_{I_*-2\delta + C'\varepsilon^{\alpha/2}}^c.
\end{aligned}
\end{equation*}
 As a consequence, we obtain that 
\begin{equation*}
\begin{aligned}
\frac{\dx}{\dx t } \Iepsdeux(t) 
& < C_1\varepsilon^{1-\alpha}+ C_2\varepsilon^{\alpha/2-1} \Iepsdeux(t)  \,.
\end{aligned}
\end{equation*}
The Gronwall Lemma yields that 
\begin{equation*}
\begin{aligned}
\Iepsdeux(\seps)& \leq \Iepsdeux(\teps) e^{C_2\varepsilon^{\alpha/2-1}(\seps-\teps)} + \frac{C_1\varepsilon^{1-\alpha}}{C_2\varepsilon^{\alpha/2-1}} \left(e^{C_2\varepsilon^{\alpha/2-1}(\seps-\teps)} -1 \right) \\
&\leq \frac{I_*}{8} e^{C_2 \varepsilon^{\beta+\alpha/2-1}} + \frac{C_1\varepsilon^{2-3\alpha/2}}{C_2} \left(e^{C_2 \varepsilon^{\beta+\alpha/2-1}} -1 \right)\,,
\end{aligned}
\end{equation*}
using the assumption on $\Iepsdeux(\teps)$ and \eqref{B2}. It follows that since $\beta+\alpha/2>1$ by assumption on $\alpha$, for $\varepsilon$ small enough $\Iepsdeux(s_\varepsilon)\le I_*/4$. Then,
\begin{equation*}
\begin{aligned}
\Iepsun(\seps)& =\Ieps(\seps) -  \Iepsdeux(\seps)\,,\\
&\geq I_*-2\delta - \frac{I_*}{4}>0\,,
\end{aligned}
\end{equation*}
for $\delta$ small enough, leading to the result for $\tau_{\varepsilon}=\seps$.

\begin{enumerate}
\item[b)] In the other case, assume that 
\end{enumerate}
\begin{equation*}
\Iepsdeux (t_{\varepsilon})  \geq \frac{I_*}{8}\,.
\end{equation*}
We first evaluate an approximation of $\Iepsdeux$ on $(0,t_{\varepsilon})$, that enables us to deduce that $\teps$ is small. From this we will obtain that  $\Iepsun(\teps)$ is bounded by below by a positive constant. For $t \in (0,t_{\varepsilon})$,
\begin{equation*}
\begin{aligned}
\frac{\dx}{\dx t } \Iepsdeux(t) &= \varepsilon \intRd \Delta(\psi (1- \varphi_{\varepsilon,I_*-2\delta})) \neps(t,x) \dx x + \frac{1}{\varepsilon} \intRd \psi \neps (1- \varphi_{\varepsilon,I_*-2\delta}) R(x,\Ieps(t)) \dx x\,.
\end{aligned}
\end{equation*}
Using \eqref{hyp:psi}, \eqref{estimate_Ieps} and \eqref{eq:D2varphi}, we have that the first term in the right-hand-side is smaller than $C_1\varepsilon^{1-\alpha}$ for some $C_1>0$. Moreover, note that $1- \varphi_{\varepsilon,I_*-2\delta}$ has its support in 
$$\Omega_{I_* - 2\delta +C'\varepsilon^{\alpha/2}}^c \supset \Omega_{I_* - \delta }^c $$
since  $  I_*-\delta >I_*-2\delta +C'\varepsilon^{\alpha/2}$ for $\varepsilon$ small enough. It follows that, since for all $t\in(0,\teps)$, $\Ieps(t) > I_*-\delta$, there exists $C_2>0$ such that $R(x,\Ieps(t)) < -C_2$ on $\Omega_{I_* - 2\delta +C'\varepsilon^{\alpha/2}}^c$. Note that the constant $C_2$ depends on $\delta$. Using this information we obtain that 
\begin{equation*}
\begin{aligned}
\frac{\dx}{\dx t } \Iepsdeux(t) 
& < C_1\varepsilon^{1-\alpha} - \frac{C_2}{\varepsilon} \Iepsdeux(t)  \,,
\end{aligned}
\end{equation*}
and the Gronwall Lemma combined with the estimate on $\Iepsdeux (t_{\varepsilon}) $ yields 
\begin{align*}
\frac{I_*}{8} \leq \Iepsdeux(\teps) &\leq \Iepsdeux(0) e^{-\frac{C_2}{\varepsilon} \teps} + \frac{C_1}{C_2} \varepsilon^{2-\alpha}\left( 1-e^{-\frac{C_2}{\varepsilon}t_\varepsilon}\right)\\
&\leq I_M e^{-\frac{C_2}{\varepsilon} \teps} + C_1' \varepsilon^{2-\alpha}\,,
\end{align*}
where we have used \eqref{hyp:I0}. It follows that necessarily, $\teps \leq A\varepsilon$ for some $A>0$.  \\
Let us now deduce a lower bound for $\Iepsun(\teps)$. We compute
\begin{equation*}
\begin{aligned}
\frac{\dx}{\dx t} \intRd \psi(x) \neps(t,x) \varphi_{\varepsilon,I_*-\delta}(x) \dx x &= \varepsilon \intRd \neps \Delta (\psi \varphi_{\varepsilon,I_*-\delta}) \dx x + \frac{1}{\varepsilon} \intRd \psi \neps R \varphi_{\varepsilon, I_*-\delta } \dx x \,.
\end{aligned}
\end{equation*}
Using \eqref{estimate_rhoeps} and \eqref{eq:D2varphi}, the first term on the right hand side is bounded by below by $-C_3\varepsilon^{1-\alpha}$ for some $C_3>0$. Moreover, $\varphi_{\varepsilon,I_*-\delta}$ has its support in $\Omega_{I_*-\delta}$ that is included in a compact from \eqref{hyp:R+compact}, so that $|R| < C_4$ for some $C_4 >0$. It follows that 
\begin{equation*}
\begin{aligned}
\frac{\dx}{\dx t} \intRd \psi \neps \varphi_{\varepsilon,I_*-\delta} \dx x &\geq  -C_3\varepsilon^{1-\alpha} - \frac{C_4}{\varepsilon} \intRd \psi \neps  \varphi_{\varepsilon, I_*-\delta } \dx x \,,
\end{aligned}
\end{equation*}
and the Gronwall Lemma yields that 
\begin{equation*}
\begin{aligned}
 \intRd \psi(x) \neps(\teps,x) \varphi_{\varepsilon,I_*-\delta}(x) \dx x &\geq  \intRd \psi(x) \neps^0(x) \varphi_{\varepsilon,I_*-\delta}(x) \dx x e^{-\frac{C_4}{\varepsilon} \teps} - \frac{C_3}{C_4} \varepsilon^{2-\alpha} \left(1 - e^{-\frac{C_4}{\varepsilon}\teps} \right)\,.
\end{aligned}
\end{equation*}
Now, we have that for $\varepsilon$ small, $I_*-\delta + C'\varepsilon < I_*$ and using the definition of $\varphi_{\varepsilon,I_*-\delta}$ and Lemma \ref{lemm:pers31}:
\begin{equation*}
\begin{aligned}
 \intRd \psi(x) \neps^0(x) \varphi_{\varepsilon,I_*-\delta}(x) \dx x & \geq  \int_{\Omega_{I_*-\delta + C'\varepsilon^{\alpha/2}}} \psi(x) \neps^0(x)  \dx x \,,\\
 &\geq \int_{\Omega_{I_*}} \psi(x) \neps^0(x)  \dx x \geq I_0.
\end{aligned}
\end{equation*}
Finally, from $\teps \leq A \varepsilon$, we obtain that 
\begin{equation*}
\begin{aligned}
 \intRd \psi(x) \neps(\teps,x) \varphi_{\varepsilon,I_*-\delta}(x) \dx x &\geq I_0 e^{-C_4 A } - \frac{C_3}{C_4} \varepsilon>0
\end{aligned}
\end{equation*}
for $\varepsilon$ small enough, leading to $\Iepsun(\teps)>\underline{I}>0$ for some $\underline{I}$, and the result is proved for $\tau_{\varepsilon} = \teps$. 
\end{proof}

We have derived a positive uniform lower bound for $\Ieps$ at some $\varepsilon$-dependent time interval. It remains to extend this result to obtain a uniform lower bound on the interval $[0,T]$. 
%
%
%
%
Write $E := \parallel \Delta \psi \parallel_{\infty} \rho_M$ and define 
\begin{displaymath}
\nu_{\varepsilon}=\inf \{t \geq \tau_{\varepsilon},\, \Jeps(t) \geq -(E+1) \varepsilon \}\,.
\end{displaymath}
Then, either $\nu_{\varepsilon} >T$ or $\nu_{\varepsilon} \leq T$, and we prove now the result in each situation. 

\begin{enumerate}
\item[i)] If $\nu_{\varepsilon} > T$, then for all $t \in [\tau_{\varepsilon}, T]$, $\Jeps(t) < - \left( E +1 \right) \varepsilon$, so that 
\end{enumerate}
\begin{equation*}
\frac{\dx \Ieps(t)}{\dx t} = \varepsilon \intRd \Delta\psi(x) \neps(t,x) \dx x + \Jeps(t) < -\varepsilon <0\,,
\end{equation*}
so that $\Ieps$ is strictly decreasing. From \eqref{hyp:dRI}, we deduce that for any $x \in \Rd$, 
\begin{displaymath}
\frac{\partial}{\partial t} R(x,\Ieps(t)) = \frac{\partial}{\partial I} R(x,\Ieps(t)) \times \frac{\dx \Ieps(t)}{\dx t} > K_1^{-1} \varepsilon >0\,,
\end{displaymath}
so that for all $t\in[ \tau_{\varepsilon},T]$,
\begin{displaymath}
\begin{aligned}
R(x,\Ieps(t)) &> R(x,\Ieps(\tau_{\varepsilon})) + (t-\tau_{\varepsilon}) K_1^{-1} \varepsilon \\
& \geq (t-\tau_{\varepsilon}) K_1^{-1}\varepsilon \quad \text{ on } \Omega_{\Ieps(\tau_{\varepsilon}) }\,.
\end{aligned}
\end{displaymath}
Now, for $t\in [ \tau_{\varepsilon},T]$, let us introduce
\begin{equation*}
\Iepstrois(t) := \intRd \psi \neps(t,x) \varphi_{\varepsilon,\Ieps(\tau_{\varepsilon})} \dx x \,.
\end{equation*}
In particular, note that $\Iepstrois(\tau_{\varepsilon}) = \Iepsun(\tau_{\varepsilon})$. We compute
\begin{equation*}
\begin{aligned}
\frac{\dx}{\dx t} \Iepstrois(t) &= \varepsilon \intRd \Delta (\psi\varphi_{\varepsilon,\Ieps(\tau_{\varepsilon})}) \neps(t,x) \dx x + \frac{1}{\varepsilon} \intRd \psi \varphi_{\varepsilon,\Ieps(\tau_{\varepsilon})} \neps(t,x) R(x,\Ieps(t))\dx x \\
& > - C \varepsilon^{1-\alpha} + (t-\tau_{\varepsilon})K_1^{-1}\Iepstrois(t) 
\end{aligned}
\end{equation*}
for some $C>0$, using \eqref{hyp:psi}, \eqref{estimate_Ieps} and \eqref{eq:D2varphi}. We deduce that for $t\in [ \tau_{\varepsilon},T]$, 
\begin{equation*}
\begin{aligned}
\Iepstrois(t) & \geq  \Iepstrois(\tau_{\varepsilon})  e^{\frac{K_1}{2}(t-\tau_{\varepsilon})^2 } 
 - C e^{\frac{K_1 (t-\tau_{\varepsilon})^2}{2}}  \varepsilon^{1-\alpha} \int_{\tau_{\varepsilon}}^t  e^{-\frac{K_1  (s-\tau_{\varepsilon})^2 }{2}} \dx s   \\
& \geq \Iepsun(\tau_{\varepsilon}) -CT  e^{\frac{K_1 T^2}{2}}  \varepsilon^{1-\alpha} 
 \\
&  \geq I_1   -C  T e^{K_1T^2}  \varepsilon^{1-\alpha}>0 
 \end{aligned}
\end{equation*}
for $\varepsilon < \varepsilon_0(T)$ small enough. It follows that there exists $\underline{I}(T)>0$ such that $\forall t\in (0,T)$, $\Ieps(t)\geq \Iepstrois(t) \geq \underline{I}(T)$.

\begin{enumerate}
\item[ii)] If $\nu_{\varepsilon}< T$, we can use the previous argument to show that
\end{enumerate}
\begin{equation*}
\exists \varepsilon_0(T)>0,\, {\underline{I}(T)}>0 \text{ such that } \forall \varepsilon <\varepsilon_0(T),\, \forall t \in [\tau_{\varepsilon},\nu_{\varepsilon}),\,\Ieps(t) \geq {\underline{I}(T)}\,.
\end{equation*}
Therefore, $\Ieps(\nu_{\varepsilon})  \geq {\underline{I}(T)}$, and $\Jeps(\nu_{\varepsilon}) \geq -(E+1)\varepsilon$. 
For $\varepsilon < \varepsilon_1(T)$ small enough and $t \in (\nu_{\varepsilon},T)$, we have that
\begin{equation}\label{estimate:Ineq_Ieps}
\Ieps(t) = \Ieps(\nu_{\varepsilon}) + \int_{\nu_{\varepsilon}}^t \Ieps'(s) \dx s = \Ieps(\nu_{\varepsilon}) + \int_{\nu_{\varepsilon}}^t \Jeps(s) \dx s + \mathcal{O}(\varepsilon )(t-\nu_{\varepsilon}) \geq \frac{{\underline{I}(T)}}{2} - \int_{\nu_{\varepsilon}}^t (\Jeps(s))_- \dx s,
\end{equation}
and we obtain from \eqref{ineq:dJeps_neg} that
\begin{equation*}
\begin{aligned}
\frac{\dx}{\dx t } (\Jeps(t))_- &\leq  G - \frac{K_1^{-1}}{\varepsilon} \left( \frac{{\underline{I}(T)}}{2} - \int_{\nu_{\varepsilon}}^t (\Jeps(s))_- \dx s\right) (\Jeps(t))_- .
\end{aligned}
\end{equation*}
Now, we want to show that under some conditions, $\forall  \varepsilon  < \varepsilon_0(T)$, $\int_{\nu_{\varepsilon}}^{T} (J_{\varepsilon}(s))_- \dx s \leq \frac{{\underline{I}(T)}}{4}$. Let us proceed by contradiction. Assume that this is not the case: there exists a sequence $(\varepsilon_k)_k$ in $(0,\varepsilon_1(T))$ with $\underset{k\rightarrow + \infty}{\lim} \varepsilon_k = 0$, and $ \forall k,\,\exists T_{\varepsilon_k}' < T$ such that $\int_{\nu_{\varepsilon_k}}^{T_{\varepsilon_k}'} (J_{\varepsilon_k}(s))_- \dx s = \frac{{\underline{I}(T)}}{4}$.
Then we have $\forall k\geq 0$,
\begin{equation*}
\frac{\dx}{\dx t } (J_{\varepsilon_k}(t))_- \leq  G - \frac{K_1^{-1}}{\varepsilon_k} \frac{{\underline{I}(T)}}{4}  (J_{\varepsilon_k}(t))_- \quad \forall \nu_{\varepsilon_k}\leq t \leq T'_{\varepsilon_k}.
\end{equation*}
As a consequence, $\forall k\geq 0$, $\forall \nu_{\varepsilon_k} \leq t \leq T'_{\varepsilon_k}$, 
\begin{equation*}
(J_{\varepsilon_k}(t))_- \leq (J_{\varepsilon_k}(\nu_{\varepsilon_k}))_- e^{- \frac{{\underline{I}(T)}}{4 K_1 \varepsilon_k} (t-\nu_{\varepsilon_k})} + \frac{4 G K_1 \varepsilon_k}{{\underline{I}(T)}} \left(1- e^{- \frac{{\underline{I}(T)}}{4 K_1 \varepsilon_k} (t-\nu_{\varepsilon_k})} \right),
\end{equation*}
or equivalently
\begin{equation*}
(J_{\varepsilon_k}(t))_- \leq \left[ (J_{\varepsilon_k}(\nu_{\varepsilon_k}))_- - \frac{4 G K_1 \varepsilon_k}{{\underline{I}(T)}}\right] e^{- \frac{{\underline{I}(T)} }{4 K_1 \varepsilon_k} (t-\nu_{\varepsilon_k})} + \frac{4 G K_1 \varepsilon_k}{{\underline{I}(T)}}.
\end{equation*}
Since $(J_{\varepsilon_k}(\nu_{\varepsilon_k}))_- \leq (E+1)\varepsilon_k$, we deduce that there exists a positive constant $G_1$ such that $\forall t \in [\nu_{\varepsilon_k},T_{\varepsilon_k}']$, $0\leq (J_{\varepsilon_k}(t))_- \leq G_1\varepsilon_k$.
As a consequence, we have that $$ 0 \leq \int_{\nu_{\varepsilon_k}}^{T_{\varepsilon_k}'} (J_{\varepsilon_k}(s))_- \dx s =  \frac{\underline{I}(T)}{4} \leq G_1 \varepsilon_k T\,,$$ which leads to a contradiction for $k$ large enough. Therefore, for all $\varepsilon < \varepsilon_1(T)$, we have $\int_{\nu_{\varepsilon}}^{T} (\Jeps(s))_- \dx s \leq \frac{{\underline{I}(T)}}{4}$, and from the estimate \eqref{estimate:Ineq_Ieps}, we have that $\forall \varepsilon < \varepsilon_1(T)$, for $t \in (\nu_{\varepsilon},T)$,
\begin{equation*}
\Ieps(t) \geq \frac{{\underline{I}(T)}}{4}\,,
\end{equation*}
and the result is proved. 


\subsubsection*{BV bound}
We derive now a sub-Lipschitz bound as well as a BV bound on $\Ieps$ that allow to pass to the limit after extraction of a subsequence, ending the proof of the first point of Theorem \ref{theo:critere_asympt}.

\begin{prop}\label{prop:cv_I_survie}
With the assumptions \eqref{hyp:psi}-\eqref{hyp:n0} and assuming that $\exists \underline{I}(T)>0$, $\varepsilon_0(T)>0$ such that $\forall \varepsilon<\varepsilon_0$, $\forall t\in [0,T]$, $\Ieps(t)>\underline{I}(T)$, we obtain the following locally uniform BV bound on $[0,T]$.
For $\varepsilon < \varepsilon_0 \leq 1$, and $C_1$, $G$ some positive constants, we have the sub-Lipschitz bound

\begin{equation}\label{eq:sub-lipschitz}
\frac{\dx \Ieps}{\dx t}(t) \geq - \varepsilon \left[C_1\rho_M+ \frac{G K_1}{\underline{I}(T)} \right] - (\Jeps(0))_- e^{- \frac{\underline{I}(T)}{K_1 \varepsilon} t}\,,
\end{equation}
so that we obtain the BV bound
\begin{equation}\label{eq:bv}
\begin{aligned}
\int_{0}^{T} \left| \frac{\dx \Ieps}{\dx t}(t) \right| \dx t  
 & \leq A(T) +
 \frac{K_1}{\underline{I}(T)}\left(\intRd \psi(x) \neps(0,x) R(x,\Ieps(0)) \dx x \right)_-\,,
\end{aligned}
\end{equation}
with $A(T)=2 I_M + T\left(2 C_1  \rho_M + \frac{G K_1}{\underline{I}(T)} \right)$.\newline
Consequently, after extraction of a subsequence, $(\Ieps(t))_{\varepsilon}$ converges a.e in $\R_+$ when $\varepsilon$ goes to zero to a function $I$ such that $\forall T>0$, there exists $\underline{I}(T)>0$ such that $I(t)\geq \underline{I}$ on $[0,T]$.

\end{prop}

\begin{proof}
We adapt the proof of \cite{Perthame2008}. We want to show that 

\begin{equation*}
\int_{0}^{T} \left| \frac{\dx \Ieps}{\dx t} \right|(t) \dx t < \overline{C}
\end{equation*}
for some positive constant $\overline{C}$. Writing

\begin{displaymath}
\left| \frac{\dx \Ieps}{\dx t} \right|(t)  = \left( \frac{\dx \Ieps}{\dx t}(t) \right)_+ +  \left( \frac{\dx \Ieps}{\dx t}(t) \right)_- = \frac{\dx \Ieps}{\dx t}(t) + 2 \left( \frac{\dx \Ieps}{\dx t}(t) \right)_-\,,
\end{displaymath}
we obtain that 

\begin{equation*}
\begin{aligned}
\int_{0}^{T} \left| \frac{\dx \Ieps}{\dx t}(t) \right| \dx t  
&= \Ieps(T) - \Ieps(0) + 2 \int_{0}^{T}  \left( \frac{\dx \Ieps}{\dx t}(t) \right)_- \dx t\,.
\end{aligned}
\end{equation*}
Now, on $[0,T]$,

\begin{equation}\label{eq:dIeps_BV}
\frac{\dx \Ieps}{\dx t}(t) = \varepsilon \intRd \Delta \psi(x) \neps(t,x) \dx x + \Jeps(t) \geq - C_1 \varepsilon \rhoeps(t) - (\Jeps(t))_-\,,
\end{equation} 
so that 

\begin{equation*}
\left( \frac{\dx \Ieps}{\dx t}(t) \right)_- \leq C_1 \varepsilon \rho_M + (\Jeps(t))_-\,.
\end{equation*}
Therefore, we have that

\begin{equation*}
\begin{aligned}
\int_{0}^{T} \left| \frac{\dx \Ieps}{\dx t}(t) \right| \dx t  
&\leq \Ieps(T) - \Ieps(0) + 2 T C_1 \varepsilon \rho_M+ \int_{0}^{T}  (\Jeps(t))_- \dx t\,,
\end{aligned}
\end{equation*}
and the uniform BV bound on $\Ieps$ relies on a uniform bound for $\int_{0}^{T} (\Jeps(t))_-  \dx t$. We use Lemma \ref{lemme:ineq_dJeps} together with the lower bound on $\Ieps$ to get

\begin{equation*}
\frac{\dx}{\dx t } (\Jeps(t))_- \leq G - \frac{\underline{I}(T)}{\varepsilon K_1}  (\Jeps(t))_-\,,
\end{equation*}
leading to 
\begin{equation*}
(\Jeps(t))_- \leq \varepsilon \frac{G K_1}{\underline{I}(T)} + e^{- \frac{\underline{I}(T)}{K_1 \varepsilon}t}(\Jeps(0))_- \,.
\end{equation*}
This inequality combined with \eqref{eq:dIeps_BV} give the sub-Lipschitz bound \eqref{eq:sub-lipschitz}. Finally, we also obtain that 

\begin{equation*}
\begin{aligned}
&
\int_{0}^{T} \left| \frac{\dx \Ieps}{\dx t}(t) \right| \dx t  
\leq \Ieps(T) - \Ieps(0) + \varepsilon T \left(2 C_1 \rho_M + \frac{G K_1}{\underline{I}(T)} \right) \\
& \hspace{5cm} + \frac{K_1 \varepsilon}{\underline{I}(T)} \left(1-e^{- \frac{\underline{I}(T)}{K_1 \varepsilon}T}\right)(\Jeps(0))_-\,,
\\
 & \leq 2I_M + C \varepsilon^2 + \varepsilon T \left(2 C_1  \rho_M + \frac{G K_1}{\underline{I}(T)} \right) +
 \frac{K_1 \varepsilon}{\underline{I}(T)} \left(1-e^{- \frac{\underline{I}(T)}{K_1 \varepsilon} T}\right)(\Jeps(0))_-\,.
\end{aligned}
\end{equation*}
Finally, the convergence in $\R_+$ of $(\Ieps(t))_{\varepsilon}$ up to a subsequence follows, ending the proof of Proposition \ref{prop:cv_I_survie}, and of the first assertion of Theorem \ref{theo:critere_asympt}. 
\end{proof}

\subsection{Asymptotic extinction on a time interval}
We show now the second statement of Theorem \ref{theo:critere_asympt}. We recall the assumption \eqref{eq:cond_ext} namely that $\exists C>0$ such that
\begin{equation*}
\Gamma_0= \{x\in\Rd,\, u(0,x)=0\} \subseteq \{ x\in \Rd,\, R(x,0) \leq-C\}\,,
\end{equation*}
Let us define for $\delta>0$, 
$$\mathcal{O}_{\delta}:=\{x \in \Rd,\, u(0,x) \geq -\delta \}$$ and recall that $\mathcal{O}_0=\Gamma_0$ since $u\le0$ from Theorem \ref{theo:CV_ueps} i).
Using \eqref{eq:bounds_ueps}, $\mathcal{O}_\delta$ is bounded and from the local uniform continuity of $u$, for $\delta$ and $a$ small enough, we have $\mathcal{O}_{\delta}\subset A_{C/2}:=\{ x\in \Rd,\, R(x,0) \leq -C/2\}$ with $\mathrm{dist}(\mathcal{O}_{\delta},\partial A_{C/2})\geq a>0$. \\
Moreover, by the uniform continuity in time of $u$, there exists $T_0>0$ so that 
\begin{displaymath}
\forall t\in [0,T_0),\, \left\{x \in \Rd,\, u(t,x)>-\frac{\delta}{2} \right\} \subseteq \mathcal{O}_{\delta}\,.
\end{displaymath}

\noindent Now, since $\psi$ is positive and bounded from \eqref{hyp:psi}, we write on $[0,T_0)$,
\begin{equation*}
\Ieps(t) \leq \psi_M \int_{\mathcal{O}_{\delta}} \neps(t,x)\dx x + \psi_M \int_{(\mathcal{O}_{\delta})^c} \neps(t,x)\dx x\,,
\end{equation*}
and we prove that each term separately goes to $0$, starting with the second term.
\medskip

From Proposition \ref{prop:regularizing_effect} (i), there exist positive constants $F_i$, $i\in \{ 1,...,4 \}$, such that for $(t,x) \in [0,T] \times \Rd$, $\forall \varepsilon< \varepsilon_0$,
\begin{displaymath}
-F_1 T - F_2 |x|^2 \leq \ueps(t,x) \leq F_3 T -F_4 |x|\,.
\end{displaymath}
Therefore there exist $r_0>0$ such that for all $t\in [0,T_0]$ and $|x| \geq r_0$, $\ueps(t,x)\leq -\frac{F_4}{2}|x| $.
Now,
\begin{equation}
\label{eq:1}\begin{aligned}
 \int_{\left( \mathcal{O}_{\delta}\right)^c}\neps(t,x) \dx x  
 &=  \int_{\left(\mathcal{O}_{\delta}\right)^c \cap B(0,r_0)}\neps(t,x) \dx x  +  \int_{\left(\mathcal{O}_{\delta}\right)^c \cap B(0,r_0)^c}\neps(t,x) \dx x \,\\
 &\leq  \int_{\left(\mathcal{O}_{\delta}\right)^c \cap B(0,r_0)}\neps(t,x) \dx x  +  \int_{\left(\mathcal{O}_{\delta}\right)^c \cap B(0,r_0)^c} e^{\frac{-F_4 |x|}{2\varepsilon}} \dx x \, .
 \end{aligned}
\end{equation}
Finally it remains to control the integral on $\left(\mathcal{O}_{\delta}\right)^c \cap B(0,r_0)$. Remark that on $\left(\mathcal{O}_{\delta}\right)^c$, for any $t\in[0,T_0)$, $u(t,\cdot) \leq - \frac{\delta}{2} <0$ and there exists $\varepsilon_0 >0$ small enough so that $\forall \varepsilon < \varepsilon_0$, on $\left(\mathcal{O}_{\delta}\right)^c$, $\ueps(t,\cdot) \leq - \frac{\delta}{4}$. We then deduce that
\begin{displaymath}
0\leq \int_{\left(\mathcal{O}_{\delta}\right)^c \cap B(0,r_0)}\neps(t,x) \dx x \leq  \int_{\left(\mathcal{O}_{\delta}\right)^c \cap B(0,r_0)} e^{-\frac{\delta}{4\varepsilon}} \dx x
\leq  |B(0,r_0) | e^{-\frac{\delta}{4\varepsilon}} \,.
\end{displaymath}
Combining with \eqref{eq:1},  $\int_{(\mathcal{O}_{\delta})^c} \neps(t,x)\dx x$ goes to $0$ as $\varepsilon\to 0$ for every $t\in[0,T_0)$.

\noindent We now consider $\int_{\mathcal{O}_{\delta}} \neps(t,x) \dx x$ on $(0,T_0)$. Let $\varphi_{\varepsilon} \in \mathcal{C}^{\infty}_{c,+}(A_{C/2})$ be a test function such that $\varphi_{\varepsilon} \equiv 1$ in $\mathcal{O}_{\delta}$, and such that $\parallel D^2 \varphi_{\varepsilon} \parallel_{\infty} <\frac{1}{\varepsilon}$. Such test function exists for $\epsilon$ small enough. Then, for $t\in [0,T_0)$,
$$  0\leq \int_{\mathcal{O}_{\delta}} \neps(t,x)\dx x\leq \intRd \varphi_{\varepsilon}(x) \neps(t,x) \dx x. $$ 
Moreover,
\begin{displaymath}
\begin{aligned}
\frac{\dx }{\dx t} \intRd \varphi_{\varepsilon}(x) \neps(t,x) \dx x
&= \varepsilon \intRd \Delta \varphi_{\varepsilon} \neps(t,x) \dx x + \frac{1}{\varepsilon} \intRd \varphi_{\varepsilon}(x) \neps(t,x) R(x,\Ieps(t))\dx x\,,\\
&\leq \varepsilon \parallel \Delta \varphi_{\varepsilon} \parallel_{\infty} \rho_M + \frac{1}{\varepsilon} \intRd \varphi_{\varepsilon}(x) \neps(t,x) R(x,0)\dx x,\\
&< \rho_M - \frac{C}{2\varepsilon} \intRd \varphi_{\varepsilon}(x) \neps(t,x) \dx x \,,\\
\end{aligned}
\end{displaymath}
using Assumption \eqref{hyp:dRI}.
It follows that on $(0,T_0)$, and for $\overline{C}>0$ some constant whose value can change from line to line, we have
\begin{equation*}
\begin{aligned}
\intRd \varphi_{\varepsilon}(x) \neps(t,x) \dx x &< \intRd \varphi_{\varepsilon}(x) \neps(0,x) \dx x e^{-\frac{C}{2\varepsilon}t} + \frac{2\rho_M}{C} \varepsilon \left(1- e^{-\frac{C}{2\varepsilon}t} \right)\,,\\
& <\overline{C} \left( e^{-\frac{C}{2\varepsilon}t} +  \varepsilon \right) \underset{\varepsilon \rightarrow 0}{\rightarrow}0\,,
\end{aligned}
\end{equation*}
using Assumption \eqref{hyp:I0}.
This concludes the proof of Theorem \ref{theo:critere_asympt} ii).

\subsection{Asymptotic extinction at a time point}
We show now the last assertion of Theorem \ref{theo:critere_asympt}, namely that if 
\begin{equation*}
\Gamma_0 \subseteq \{ x \in \Rd,\, R(x,0)\leq 0\}\,,
\end{equation*}
then, for any $T<+\infty$,
\begin{equation}\label{eq:extinction_ponctuelle}
\forall \lambda>0,\, \exists \varepsilon_{\lambda} >0,\, \forall \varepsilon < \varepsilon_{\lambda},\, \exists t_{\varepsilon}\in [0,T],\, \Ieps(t_{\varepsilon}) < \lambda\,.
\end{equation}
Recall that the space of possible concentration points of the population at the limit writes $\Gamma_t:= \{x \in \Rd,\, u(t,x)=0 \}$.
We prove the result with a contradiction argument. If \eqref{eq:extinction_ponctuelle} is not true, then $\exists T,\, \lambda >0$, $\exists (\varepsilon_k)_k \rightarrow 0$ such that $\forall t \in [0,T]$, $\Iepsk(t) \geq \lambda > 0$. Then, following 
Proposition \ref{prop:cv_I_survie}, $(\Iepsk)_{k}$ converges on $(0,T)$ towards a function $I:t \mapsto I(t) \geq \lambda$, and by Theorem \ref{theo:CV_ueps}, $u$ is then a viscosity solution of the following constrained Hamilton-Jacobi problem: 
\begin{equation*}
\left \lbrace
\begin{aligned}
&\partial_t u = |\nabla u |^2 + R(x,I(t))\,,\\
&\max_x u(t,x) = 0\,,\\
&\text{Supp }n(t,\cdot) \subset \Gamma_t \subset \{R(\cdot,I(t))=0 \} \text{ for all Lebesgue point of }I\,.
\end{aligned}\right.
\end{equation*}
Now, by assumption \eqref{hyp:dRI}, 
\begin{displaymath}
\forall x \in \Gamma_0,\, R(x,\lambda) < R(x,0) + \sup \partial_I R(x,I) \lambda\leq - K_1^{-1} \lambda <0\,.
\end{displaymath}
By continuity of $u$, there exists $\delta >0$ small enough so that 
\begin{displaymath}
\{ u(0,\cdot) > -2\delta \} \subset \left\{R(\cdot,\lambda) < -\frac{\lambda}{2 K_1} \right\}\,,
\end{displaymath}
and there exists $t_1>0$ small enough so that 
\begin{displaymath}
\forall t \in [0,t_1),\, \{ u(t,\cdot) > -\delta \} \subset \{ u(0,\cdot) > -2\delta \}\,.
\end{displaymath}
Therefore, for $t\in (0,t_1)$ and $x \in \Gamma_t$, we have that 
\begin{displaymath}
0=R(x,I(t)) \leq R(x,\lambda) < -\frac{\lambda}{2 K_1} <0\,,
\end{displaymath}
leading to a contradiction.

\section{The concave case}\label{sec:concave}
In this part, we give the proofs of Theorem \ref{theo:CV_concave} as well as Corollary \ref{theo:CV_concave_variable} that deal with a constant or piecewise constant concave environments. 

\subsection{Proof of Theorem \ref{theo:CV_concave} }

i) One can verify that Assumptions  \eqref{hyp:supR_concave}--\eqref{hyp:conc_I0} imply \eqref{hyp:supR}--\eqref{hyp:n0}. Therefore, one can use Theorem \ref{theo:CV_ueps} which implies that $(\ueps)_{\varepsilon}$ converges locally uniformly to a function $u\in \mathcal{C}([0,\infty)\times \Rd)$ with $u \leq 0$. Moreover, in \cite{Lorz2011} it is proved that under the assumptions of Theorem \ref{theo:CV_concave}, the function $u$ is indeed strictly concave and $u\in L_{\text{loc}}^{\infty} (\R_+;W_{\text{loc}}^{2,\infty}(\Rd)) \cap W_{\text{loc}}^{1,\infty} (\R_+; L_{\text{loc}}^{\infty}(\Rd))$. Consequently, $u(t,\cdot)$ has a unique  maximum point $\overline x(t)$.\\

ii)  Note first that $\text{Supp } n^0=\Gamma_0$ since $\Gamma_0$ has a unique point. Therefore, if $\Gamma_0 \subseteq \{ R(\cdot,0)>0\}$, then   $(\Ieps)_{\varepsilon}$ converges to $I\in  W^{1,\infty}(\R_+)$, with $I>0$, thanks to Theorem \ref{theo:critere_asympt}-(i). Moreover, thanks to Theorem \ref{theo:CV_ueps}-(iii) $u$ is a viscosity solution to \eqref{eq:HJ_theo}-\eqref{eq:HJ_theo_constaint}, together with $u(0,\cdot)=u^0$. In \cite{mirrahimi2016class} it is proved that the viscosity solution $(u,I)$ to such Hamilton-Jacobi equation is indeed unique and smooth. Moreover, under the assumptions of the theorem, the canonical equation \eqref{eq:canonique_theo} can be derived similarly to  \cite{Lorz2011}.\\

iii) Since $\Gamma_0 \subseteq \{ R(\cdot,0)<0\}$, thanks to Theorem   \ref{theo:critere_asympt}-(ii), there exists $T_0>0$ such that $\, \lim_{\varepsilon \rightarrow 0} \Ieps(t)\vert_{(0,T_0)} =0$. We define $T_m$ to be the maximal point such that this property holds. It is immediate that $T_m\geq \overline T$. Otherwise, one can extend $T_m$ to greater values thanks to  Theorem   \ref{theo:critere_asympt}-(ii).  We expect indeed that $T_m = \overline T$.  We next, deduce that $\forall t \in (0,\overline{T}),\, \lim_{\varepsilon \rightarrow 0} \Ieps(t) =0$. Equations \eqref{eq:HJ_conc_ext} and \eqref{eq:canonique_conc_ext} can then be derived similarly to \eqref{eq:HJ_theo} and \eqref{eq:canonique_theo}.

Next, we study $h$ defined by $h(t)= R(\overline{x}(t),0)$. We compute
\begin{equation*}
 h'(t) = \nabla_x R(\overline{x}(t),0) \overset{\cdot}{\overline{x}}(t) = \nabla_x R(\overline{x}(t),0) (-D^2 u(t, \overline{x}(t)))^{-1}   \nabla_x R(\overline{x}(t),0) \,,
\end{equation*}
so that $h$ is increasing while $|\nabla_x R(\overline{x}(t),0)| \neq 0$, and non-decreasing in the general case.
As a consequence, since thanks to the assumptions \eqref{hyp:maxR}--\eqref{hyp:dRI} we have $\max_{x \in \Rd} R(x,0)> 0$, we deduce that $\overline T<+\infty$.
We provide some estimates for $\overline{T}$. From \cite{mirrahimi2016class} (Theorem 1.1), and using the estimates \eqref{hyp:conc_D2R} and \eqref{hyp:conc_D2u0}, we have that on $[0,\overline{T}]\times \Rd$,
\begin{equation*}
0< \min(2 \overline{L}_1, \sqrt{\overline{K}_2}) \leq -D^2u(t,x) \leq \max(2 \underline{L}_1, \sqrt{\underline{K}_2})\,.
\end{equation*}
Moreover, by the concavity assumption of $R$, we have that for $t\in (0,\overline{T})$, 
\begin{equation*}
|\nabla_x R(\overline{x}(\overline{T}),0)|^2 \leq |\nabla_x R(\overline{x}(t),0)|^2 \leq |\nabla_x R(\overline{x}(0),0)|^2\,.
\end{equation*}
It leads to 
\begin{equation*}
\frac{1}{ \max(2 \underline{L}_1, \sqrt{\underline{K}_2})} |\nabla_x R(\overline{x}(\overline{T}),0)|^2 \leq 
h'(t)\leq  \frac{1}{\min(2 \overline{L}_1, \sqrt{\overline{K}_2})}  |\nabla_x R(\overline{x}(0),0)|^2\,.
\end{equation*}
 The result then follows from this estimate combined with the equality
\begin{equation*}
\int_0^{\overline{T}} h'(t) \dx t = -h(0)\,.
\end{equation*}

\subsection{Proof of  Corollary \ref{theo:CV_concave_variable}}
In this part, we prove Corollary \ref{theo:CV_concave_variable} that describes the dynamics of the population when a switch in the environment occurs. First, the convergence of $(\ueps)_{\varepsilon}$ and the one of $(\neps)_{\varepsilon}$ for $t\in [0,T_1)$ follows from Theorem \ref{theo:CV_concave}, that also yields point $i)$ and that the $\ueps(T_1,x)$ are uniformly strictly concave and verify Assumption \eqref{hyp:conc_D2u0}. By definition of $\ueps(T_1,x)$ and $\Ieps(T_1)$ and by continuity of $u$ and $I$, we have that $I(T_1) = \lim_{t\rightarrow T_1^-} I(t)$ satisfies Assumption \eqref{hyp:conc_I0}, and that $u(T_1,x)= \lim_{t \rightarrow T_1} u(t,x)$ is well-defined and strictly concave. Therefore,  the convergence of $(\ueps)_{\varepsilon}$ and of $(\neps)_{\varepsilon}$ for $t\in [T_1,T_2)$ follows from Theorem \ref{theo:CV_concave}. Moreover, $\Gamma_{T_1}=\{\overline{x}(T_1)\}$. As a consequence, point $(ii)$ follows from Theorem \ref{theo:critere_asympt}, and $iii)$ follows from Theorem \ref{theo:critere_asympt} combined with Theorem \ref{theo:CV_concave}.

\section{Numerics and application to switching environments}\label{sec:numerics}
In this part, we perform some numerical simulations of \eqref{eq:neps}-\eqref{eq:Ieps} and \eqref{eq:neps_env} to illustrate the selection-mutation dynamics in temporally constant and piecewise constant environments. For that purpose, we use a finite difference scheme with an implicit time discretization scheme, at the exception of the nonlinear term $I_{\varepsilon}(t)$ that is treated explicitly. 

\subsection{Constant environment}
We begin with numerical simulations of the problem \eqref{eq:neps}-\eqref{eq:Ieps} that illustrate Theorem \ref{theo:critere_asympt}. We consider the growth rate given by
\begin{equation*}
R(x,\Ieps) = a(x) - \Ieps\,,
\end{equation*}
with $a$ to define. 

\subsubsection*{Asymptotic persistence and extinction on a time interval}
We choose $a$ as a quadratic function with
\begin{equation}\label{eq:fig1_a}
 a(x) = 0.25 -  x^2 \,,
\end{equation}
so that $a$ is strictly positive in $(-0.5,0.5)$. We choose two expressions for the initial condition, in order to illustrate the two first cases of Theorem \ref{theo:critere_asympt}. They are given by (see Figure \ref{fig12:R_CI})
\begin{equation}\label{fig1_CI}
\neps^0(x) = \frac{\Ieps(0)}{c-b} \mathds{1}_{[b,c]}\,,\quad \Ieps(0)=0.2,
\end{equation}
\begin{equation}\label{fig1_CRbis}
\text{Case 1: }(b,c)=(-0.6,-0.4),\qquad \text{Case 2: }(b,c)=(-0.7,-0.6).
\end{equation}
In the first situation, a part of the initial population is composed of individuals having a positive growth rate in the absence of competition. In this case  we have asymptotic persistence of the population. One can see that $\neps$ evolves towards the best trait $x=0$ (see Figure \ref{sub12pers:sol}). For small times, the population size drops (see Figure \ref{sub12pers:Rho}) as a result of the extinction of the part of the population that is not viable, but it does not reach zero as the population seems to be sustained thanks to larger trait values (see Figure \ref{sub12pers:sol_ext}).
In the second case, the initial population size vanishes near $t=0$ and the population gets extinct asymptotically (see Figures \ref{sub12ext:sol} and \ref{sub12ext:sol_ext}). After some time, the population grows again at some trait values having a nonnegative growth rate in the absence of competition (see Figure \ref{sub12ext:Rho}). This phenomenon is surprising, and shows a limitation of this asymptotic approach. While the population initially gets extinct asymptotically, when $\eps>0$ still a very small population persists and evolves gradually  towards better traits and after some time the population becomes viable and may grow again.
\begin{figure}[!ht]
  \begin{center}
  	\subfloat[$R(x,0)$ and $\neps(0,x)$]{
	\includegraphics[scale=0.6]{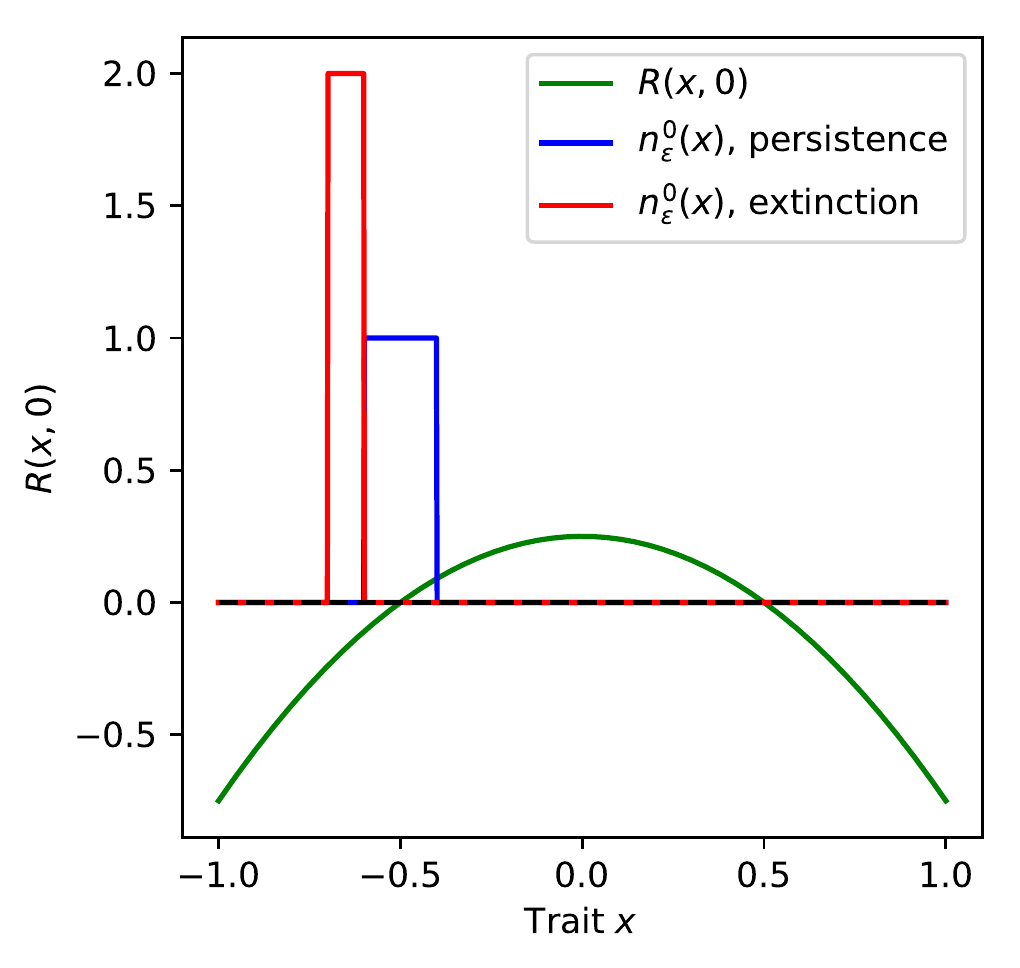} \label{fig12:R_CI}  
	}\\
	\rotatebox{90}{\hspace{1cm} Persistence case}\quad
    \subfloat[$\neps(t,x)$]{
      \includegraphics[width=0.3\textwidth]{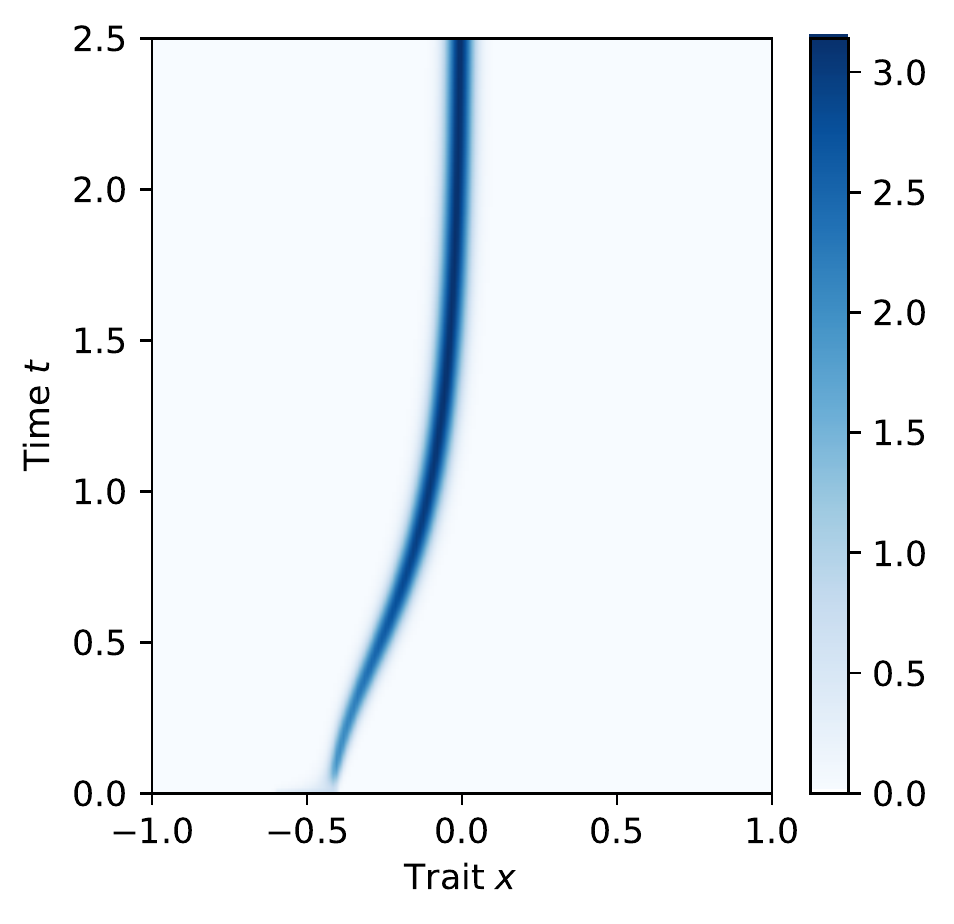}
      \label{sub12pers:sol}
                         }
    \subfloat[$\neps(t,x)$ for $t$ small]{
      \includegraphics[width=0.3\textwidth]{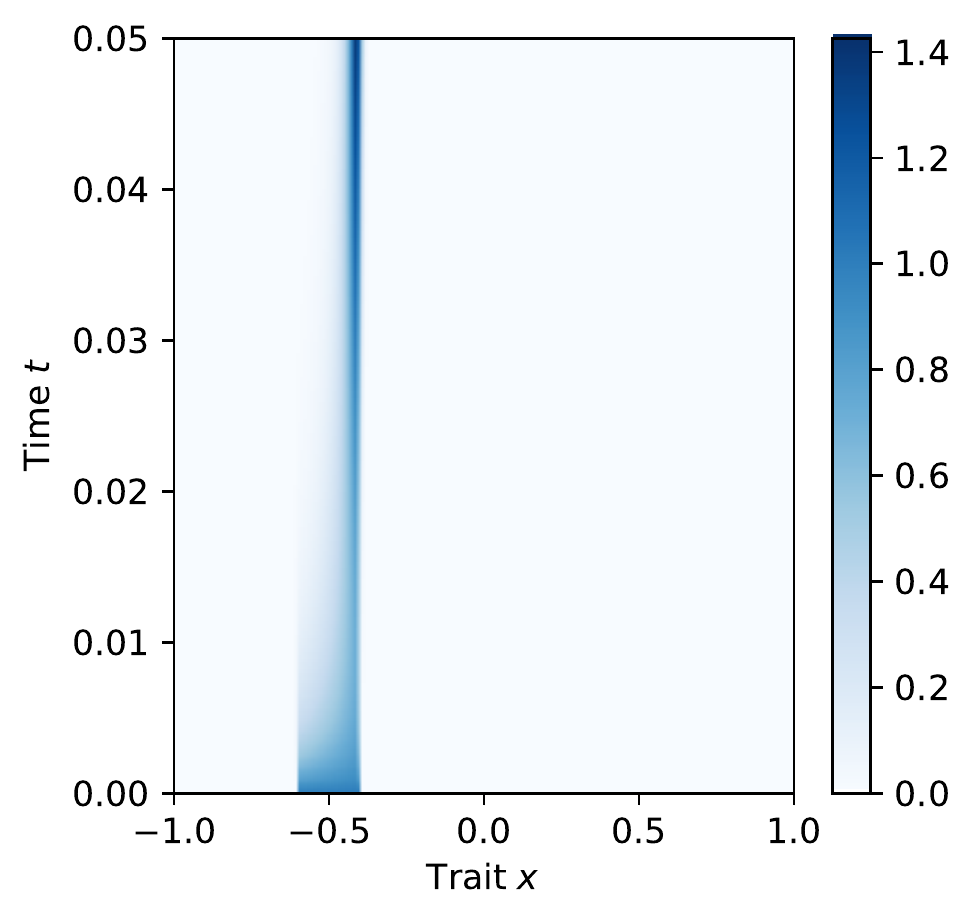}
      \label{sub12pers:sol_ext}
                         }
    \subfloat[$\rhoeps(t)= \intRd \neps(t,x) \dx x$]{
      \includegraphics[width=0.3\textwidth]{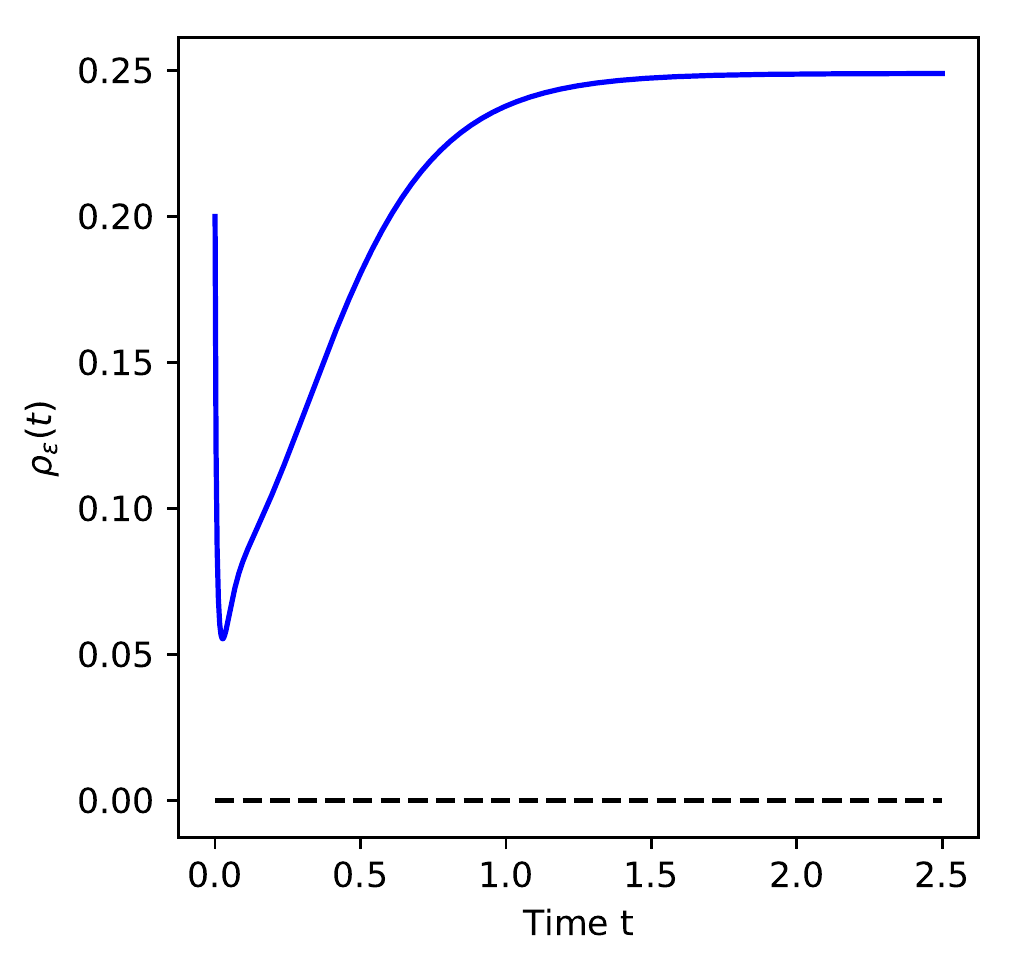}
      \label{sub12pers:Rho}
                         }\\
      \rotatebox{90}{\hspace{1cm} Extinction case}\quad
         \subfloat[$\neps(t,x)$]{
      \includegraphics[width=0.3\textwidth]{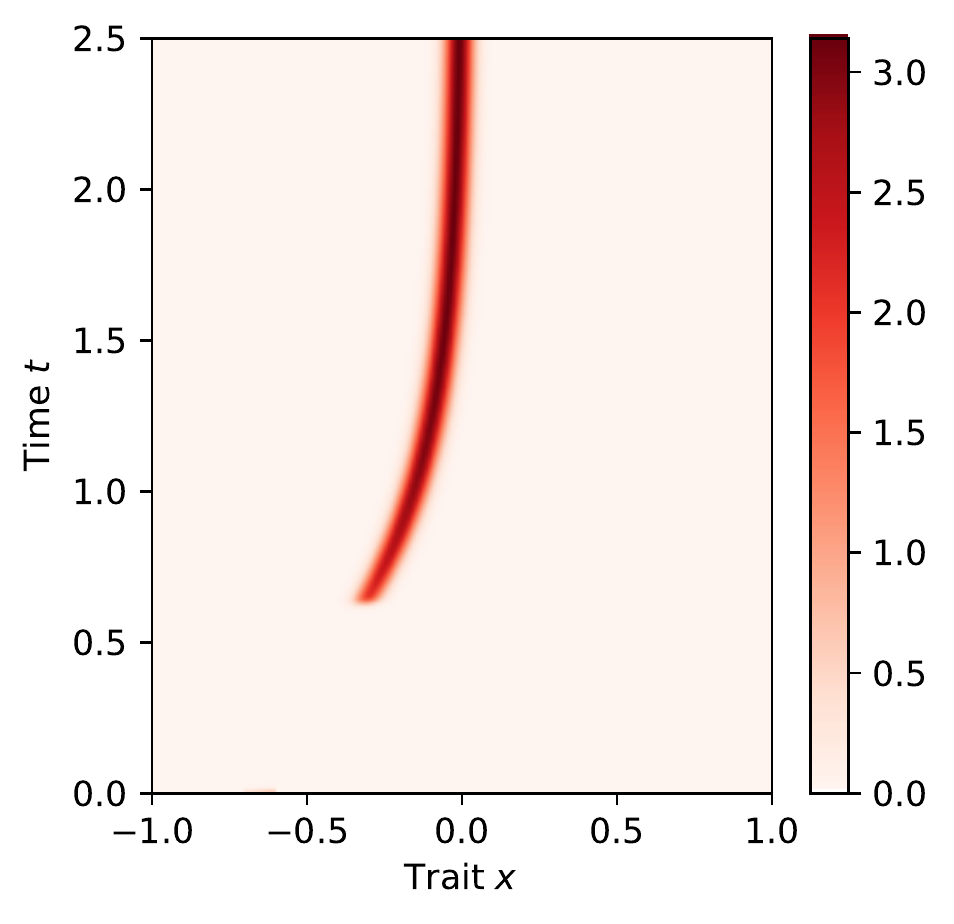}
      \label{sub12ext:sol}
                         }
    \subfloat[$\neps(t,x)$ for $t$ small]{
      \includegraphics[width=0.3\textwidth]{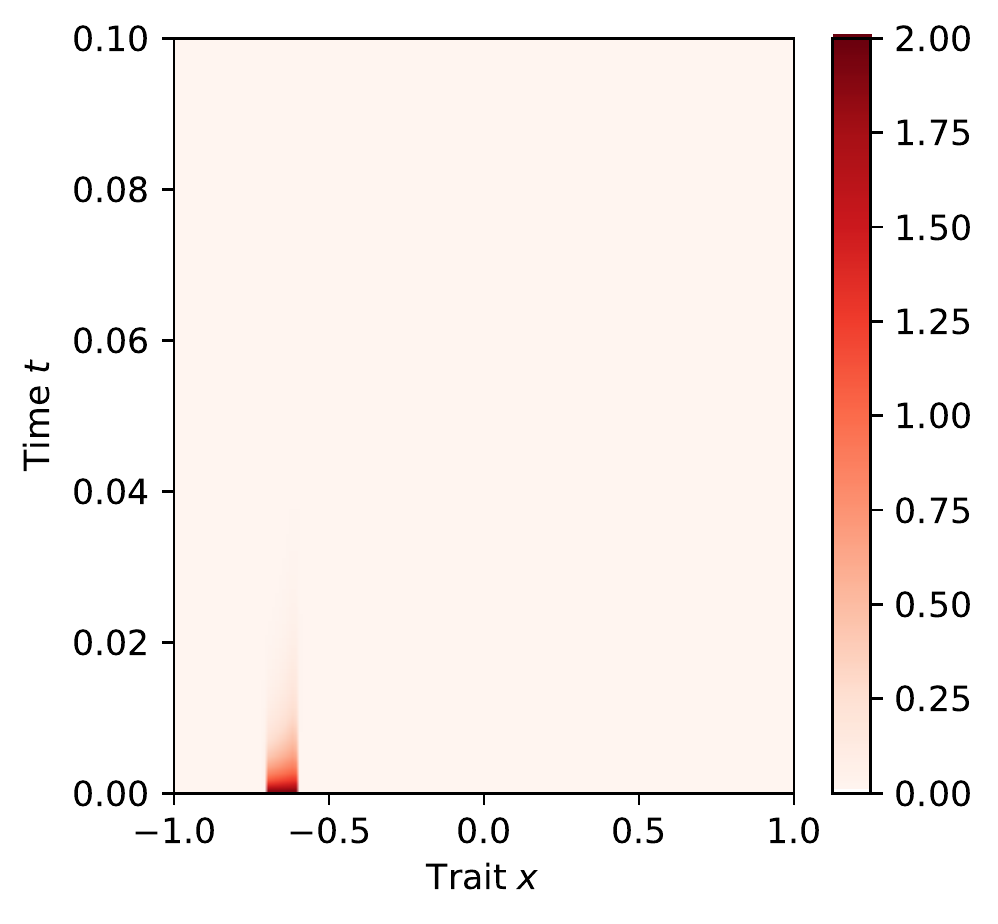}
      \label{sub12ext:sol_ext}
                         }
    \subfloat[$\rhoeps(t)= \intRd \neps(t,x) \dx x$]{
      \includegraphics[width=0.3\textwidth]{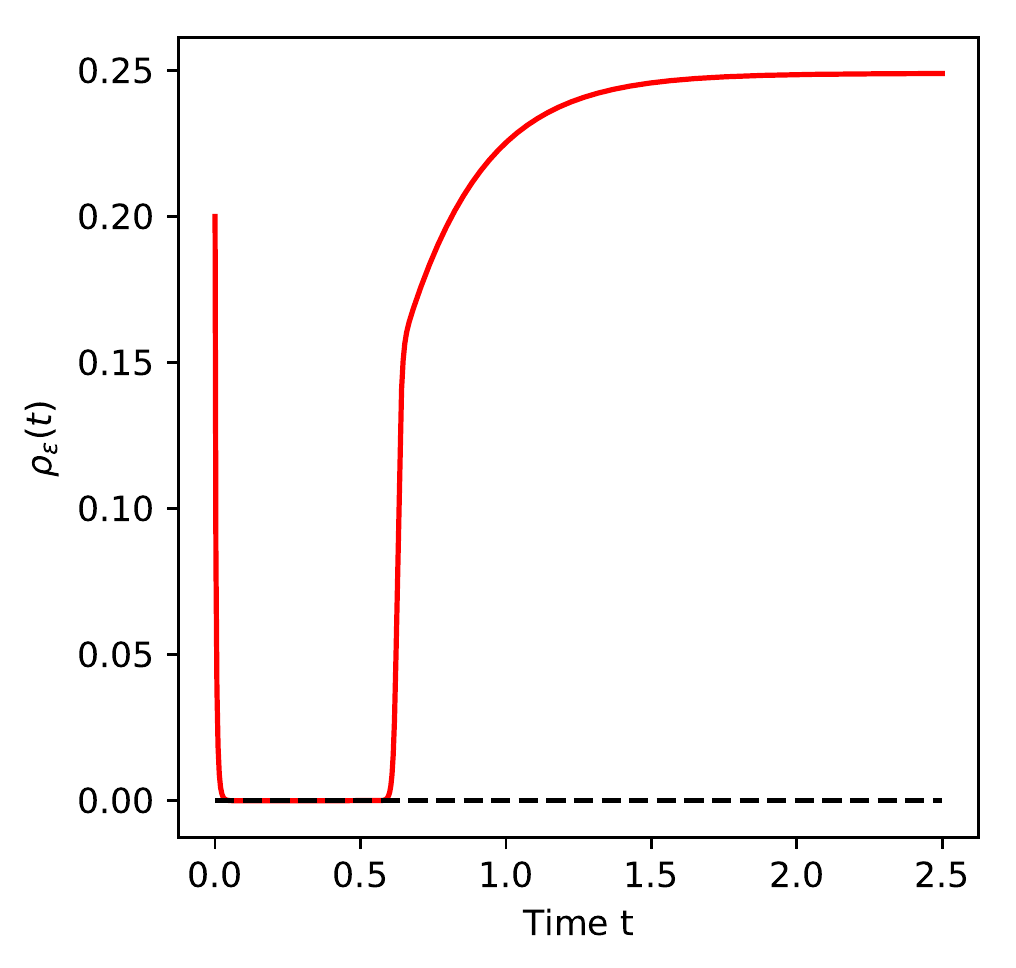}
      \label{sub12ext:Rho}
                         }
    \caption{Numerical simulations of \eqref{eq:neps}-\eqref{eq:Ieps} for $R(x,0)$ given by \eqref{eq:fig1_a} and initial conditions given by \eqref{fig1_CI}--\eqref{fig1_CRbis}. (a): initial conditions and growth rate $R$. (b)-(c)-(d): situation when Supp $n^0 \cap \{x\in \Rd,\, R(x,0)>0 \} \neq \emptyset$. The population density evolves towards better trait values, while the population size stays strictly bounded away from zero. (e)-(f)-(g): situation when $\Gamma_0 \subseteq \{x \in \Rd,\, R(x,0)\leq 0 \}$ for some $C>0$. The population size almost immediately drops to very small values, before growing again when the population density concentrates around trait values that have a positive growth rate. Parameters: $r=0.25$, $g=1$, $\Ieps(0)=0.2$ ; $\dx x = \varepsilon = 10^{-3}$ ; $\dx t = 10^{-4}$.
    }
    \label{fig1:constant_env}
  \end{center}
\end{figure}

\subsubsection*{Asymptotic extinction: critical case}
Figure \ref{fig2:constant_critique} illustrates the third situation described in Theorem \ref{theo:critere_asympt}, occurring when $\Gamma_0 \subseteq \{R(\cdot,0)=0 \}$. For that purpose, we consider
\begin{equation}\label{eq:theo3_R}
R(x,0) = a(x) = -x^2 (x-0.75) (x-2)\,.
\end{equation}
Moreover, we consider two initial conditions given by
\begin{equation}\label{eq:theo3_CI}
\neps^{0,1}(x) = \frac{\Ieps(0)}{\sqrt{2\pi \varepsilon}} e^{- \frac{x^2}{2 \varepsilon}}\,, \qquad
\neps^{0,2}(x) = \frac{\Ieps(0)}{\sqrt{2\pi \varepsilon}} e^{- \frac{(x-0.75)^2}{2 \varepsilon}}\,,  \qquad \Ieps(0)=0.2 \,.
\end{equation}
Figure \ref{fig2:R_CI} shows the corresponding growth function as well as the two initial conditions that we consider. One can see in Figure \ref{fig2:constant_critique} that the solutions behave differently. The solution issued from $\neps^{0,1}$ (Figures \ref{fig2a:sol}-\ref{fig2a:sol_ext}-\ref{fig2a:Rho}) keeps a total population size that is close to zero during a whole time interval, before increasing again very fast from a population density concentrated at trait values having a positive growth rate. Figures \ref{fig2b:sol}-\ref{fig2b:sol_ext}-\ref{fig2b:Rho} illustrate the situation when the initial population is concentrated near values that have a positive growth rate in the absence of competition. In that case, the population size first drops and gets close to zero, but then the mutations allow reaching better traits fast enough to rescue the population. This illustrates the critical situation where asymptotic extinction occurs punctually in time. Note however that here the population size does not reach zero since $\eps>0$.

\begin{figure}[!ht]
  \begin{center}
   \subfloat[$R(x,0)$ and $\neps^0(x)$]{
      \includegraphics[scale=0.6]{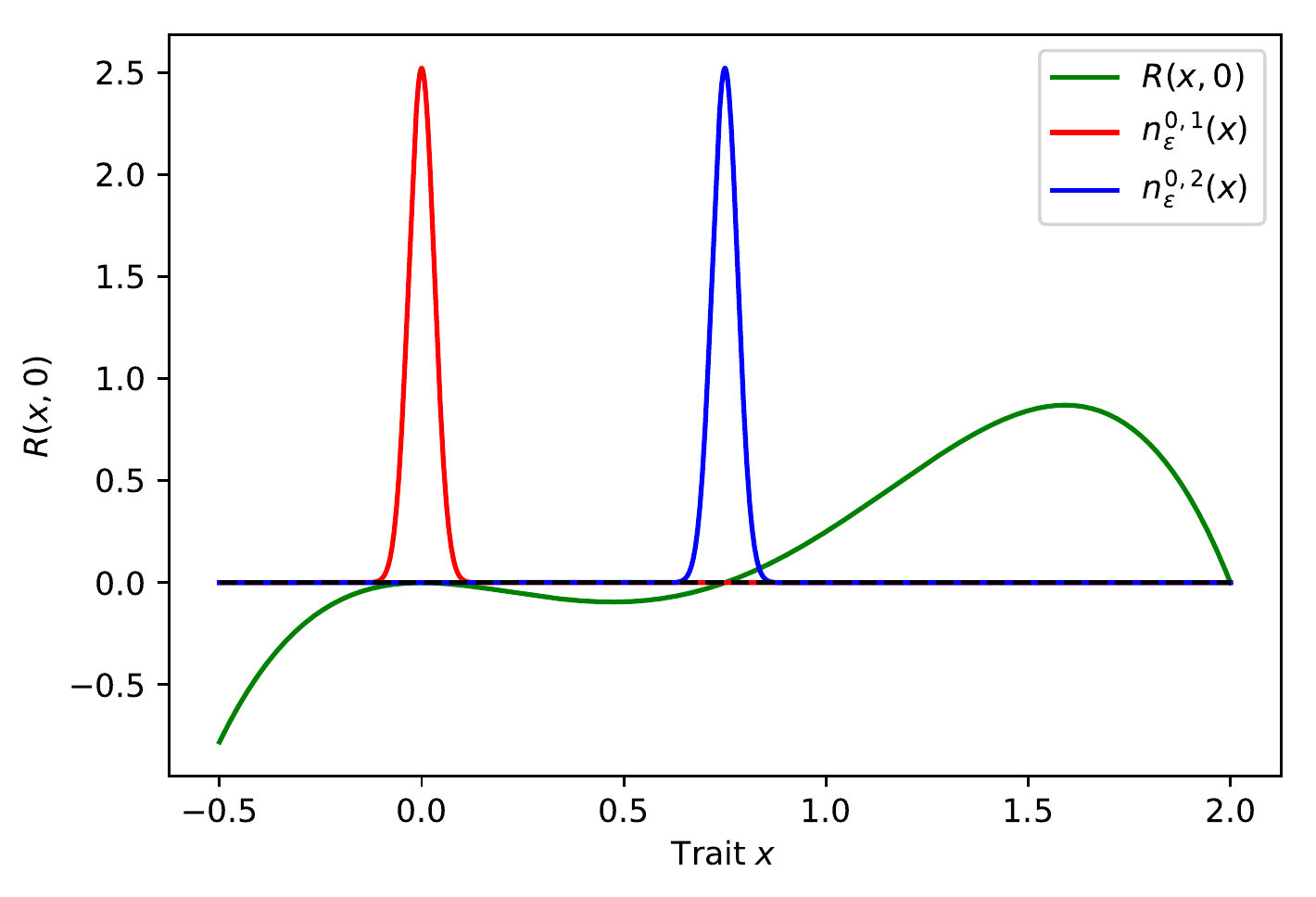} 
      \label{fig2:R_CI}
                         }\\
    \subfloat[$\neps(t,x)$]{
      \includegraphics[width=0.3\textwidth]{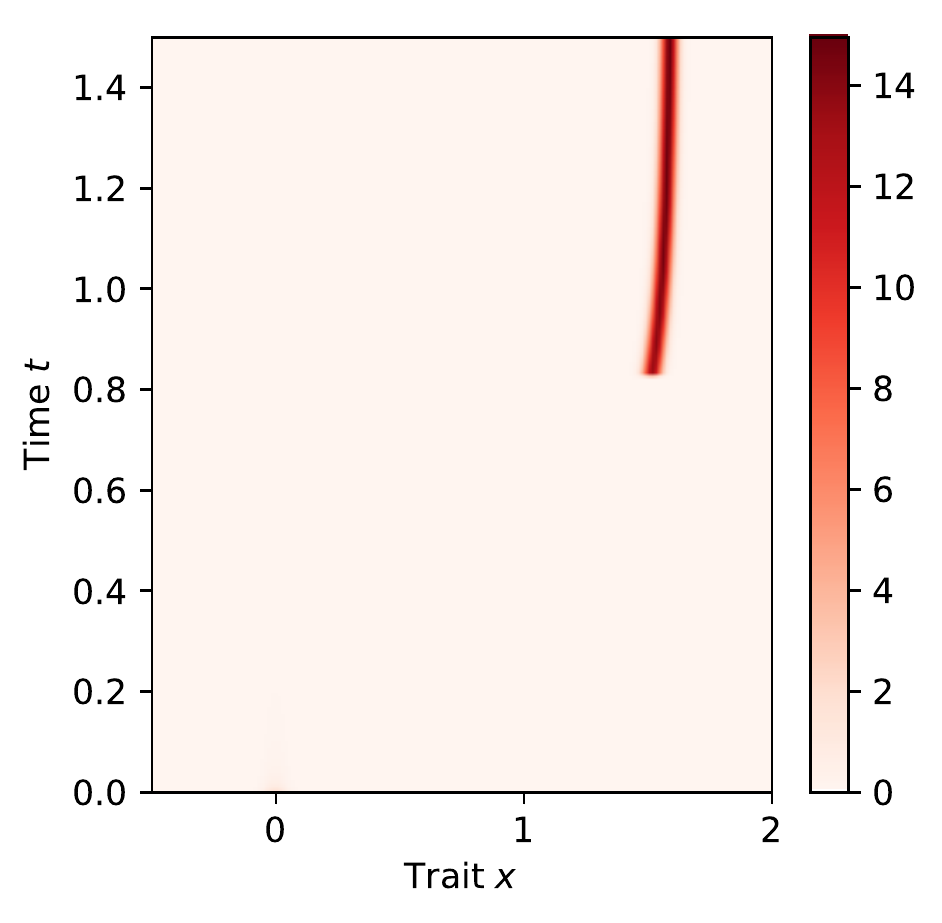}
      \label{fig2a:sol}
                         }
    \subfloat[$\neps(t,x)$ for $t$ small]{
      \includegraphics[width=0.31\textwidth]{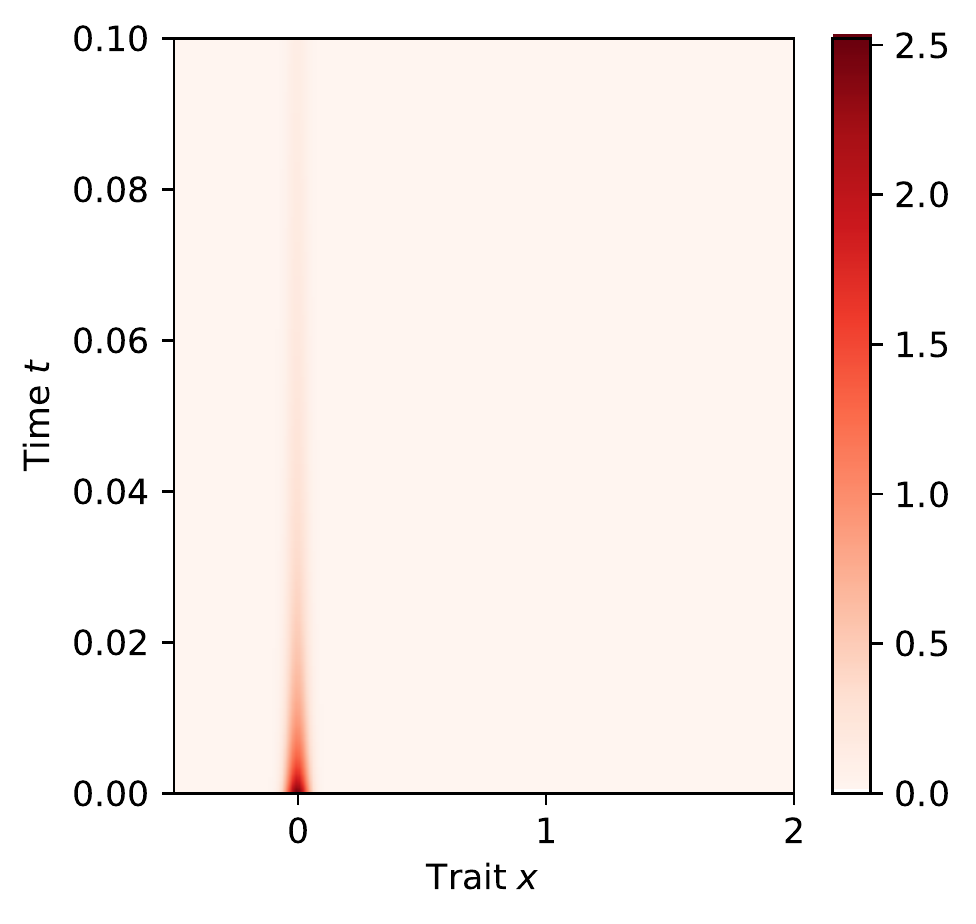}
      \label{fig2a:sol_ext}
                         }
    \subfloat[$\rhoeps(t)= \intRd \neps(t,x) \dx x$]{
      \includegraphics[width=0.3\textwidth]{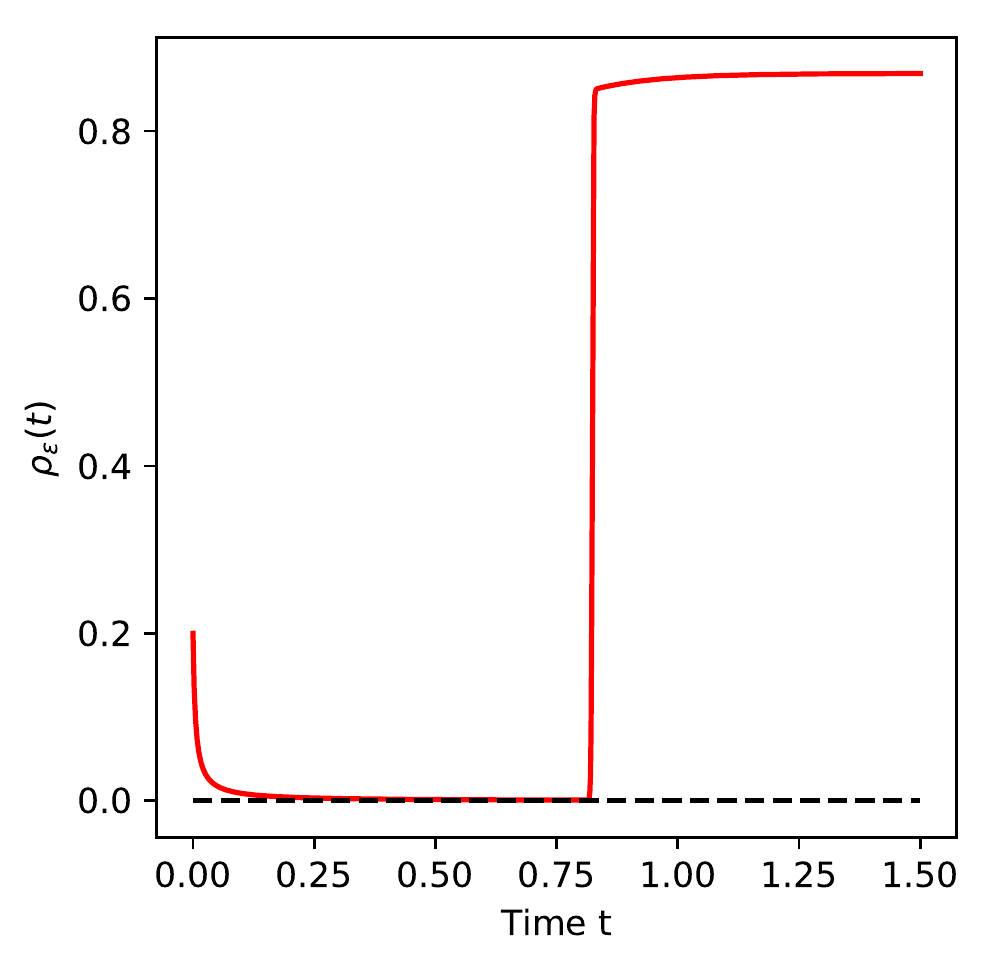}
      \label{fig2a:Rho}
                         }\\
     \subfloat[$\neps(t,x)$]{
      \includegraphics[width=0.3\textwidth]{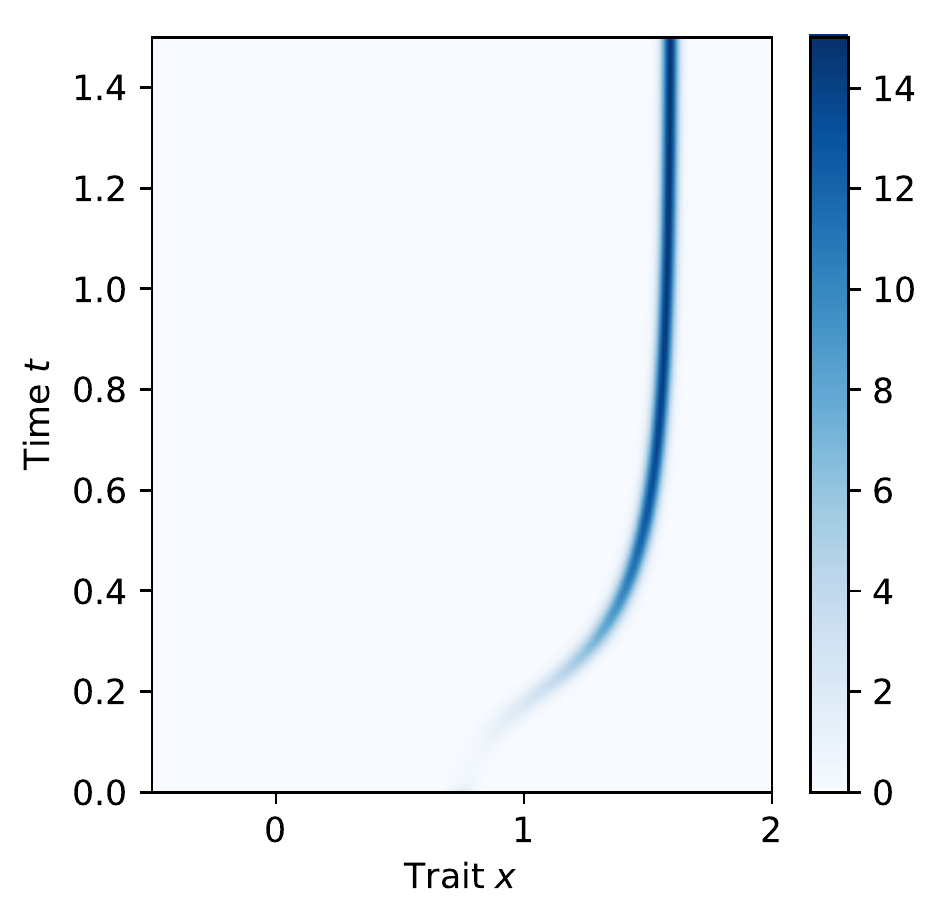}
      \label{fig2b:sol}
                         }
    \subfloat[$\neps(t,x)$ for $t$ small]{
      \includegraphics[width=0.31\textwidth]{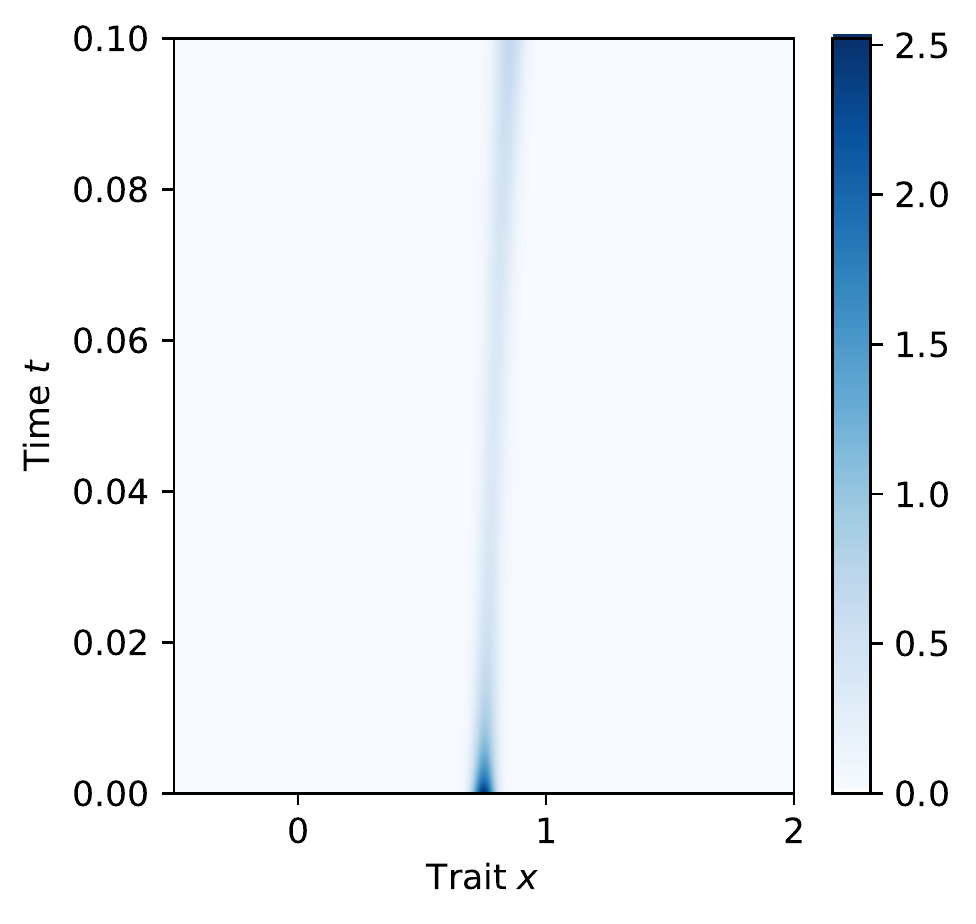}
      \label{fig2b:sol_ext}
                         }
    \subfloat[$\rhoeps(t)= \intRd \neps(t,x) \dx x$]{
      \includegraphics[width=0.3\textwidth]{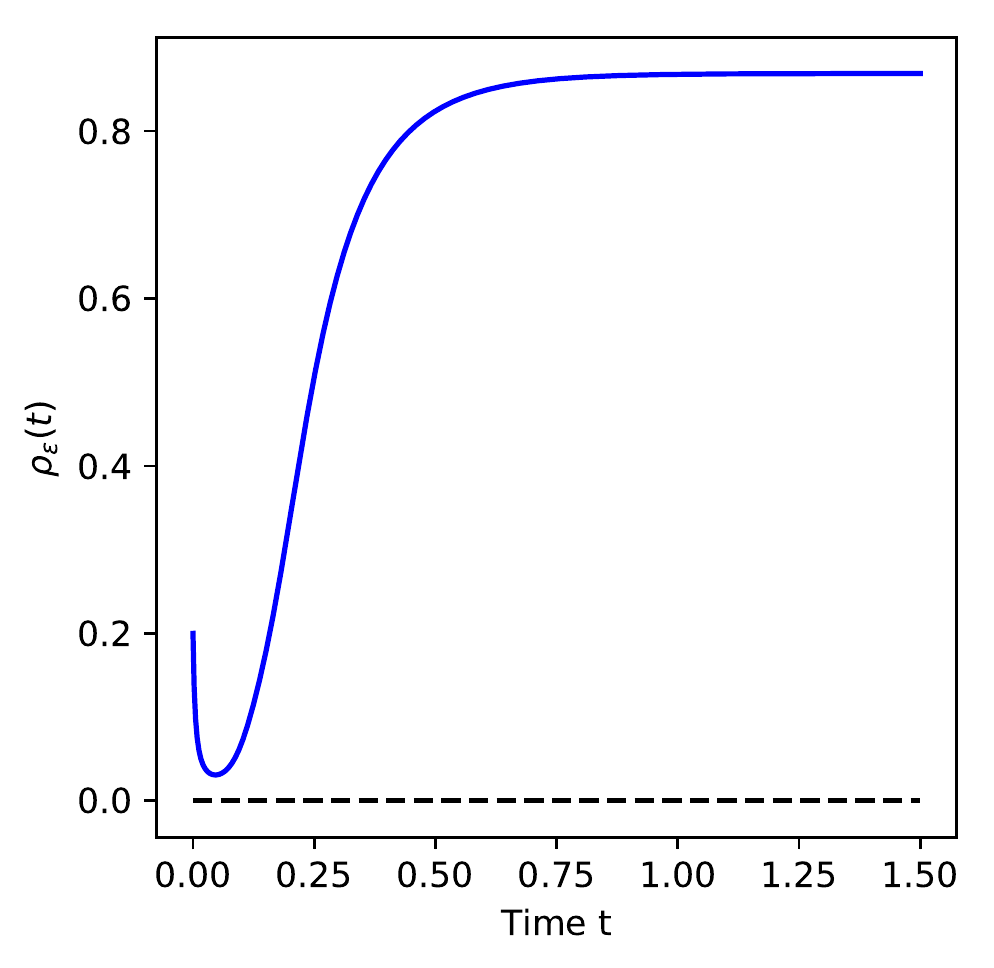}
      \label{fig2b:Rho}
                         }   
\caption{Numerical simulations of \eqref{eq:neps}-\eqref{eq:Ieps} in the case where $\Gamma_0 \subseteq \{R(\cdot,0)=0\}$, for $R(x,0)$ given by \eqref{eq:theo3_R} and initial conditions given by \eqref{eq:theo3_CI}. (a): initial conditions and growth rate $R$. (b)-(c)-(d): situation when the initial population concentrates far from viable traits. The population size drops to very small values on some time interval. (e)-(f)-(g): situation when the initial population concentrates near viable traits. The population size initially drops but rapidly increases away from zero. Parameters: $\Ieps(0)=0.2$ ; $\dx x = \varepsilon = 10^{-3}$ ; $\dx t = 10^{-4}$.}
    \label{fig2:constant_critique}
  \end{center}
\end{figure}

\subsubsection*{Case not treated by Theorem \ref{theo:critere_asympt}}
Figure \ref{fig4} illustrates the situation described in Remark \ref{rema:cas_non_decrit} where Supp $n^0 \subseteq \{R(\cdot,0)\leq 0 \}$ and $\Gamma_0 \cap \{ R(\cdot,0)>0\} \neq \emptyset$. For that purpose, we consider again \eqref{eq:fig1_a}
and 
\begin{equation}\label{eq:fig_Remark_CI}
\neps^0(x) = \frac{0.2}{\sqrt{2 \pi \varepsilon}} \exp\left(- \frac{(x+0.75)^2}{2 \varepsilon}  \right) 
+ \frac{\varepsilon}{\sqrt{2 \pi \varepsilon}} \exp\left(- \frac{x^2}{2 \varepsilon}  \right) \,,
\end{equation}
so that $\Gamma_0=\{-0.75,0\}$ and Supp $n^0 = \{ -0.75\}$, with $R(0,0)>0$ and $R(-0.75,0)\leq 0$. Therefore, $\Gamma_0$ meets a viable trait ($x=0$), but the population size at this point vanishes as $\varepsilon$ goes to zero (see Figure \ref{sub4:CI_R}). In this case, the individuals holding traits with negative growth rates die quickly which leads to a rapid drop of the population size to very small values. Afterwards, the population stabilizes around the best trait and grows again. 

\begin{figure}[!ht]
  \begin{center}
    \subfloat[$\neps^0(x)$ and $R(x,0)$]{
\includegraphics[width=0.3\textwidth]{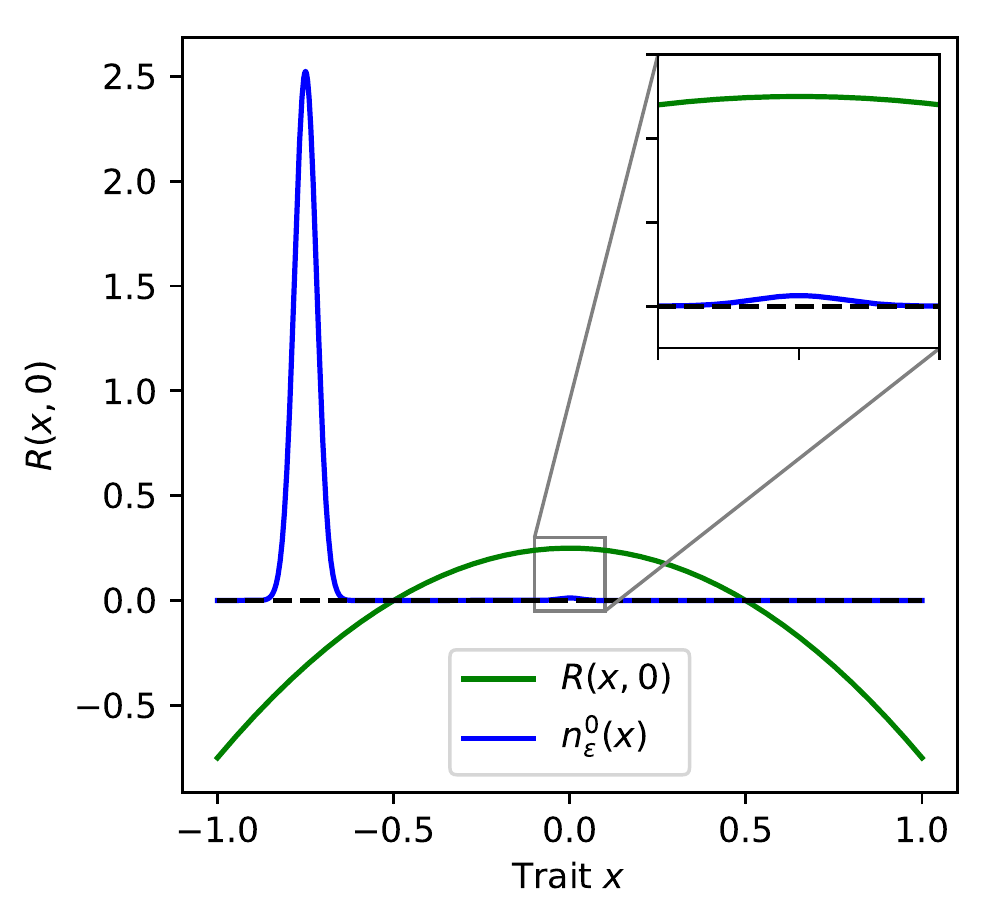} 
      \label{sub4:CI_R}
                         }
    \subfloat[$\neps(t,x)$]{
      \includegraphics[width=0.3\textwidth]{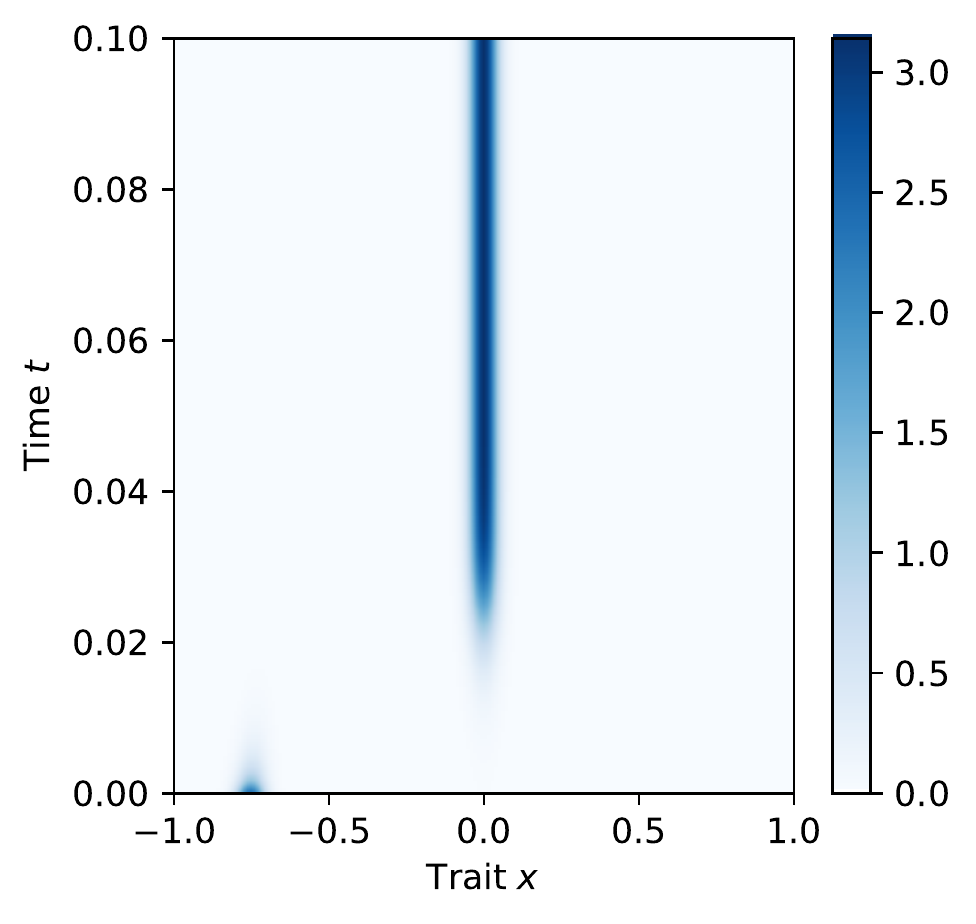}
      \label{sub4:sol}
                         }
    \subfloat[$\rhoeps(t)= \intRd \neps(t,x) \dx x$]{
      \includegraphics[width=0.3\textwidth]{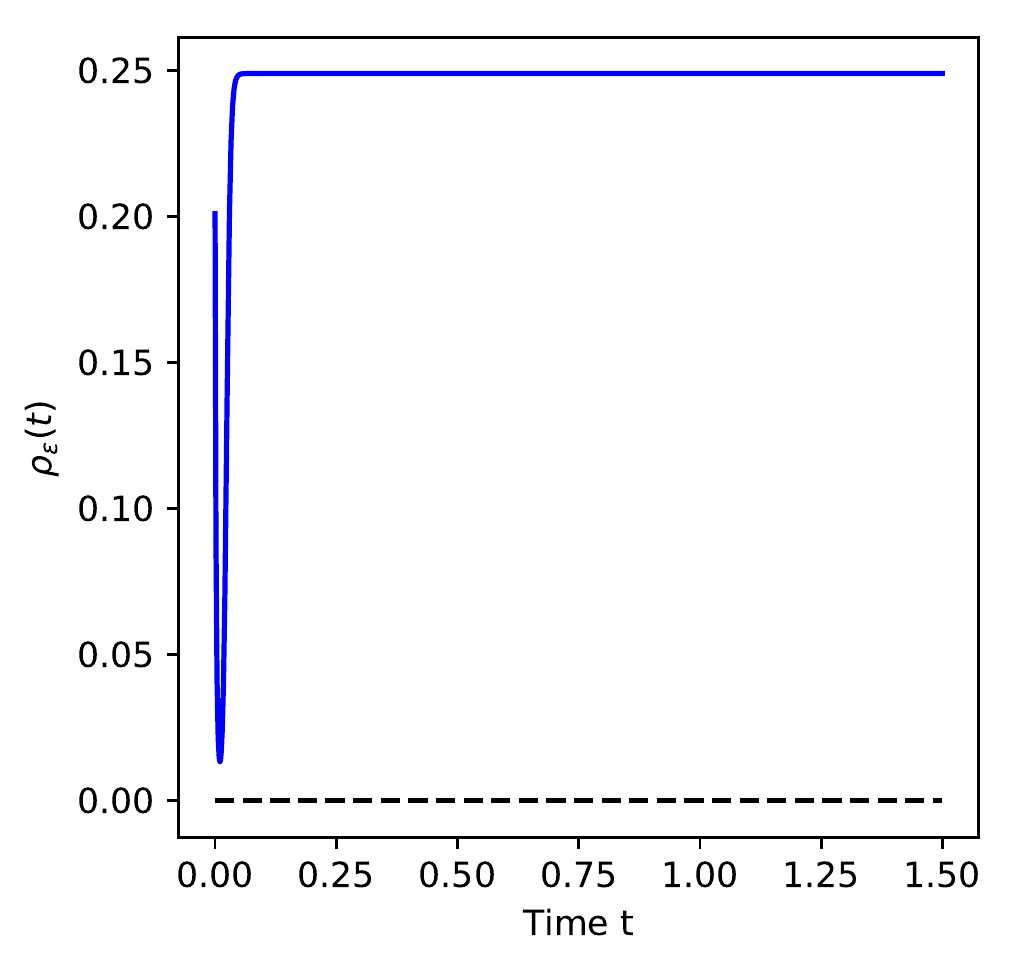}
      \label{sub4:Rho}
                         }
    \caption{Numerical simulations of \eqref{eq:neps}-\eqref{eq:Ieps} in the case of Remark \ref{rema:cas_non_decrit},    
    where Supp $n^0 \subseteq \{R(\cdot,0)\leq 0 \}$ and $\Gamma_0 \cap \{ R(\cdot,0)>0\} \neq \emptyset$, with $a(x)$ given by \eqref{eq:fig1_a} and an initial condition given by \eqref{eq:fig_Remark_CI}. The population density is split between a subpopulation concentrated at non viable trait values, and a subpopulation concentrated at viable trait values, but having a mass going to zero as $\varepsilon \rightarrow 0$. The population size drops immediately as the population carrying traits with negative growth rates disappears, but it quickly increases again towards a positive stationary value, as the population concentrates around the best possible trait. Parameters: $\dx x = \varepsilon = 10^{-3}$ ; $\dx t = 10^{-4}$.}
    \label{fig4}
  \end{center}
\end{figure}

\subsection{Piecewise constant environment}
In this section, we illustrate different phenomena arising in a temporally piecewise constant environment. More precisely, in the case where the environment switches between two states, we consider Problem \eqref{eq:neps_env} for $e$ a periodic function of time with period $T$ such that for $t\in [0,T]$, $e(t)=\mathds{1}_{[0,T/2)}(t) + 2 \mathds{1}_{[T/2,T)}(t)$. Then, for $(t,x) \in [0,T]\times \Rd$, we define the corresponding growth rate $R(x,e(t),\Ieps(t))$ by $R_{e(t)}(x,\Ieps(t))$.
We investigate numerically situations where the period of fluctuations has an effect on the dynamics, whether it acts on the fate of the population, the mean population size, or on the dynamics of the optimal trait.

\subsubsection*{Effect of the period of fluctuations on the fate of the population}
We consider here a periodic switch between two concave growth rates given by
\begin{equation}\label{eq:switch_ext_R}
\begin{aligned}
R_{1}(x, \Ieps(t)) &=  r-g(x+\theta)^2 - \Ieps(t)\quad \text{ and }\quad 
R_2(x, \Ieps(t)) &=  r-g(x-\theta)^2 - \Ieps(t)\,,
\end{aligned}
\end{equation}
with $g \theta^2 < r < 4g \theta^2$ to ensure that there are traits viable in both environments. Here, we take 
\begin{equation}
\label{datafig4}
\theta=0.5,\quad g=1 \quad \text{and}\quad  r=0.5.
\end{equation}
 The initial condition is given by (see Figure \ref{sub5ext:CI_R})
\begin{equation}\label{eq:switch_ext_CI}
\neps^0(x) = \rho_0 \frac{g^{\frac{1}{4}}}{\sqrt{2 \pi \varepsilon}} e^{-\frac{\sqrt{g} X^2}{2 \varepsilon}}\,,\quad \rho_0=0.25\,,
\end{equation}
and allows initial persistence in both environments. Figures \ref{sub5ext:sol} and \ref{sub5ext:Rho} show the evolution of $\neps$ and $\rhoeps$ over the first environmental switch when the period of fluctuation is large ($T=1$). On $[0,T/2)$, the population concentrates on better traits relatively to the first environment. However, these traits have negative growth rates in the second environment. As a result, at switching time, the population is in a situation of asymptotic extinction.

Then, we consider the case where the period of fluctuations is smaller $(T=0.2$), for the same initial condition. In this situation, the period is small enough so that the population remains concentrated in traits which are viable in both environments. The population is therefore persistent, and $\neps$ is periodic in time (see Figures \ref{sub5pers:sol} and \ref{sub5pers:Rho}). This illustrates how the fate of a population may be drastically affected by the timing of environmental fluctuations.

\begin{figure}[!ht]
  \begin{center}
    \subfloat[$R_{1}(x,0), R_2(x,0)$ and $\neps^0(x)$ from \eqref{eq:switch_ext_CI}]{
\includegraphics[width=0.3\textwidth]{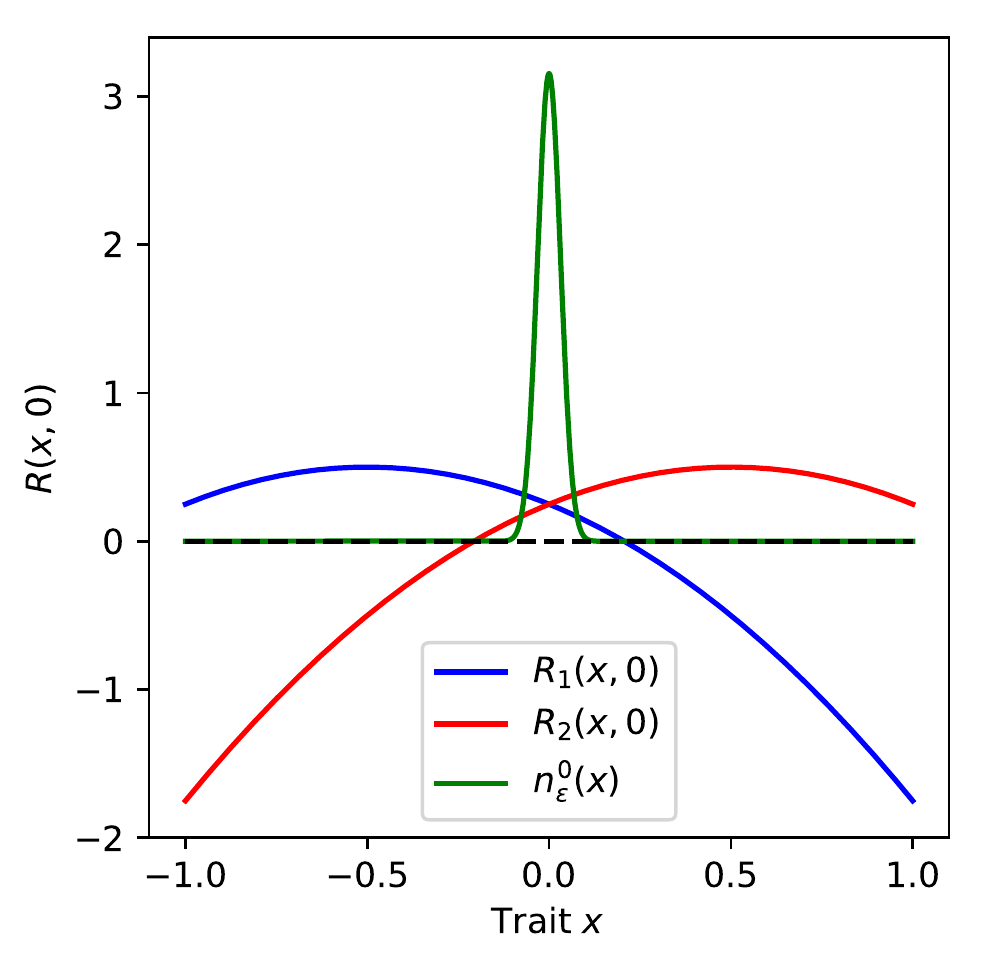} 
      \label{sub5ext:CI_R}
                         }
    \rotatebox{90}{\hspace{2cm} $T=1$}
    \subfloat[$\neps(t,x)$]{
      \includegraphics[width=0.3\textwidth]{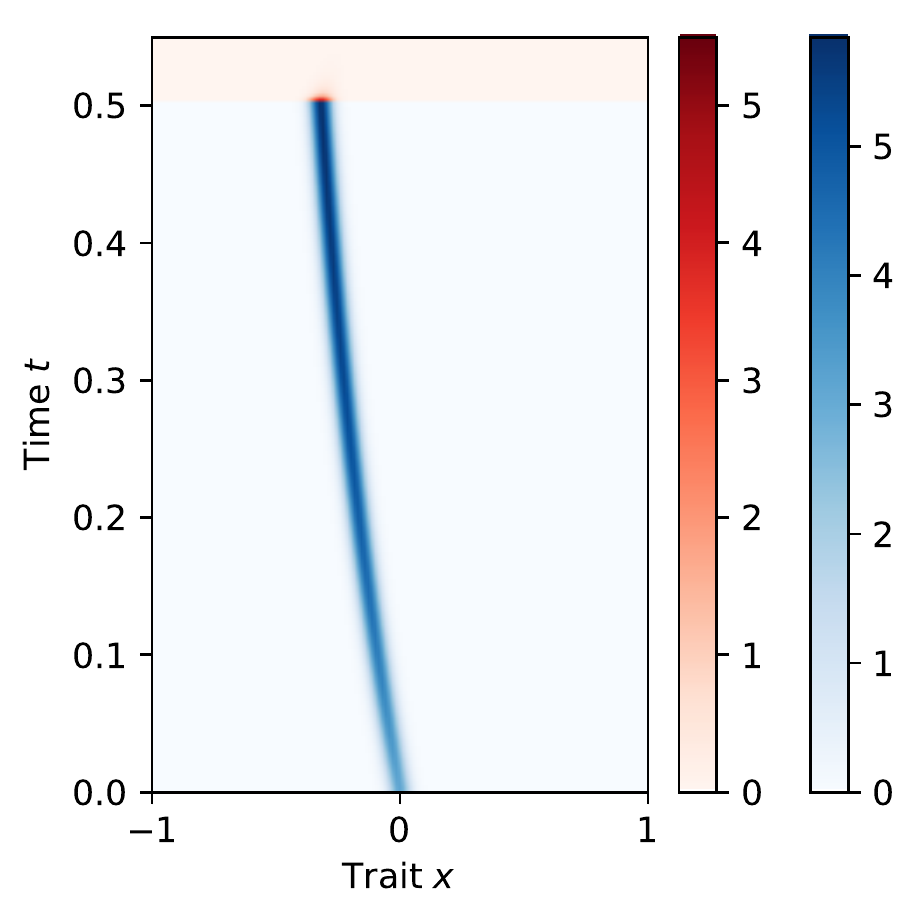}
      \label{sub5ext:sol}
                         }
    \subfloat[$\rhoeps(t)= \intRd \neps(t,x) \dx x$]{
      \includegraphics[width=0.3\textwidth]{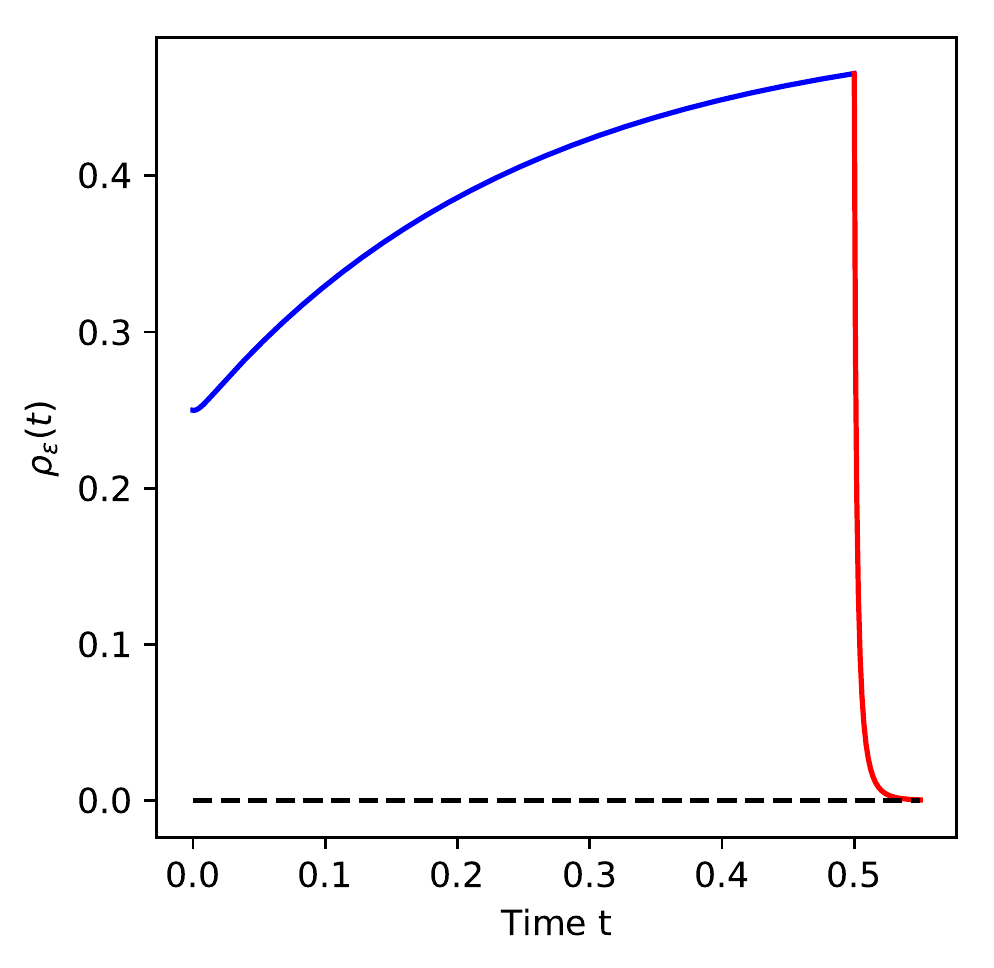}
      \label{sub5ext:Rho}
                         }\\
                         \rotatebox{90}{\hspace{1cm} $T=0.2$}
    \subfloat[$\neps(t,x)$]{
      \includegraphics[width=0.3\textwidth]{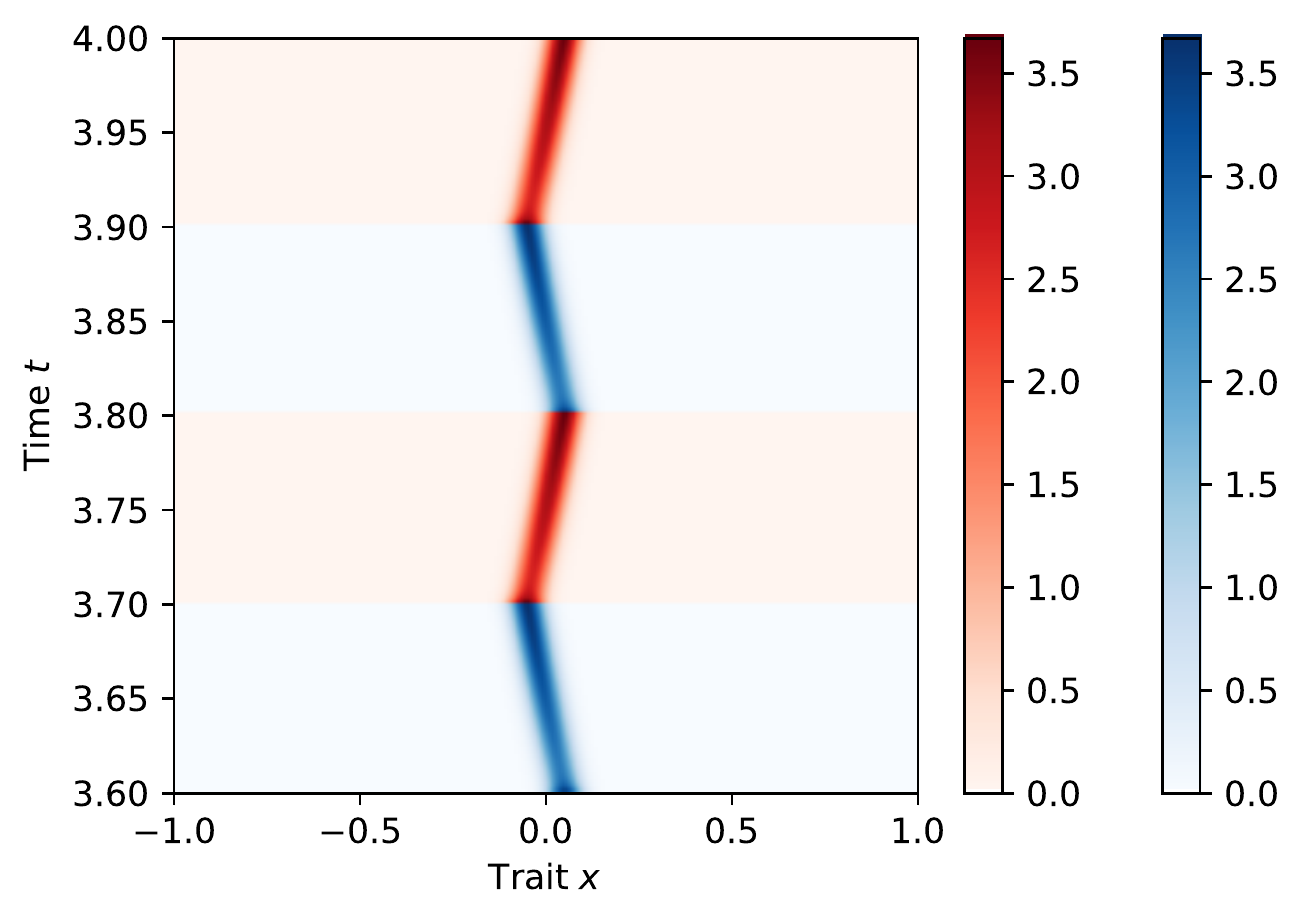}
      \label{sub5pers:sol}
                         }
    \subfloat[$\rhoeps(t)= \intRd \neps(t,x) \dx x$]{
      \includegraphics[width=0.3\textwidth]{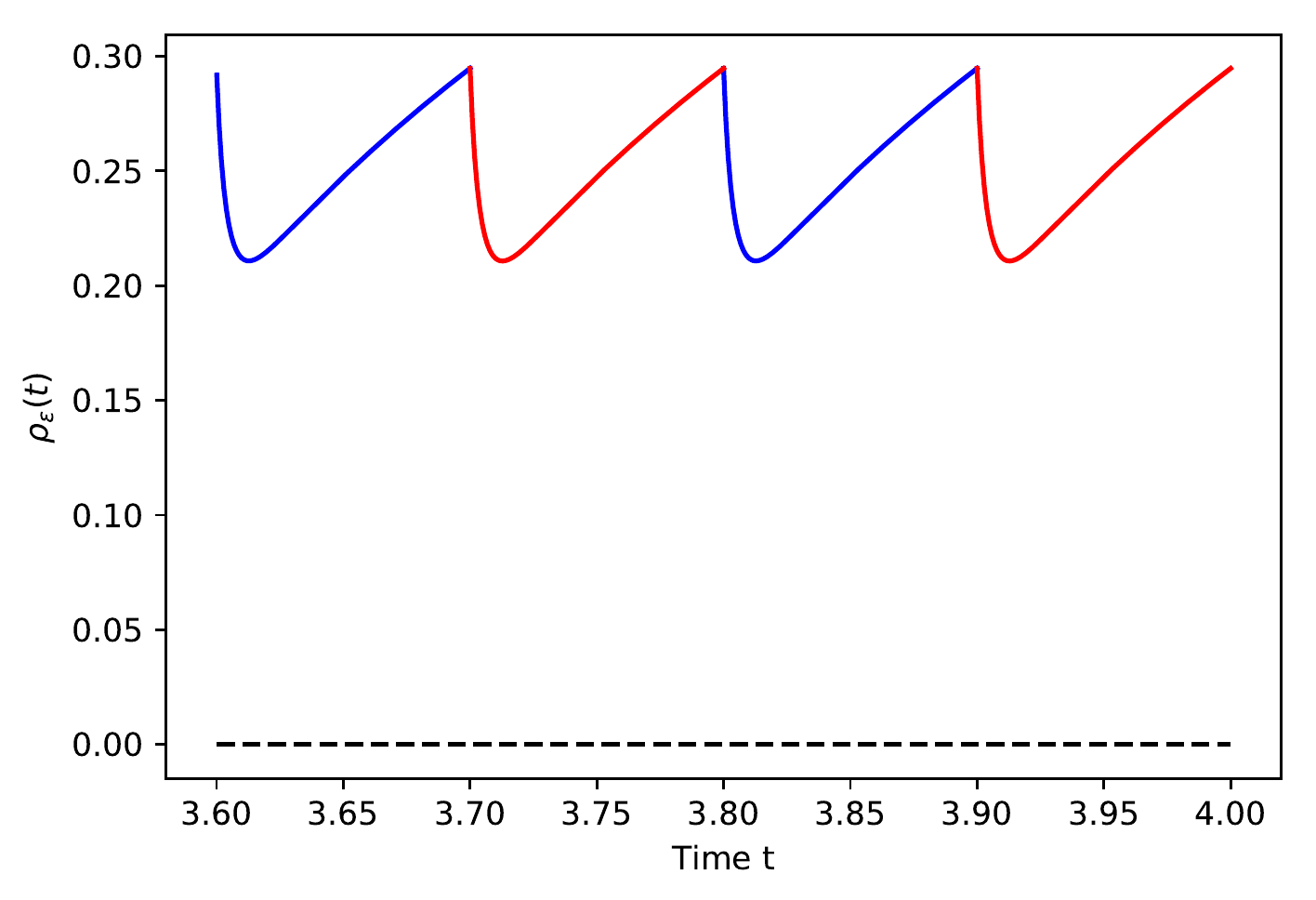}
      \label{sub5pers:Rho}
                         }
    \caption{Numerical simulations of \eqref{eq:neps_env} for a growth rate given by \eqref{eq:switch_ext_R}--\eqref{datafig4} and an initial condition given by \eqref{eq:switch_ext_CI}. (a): initial condition and the growth rates $R_i$. (b)-(c): situation when $T=1$. There is asymptotic extinction at the first switching time, since the population is concentrated at trait values with positive growth rates in the first environment and to negative growth rates in the second one. (d)-(e): periodic solution when $T=0.2$. The solution remains concentrated at trait values that correspond to positive growth rates in both environments, while the population size reaches a periodic time evolution. Parameters: $\dx x = \varepsilon = 10^{-3}$ ; $\dx t = 10^{-4}$.}
    \label{fig5ext}
  \end{center}
\end{figure}

\subsubsection*{Effect of the period of fluctuations on the mean population size}
We are now interested in highlighting how the period of fluctuations between two environments can have an effect on the mean population size of the population. To do so, we consider the growth rates given by \eqref{eq:switch_ext_R} with $\theta=1$, $r=1$ and $g=0.2$. This choice aims at considering environments where each respective optimal trait is also viable in the other environment and the asymptotic extinction of the population can not occur. The initial condition is given by \eqref{eq:switch_ext_CI} 
(see Figure \ref{sub6:CI_R}). A natural indicator for the mean population size during a period is given by 
\begin{equation*}
\overline{\rhoeps}(T) := \frac{1}{T} \int_0^T \rhoeps(t) \dx t\,.
\end{equation*}
We perfom numerical simulations of the solution of \eqref{eq:neps_env} for $T$ ranging in $[0.1,5]$. Figure \ref{sub6:rhoeps} shows several simulations of the time evolution of the size of the population for increasing values of the period $T$, when a stationary regime is attained. It is observed that the population size drop gets larger and larger for an increasing time spent in each environment. 
However, Figure \ref{sub6:sol} shows that the stationary mean population size over a period of fluctuations increases with the length of the period. 
This can be interpreted as follows. For a small period of fluctuations, the population never gets fully adapted to an environment, but does not suffer much either from the fluctuations. If the period of fluctuations is larger, this allows the population to concentrate on the optimal trait in each environment. As a consequence, the population size can drop significantly at switching time, which seems more costly. Understanding which situation is better from an evolutive point of vue is not intuitive and depends strongly on the shape of the growth rates. Our simulations show an example where larger periods of fluctuations are better from an evolutive point of vue, even if the population is less stable in the ecosystem.

\begin{figure}[!ht]
  \begin{center}
    \subfloat[$\neps^0(x)$, $R_{1}(x,0)$ and $R_{2}(x,0)$]{
\includegraphics[width=0.45\textwidth]{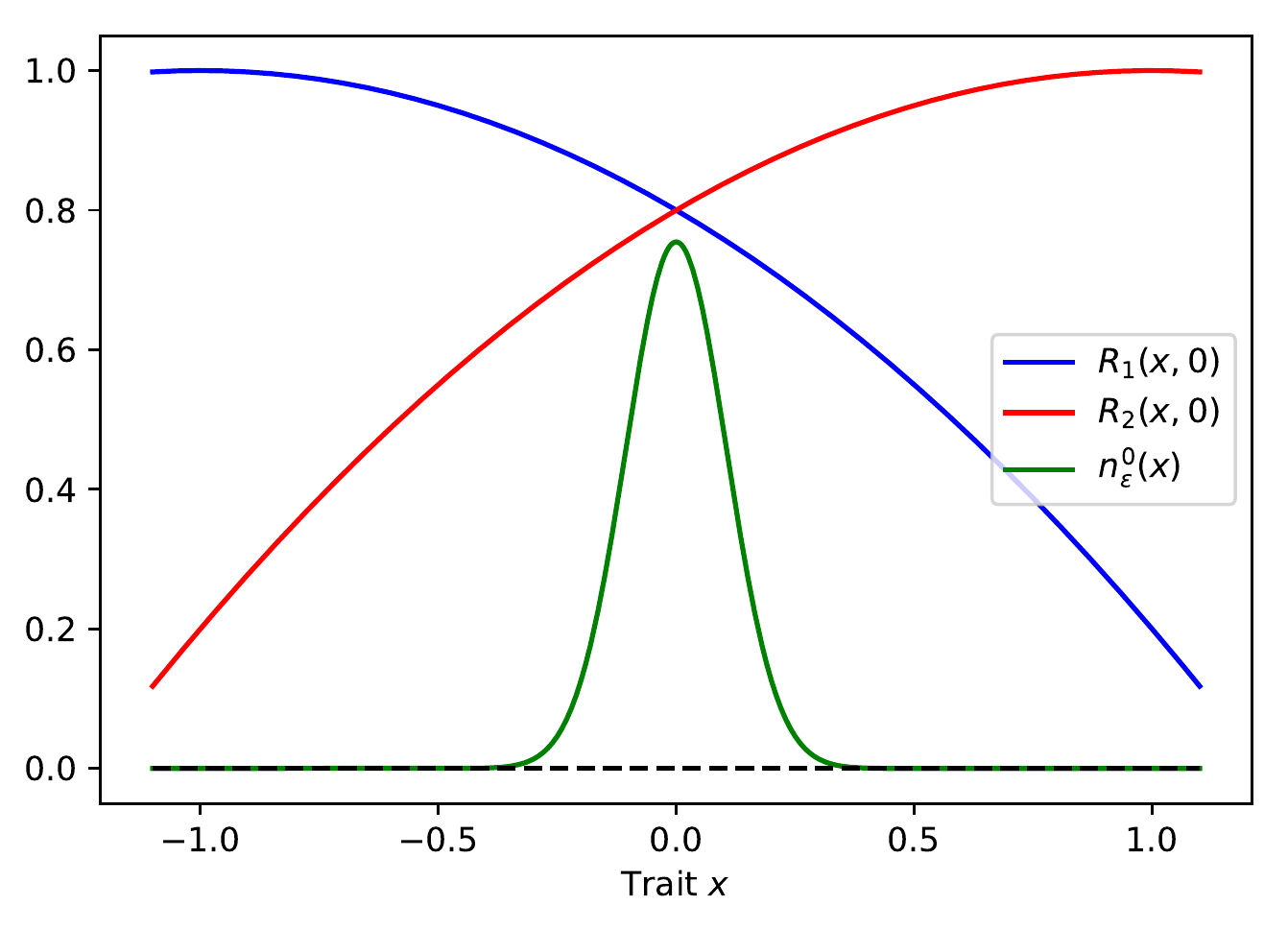} 
      \label{sub6:CI_R}
                         }\\
    \subfloat[$\rhoeps(t)$ for increasing $T$ from left to right]{
      \includegraphics[width=0.48\textwidth]{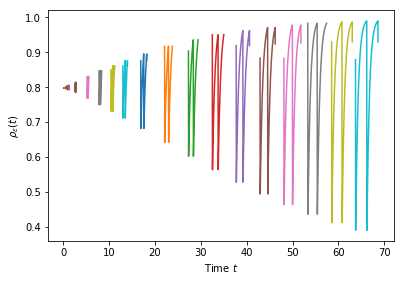}
      \label{sub6:rhoeps}
                         }
    \subfloat[$\overline{\rhoeps}(T)$]{
      \includegraphics[width=0.48\textwidth]{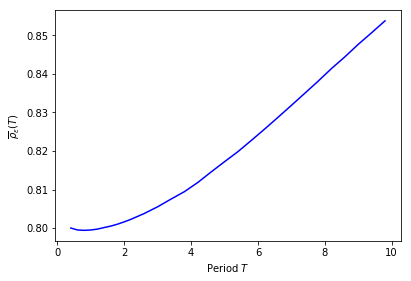}
      \label{sub6:sol}
                         }
    \caption{Numerical simulations of \eqref{eq:neps_env} for a growth rate given by \eqref{eq:switch_ext_R} with $\theta=1$, $r=1$ and $g=0.2$ and an initial condition given by \eqref{eq:switch_ext_CI}. (a): growth rates for each environment and initial condition. (b): evolution of $\rhoeps(t)$ over a period for different period values $T$. Drops in population size become larger when the period of fluctuations increases.   
    (c): stationary mean population size over a period of environmental fluctuations as a function of the duration $T$ of the period. The mean population size increases as the period of fluctuations gets larger. Parameters: $\dx x = \varepsilon = 10^{-3}$ ; $\dx t = 10^{-4}$.}
    \label{fig6pers}
  \end{center}
\end{figure}

\subsubsection*{Effect of the period of fluctuations on the concentration trait}
Finally, we illustrate in Figure \ref{fig7} the situation where the period of fluctuations affects the trait value at which the population concentrates. For that purpose, we consider an environmental switch between a concave and a symmetric bimodal shape. More precisely, consider 
\begin{equation}\label{eq:fig7_R}
\begin{aligned}
R_{1}(x, \Ieps(t)) =  0.7-\frac{1}{5}
x^2 - \Ieps(t)\quad \text{ and } \quad
R_2(x, \Ieps(t)) =  0.2-\frac{2}{3} x^4 + \frac{4}{5}x^2 - \Ieps(t)\,,
\end{aligned}
\end{equation}
with an initial condition given by
\begin{equation}\label{eq:fig7_CI}
\neps^0(x) =  \frac{ \rho_0 }{\sqrt{2 \pi \varepsilon}} e^{-\frac{(x-1)^2}{2\varepsilon}}\,,\quad \rho_0=0.25\,.
\end{equation}
 Both growth rates and the initial condition are shown in Figure \ref{sub7:CI_R}. Figures \ref{sub7:sol}-\ref{sub7:Rho} (resp. \ref{sub7split:sol}-\ref{sub7split:Rho}) show the evolution of the solution $\neps$ and the total population size $\rhoeps$ during two periods, when a stationary periodic solution is attained, for a short period (resp. large period). One can see that when the period of fluctuations is small, the population remains monomorphic and mostly adapted to the bimodal environment. When the period of fluctuations gets larger, the population has enough time to adapt to the unimodal environment. In this situation, the population becomes dimorphic in the bimodal environment. This may be an effect of the fact that $\varepsilon \neq 0$, so that very small mutations can have an effect on the population dynamics. Overall, these simulations show that when the environment switches between very different phenotypic landscapes, complex phenomena may appear and the trait at which the population concentrates may by hard to predict.

\begin{figure}[!ht]
  \begin{center}
    \subfloat[$\neps^0(x)$ and $R_{1}(x,0)$, $R_{2}(x,0)$]{
\includegraphics[width=0.32\textwidth]{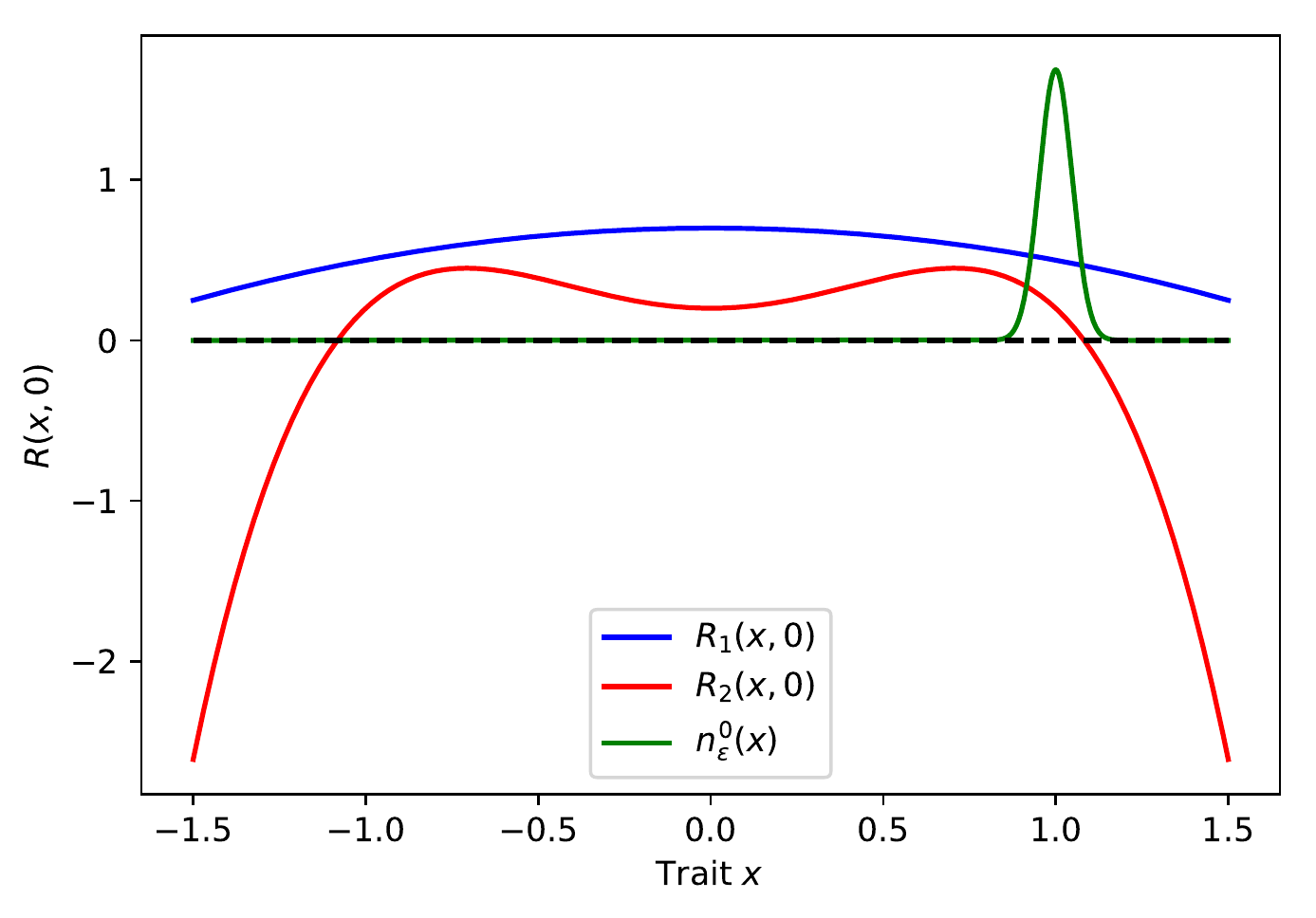} 
      \label{sub7:CI_R}
                         }
                      \rotatebox{90}{\hspace{2cm}$T=1$}
    \subfloat[$\neps(t,x)$ for $t$ large]{
      \includegraphics[width=0.3\textwidth]{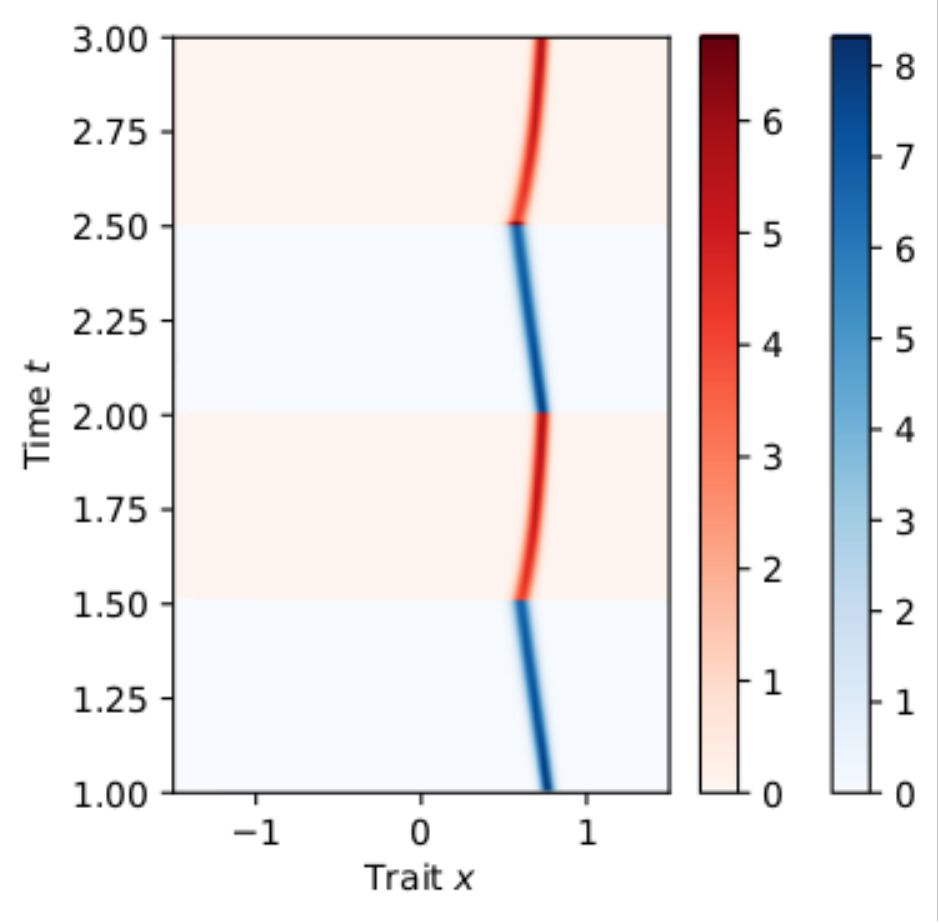}
      \label{sub7:sol}
                         }
    \subfloat[$\rhoeps(t)= \intRd \neps(t,x) \dx x$ for $t$ large]{
      \includegraphics[width=0.3\textwidth]{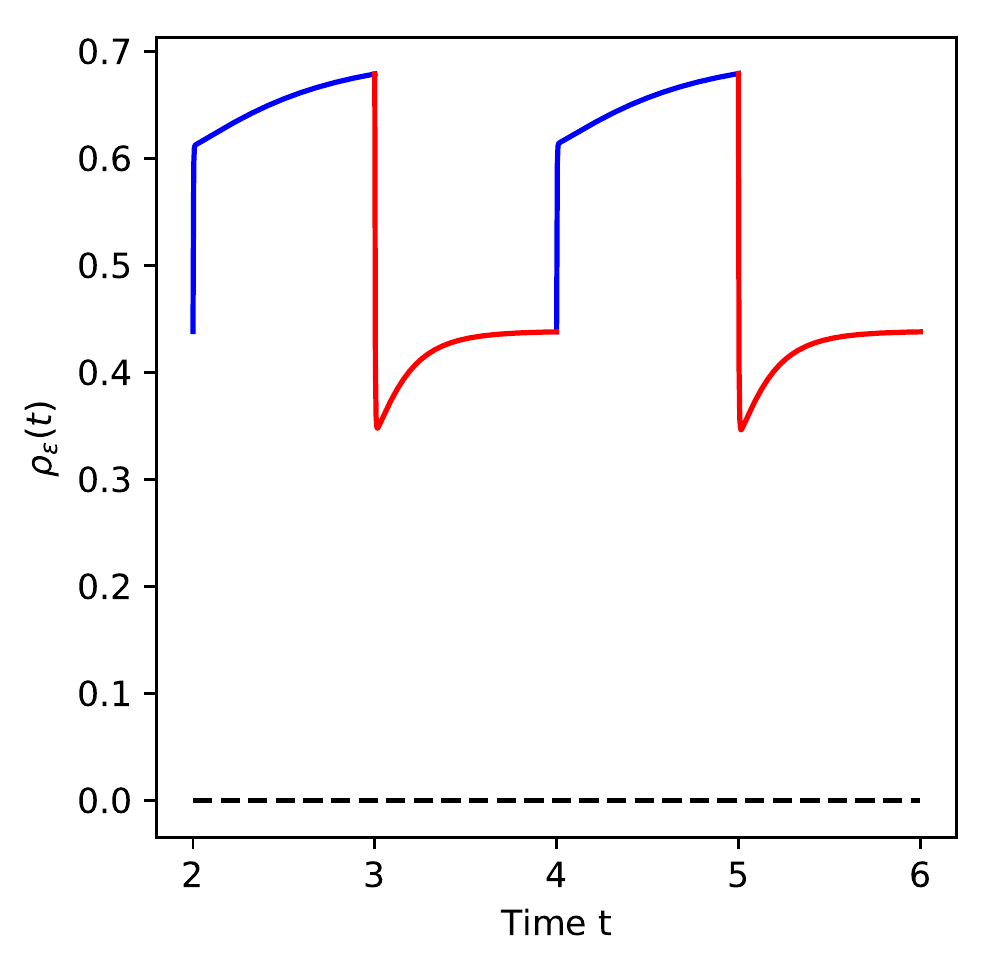}
      \label{sub7:Rho}
                         }\\
                          \rotatebox{90}{\hspace{2cm}$T=10$}
    \subfloat[$\neps(t,x)$ for $t$ small]{
      \includegraphics[width=0.31\textwidth]{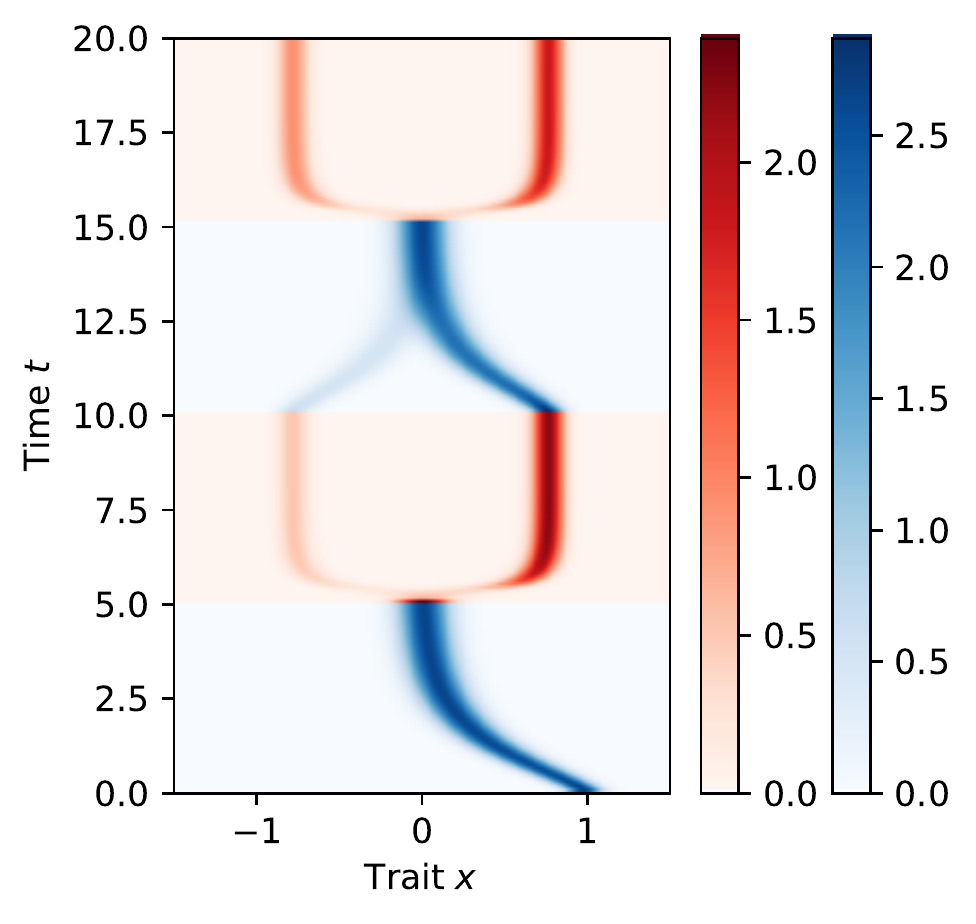}
      \label{sub7split:sol_deb}
                         }               
    \subfloat[$\neps(t,x)$ for $t$ large]{
      \includegraphics[width=0.31\textwidth]{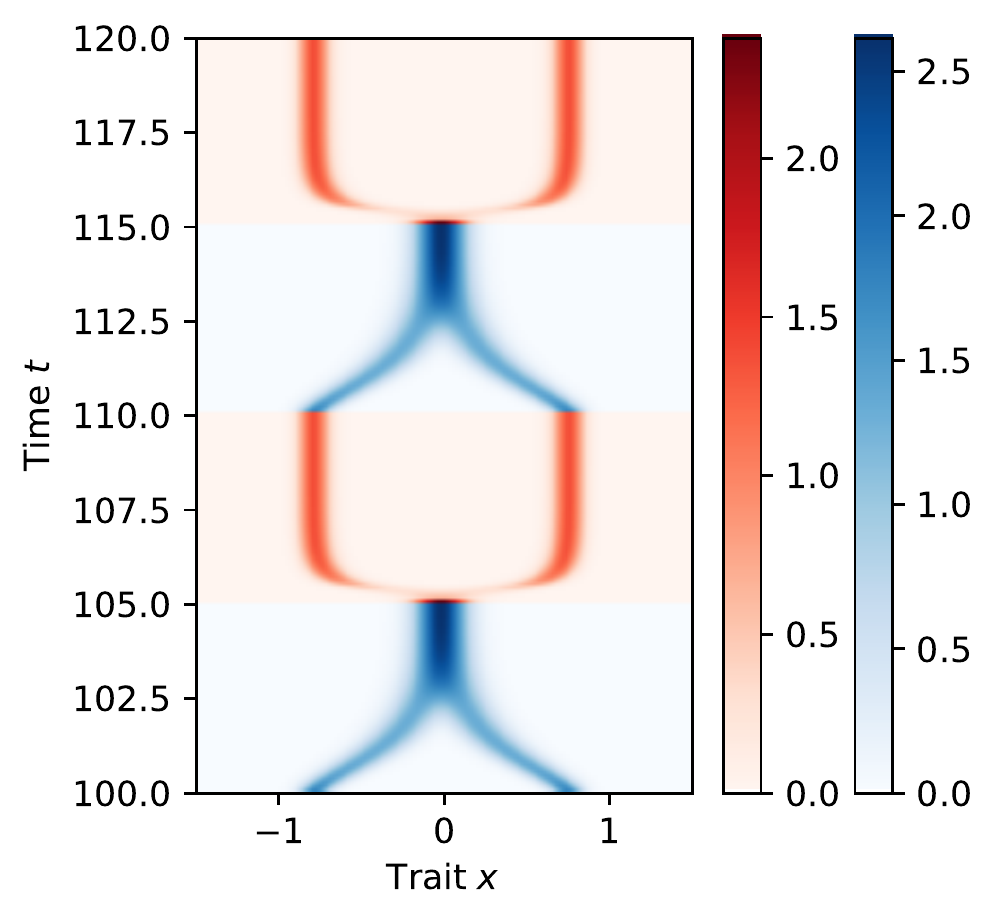}
      \label{sub7split:sol}
                         }
    \subfloat[$\rhoeps(t)= \intRd \neps(t,x) \dx x$ for $t$ large]{
      \includegraphics[width=0.3\textwidth]{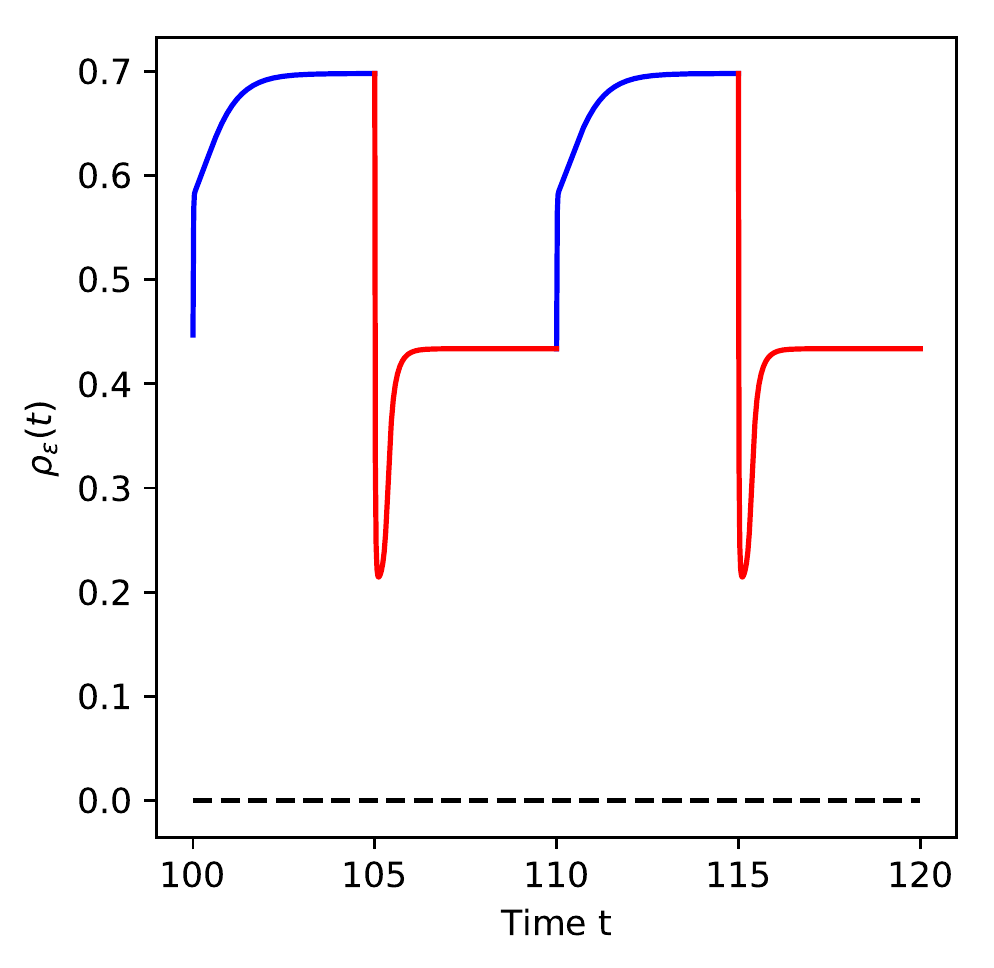}
      \label{sub7split:Rho}
                         }
    \caption{Numerical simulations of \eqref{eq:neps_env} for a growth rate given by \eqref{eq:fig7_R} and the initial condition given by \eqref{eq:fig7_CI}. (a): the initial conditions and the growth rates $R_i$. (b)-(c): long-time periodic solution and the total population size evolution when $T=1$: the switching frequency is high enough to maintain the trait in the same subpart of the phenotypic landscape. (d)-(e)-(f): long-time periodic solution when $T=10$. The switching frequency is low enough to allow for an exploration of the bimodal environment. The apparition of a dimorphic population may occur only for $\varepsilon \neq 0$. Parameters: $\dx x = \varepsilon = 10^{-3}$ ; $\dx t = 10^{-4}$. }
    \label{fig7}
  \end{center}
\end{figure}

\appendix

\section{Proof of Proposition \ref{prop:regularizing_effect}}\label{appendix:prop_regularizing}

\begin{proof} 
\begin{enumerate}[label=\bfseries(\roman*)]
\item \textbf{Upper bound on $\mathbf{(\ueps)_{\varepsilon}}$.} \\
For $(t,x) \in \R_+ \times \Rd$, let us define $\overline{u}(t,x)=-B_2|x|+A_2+ (B_2^2 + K_0)t$. Then by \eqref{hyp:n0}, we have that  $\forall \varepsilon>0 ,\, \forall x\in \Rd$, $\ueps^0(x) \leq \overline{u}(0,x)$, and for $\Ieps$ defined by \eqref{eq:Ieps} and a.e $(t,x) \in \R_+ \times \Rd$, 
\begin{displaymath}
\begin{aligned}
\partial_t \overline{u} - \varepsilon \Delta \overline{u} - |\nabla \overline{u} |^2 - R(x, \Ieps(t)) 
&\geq B_2^2 + K_0 + \varepsilon B_2\frac{d-1}{|x|} - B_2^2- K_0 \geq 0,
\end{aligned}
\end{displaymath}
 using \eqref{hyp:supR}. As a consequence, $\overline{u}$ is a supersolution to \eqref{eq:ueps}. Using a comparison principle in the class of $L^2$ functions, we obtain that for $(t,x)\in \R_+ \times \Rd$,
 \begin{equation*}
 \ueps(t,x) \leq -B_2|x|+A_2+ (B_2^2 + K_0)t\,.
 \end{equation*}
 
\item[] \textbf{Lower bound on $\mathbf{(\ueps)_{\varepsilon}}$.} \\ 
Denote $M_1=\max\left( \frac{\sqrt{K_3}}{2},
B_1\right)$ and define for $(t,x) \in [0,T]\times \Rd$, $\underline{\ueps}(t,x) = -A_1 - M_1 |x|^2 - (2d \varepsilon M_1 + K_2) t$. From \eqref{hyp:n0}, $\underline{\ueps}(0,x) \leq \ueps^0(x)$ on $\Rd$, and \eqref{hyp:supR} yields that 
\begin{displaymath}
\begin{aligned}
\partial_t \underline{\ueps} - | \nabla \underline{\ueps} |^2 - R(x,\Ieps) - \varepsilon \Delta \underline{\ueps} 
& \leq -K_2 - 4 M_1^2 |x|^2 +K_2+K_3  |x|^2  \leq 0,
\end{aligned}
\end{displaymath}
and the lower bound on $\ueps$ follows.  Moreover, the lower bound on $\ueps$ leads  to \eqref{bound-ve}.

\item[\textbf{(b) $\&$ (c)}] \textbf{Regularizing effect in time}
Finally, we show the local uniform continuity in time of $(\ueps)_{\varepsilon}$ on either $[t_0,T]\times B_{L/2}(0)$ or $[0,T]\times B_{L/2}(0)$, depending on the hypothesis on the initial condition. In the following, we work with the general notation $[t_i,T] \times B_L(0)$. \par

Let us show that $\forall \eta >0,\, \exists \theta >0$ such that $\forall (t,s,x) \in [t_i,T]^2 \times B_{L/2}(0)$ with $0\leq t-s \leq \theta$, for all $\varepsilon < \varepsilon_0$, we have that 
\begin{equation*}
|\ueps(t,x) - \ueps(s,x) | \leq 2 \eta.
\end{equation*}

We follow the proofs from \cite[Lemma 9.1]{Barles2002} and \cite[sec. 3.4]{Barles2009}. It consists in using the local uniform $L^\infty$ bounds on $(\ueps)_{\varepsilon}$ and the uniform continuity in space to obtain the uniform local time continuity. 
Take $(s,x) \in [t_i,T[\times B_{L/2}(0)$, and define for $(t,y) \in [s,T[ \times B_L(0)$ and any $\eta>0$,
\begin{displaymath}
\zeta(t,y) = \ueps(s,x) + \eta + A|x-y|^2 + B(t-s),
\end{displaymath}
with $A$ and $B$ constants to be defined. We show that for $A$ and $B$ large enough, $\zeta$ is a strict supersolution to \eqref{eq:ueps} on $[s,T]\times B_L(0)$, and $\zeta(t,y) > \ueps(t,y)$ on $\{s\}\times B(0,L) \cup [s,T]\times \partial B_L(0)$.
First, using point \textit{(a)}, $(\ueps)_{\varepsilon}$ is locally uniformly bounded, so that we can take $A$ such that $\forall \varepsilon < \varepsilon_0$, 
\begin{displaymath}
\frac{8 \parallel \ueps \parallel_{L^{\infty}([t_i,T]\times B_L(0))}}{L^2}\leq A\,.
\end{displaymath}
With this choice, for $(t,y) \in [s,T]\times \partial B_L(0)$ and any $\eta>0$, $B>0$, $\zeta(t,y) > \ueps(t,y)$.

Now, on $\{s\}\times B(0,L)$, we need to show that for $A$ large enough, $\zeta(s,y) > \ueps(s,y)$. Let us proceed by contradiction. If there exists $\eta>0,\, \forall A>0$, $\exists y_{A,\varepsilon} \in B(0,L)$ such that $\zeta(s,y_{A,\varepsilon}) \leq \ueps(s,y_{A,\varepsilon})$, or equivalently
\begin{equation}\label{ineq:reg_time_contradiction}
\ueps(s,y_{A,\varepsilon}) - \ueps(s,x) \geq  \eta + A|x-y_{A,\varepsilon}|^2,
\end{equation}
then we obtain that 
\begin{displaymath}
|x-y_{A,\varepsilon}|\leq \sqrt{\frac{\ueps(s,y_{A,\varepsilon}) - \ueps(s,x) - \eta}{A}} \leq \sqrt{\frac{2 M}{A}}
\end{displaymath}
with $M$ a uniform upper bound on $\parallel \ueps \parallel_{L^{\infty}([t_i,T]\times B(0,L))}$. As a consequence, for all $\varepsilon >0$, $\lim_{A \rightarrow \infty} |x-y_{A,\varepsilon}|=0$. 
Since $(\ueps)_{\varepsilon}$ is uniformly continuous in space on $B(0,L)$, there exists $h>0$ such that for all $\varepsilon >0$, if $|x-y_{A,\varepsilon}|\leq h$, then $|\ueps(s,x)-\ueps(s,y_{A,\varepsilon})| < \frac{\eta}{2}$. This contradicts \eqref{ineq:reg_time_contradiction}, and we deduce that $\zeta(s,y) > \ueps(s,y)$ on $B(0,L)$. 
Finally, we have that in $[s,T]\times B(0,L)$, for $B$ large enough and $C \geq \sup_{I_m\leq \Ieps \leq 2 I_M} \parallel R(y, \Ieps(t))\parallel_{L^{\infty}([s,T]\times B(0,L))}$, 
\begin{displaymath}
\begin{aligned}
\partial_t \zeta(t,y) - \varepsilon  \Delta \zeta(t,y) - |\nabla \zeta(t,y)|^2 - R(y,\Ieps(t)) 
&\geq B - 2A d\varepsilon - 9 A^2 L^2 - C \geq 0 
\end{aligned}
\end{displaymath}
and $\zeta$ is a supersolution of \eqref{eq:ueps}. Now since $\ueps$ is a solution of \eqref{eq:ueps}, we deduce that 
for all $(t,y) \in [s,T]\times B(0,L)$,
\begin{equation*}
\ueps(t,y) \leq \zeta(t,y) = \ueps(s,x) + \eta + A[x-y|^2 + B(t-s).
\end{equation*}
We can prove similarly that, up to changing $A$ and $B$, $\ueps(t,y) - \ueps(s,x) \geq -\eta - A |x-y|^2 - B(t-s)$. We conclude by taking $x=y$ and $\theta < \frac{\eta}{B}$ in both inequalities.

\end{enumerate}

\end{proof}

\section{Proof of Theorem \ref{theo:CV_ueps} (i)-(ii)}\label{app:cv_u}
We use now the regularity properties obtained in Proposition \ref{prop:regularizing_effect} to prove the convergence of $(\ueps)_{\varepsilon}$ and of $(\neps)_{\varepsilon}$ .

\subsection*{Convergence of $(\ueps)_{\varepsilon}$}
From Proposition \ref{prop:regularizing_effect}, we know that $(\ueps)_{\varepsilon}$ is locally uniformly bounded and continuous. We use the ArzelÃ -Ascoli theorem to deduce that up to a subsequence, $(\ueps)_{\varepsilon}$ converges locally uniformly to a continuous function $u$ in $(0,T)\times \Rd$. If moreover $(\nabla \ueps^0)_{\varepsilon}$ is locally uniformly bounded, then $(\ueps)_{\varepsilon}$ is locally uniformly bounded and continuous on $[0,T]\times \Rd$, and the ArzelÃ -Ascoli theorem applied near $t=0$ shows that $u\in \mathcal{C}([0,\infty)\times \Rd)$. In particular, $u(0,x) = \lim_{\varepsilon \rightarrow 0} \ueps(0,x) = u^0(x)$.

\subsection*{Proof of $u\leq 0$}
Assume that for some $(t,x)$, there exists $b$ such that $0<b\leq u(t,x)$. Then, by continuity of $u$ and by locally uniform convergence of $(u_\eps)$, there exists $r>0$ and $\eps_0>0$ such that $\forall (t,y) \in B(x,r)$ and $\eps\leq \eps_0$, $u_\eps(t,y) \geq \frac{b}{2}$. As a consequence, on $B(x,r)$, $\neps(t,y) \rightarrow + \infty $ when $\varepsilon \rightarrow 0$, which contradicts the upper bound \eqref{estimate_Ieps} on $\Ieps$.

\subsection*{Proof of the convergence of $(n_{\varepsilon})_{\varepsilon}$}
We know from Proposition \ref{prop:estimates_ieps} that $n_\eps$ is uniformly bounded in $L^\infty_tL^1_x(\R_+\times\Rd)$. As a consequence, $(\neps)_{\varepsilon}$ converges in $\mathrm{L^\infty}\left(\mathrm{w*} (0,\infty) ; \mathcal{M}^1(\R^d) \right)$ to a measure $n$.
%

\subsection*{Proof of $\text{supp }(n(t,\cdot)) \subseteq \{u(t,\cdot)=0 \}$}
Assume that there exists $x^*\in \text{supp }n(t^*,\cdot)$ such that $u(t^*,x^*)<0$. Then, since $(\ueps)_{\varepsilon}$ is uniformly continuous on a neighborhood of $(t^*,x^*)$, we obtain that for $\varepsilon$ small enough, there exists $a,\,\delta >0$ such that on $V_{\delta}:=(t^*-\delta,t^*+\delta)\times B(x^*,\delta)$, we have that $\ueps(t,x) \leq -\frac{a}{2}<0$. We deduce that 
\begin{displaymath}
\int_{V_\delta} n \dx t \dx x = \int_{V_\delta} \lim_{\varepsilon \rightarrow 0} e^{\frac{\ueps(t,x)}{\varepsilon}} \dx t \dx x = 0\,,
\end{displaymath}
which is a contradiction. Therefore, for almost every $t$, $supp(n(t,\cdot)) \subset \{u(t,\cdot)=0 \}$.

\section{Proof of Theorem \ref{theo:CV_ueps} (iii)}
In this section, we identify $u = \lim_{\varepsilon \rightarrow 0} \ueps$, assuming that $(\Ieps)_{\varepsilon}$ converges to a function $I$. 

\subsection*{Identification of the Hamilton-Jacobi equation \eqref{eq:HJ_theo}}
We define $\Phi_{\varepsilon}(t,x) := \ueps(t,x) - \int_0^t \nabla R(x,\Ieps(s)) \dx s$. From \eqref{eq:ueps}, we deduce that
\begin{equation*}
\begin{split}
\partial_t \phieps(t,x) - \varepsilon \Delta \phieps(t,x) - | \nabla \phieps(t,x)|^2 - 2 \nabla \phieps(t,x) \cdot \int_0^t \nabla R(x, \Ieps(s))\dx s \\
= \varepsilon \int_0^t \Delta R(x,\Ieps(s)) \dx s + \left| \int_0^t \nabla R(x,\Ieps(s)) \dx s \right|^2\,.
\end{split}
\end{equation*}
Our goal is to pass to the limit $\varepsilon \rightarrow 0$. Since $I\mapsto R(x,I)$ is smooth, we obtain the locally uniform limits on $[0,T]$:
\begin{equation*}
\begin{aligned}
\underset{\varepsilon \rightarrow 0}{\lim} \int_0^t R(x, \Ieps(s)) \dx s &= \int_0^t R(x,I(s)) \dx s \,,\\
\underset{\varepsilon \rightarrow 0}{\lim} \int_0^t \nabla R(x, \Ieps(s)) \dx s &= \int_0^t \nabla R(x,I(s)) \dx s \,,\\
\underset{\varepsilon \rightarrow 0}{\lim} \int_0^t \Delta R(x, \Ieps(s)) \dx s &= \int_0^t \Delta R(x,I(s)) \dx s \,.
\end{aligned}
\end{equation*}
These limiting functions are continuous. Moreover, since, $(\ueps)_{\varepsilon}$ converges locally uniformly to the continuous function $u$ as $\varepsilon$ goes to zero, it follows that $(\phieps)_{\varepsilon}$ converges locally uniformly to $\Phi$ with $\Phi(t,x) = u(t,x) - \int_0^t R(x,I(s)) \dx s$ as $\varepsilon$ goes to $0$, and this function is continuous. Next, let us show that $\Phi$ is a viscosity solution of
\begin{equation}\label{eq:viscosity_phi}
\partial_t \Phi(t,x) - | \nabla \Phi(t,x)|^2 - 2 \nabla \Phi(t,x) \cdot \int_0^t \nabla R(x,I(s)) \dx s = \left| \int_0^t \nabla R(x,I(s)) \dx s \right|^2\,.
\end{equation}
Then, it will be straightforward that $u$ is a viscosity solution of 
\begin{equation}\label{eq:viscosity_u}
\partial_t u(t,x) = | \nabla u(t,x)|^2 + R(x,I(t)) \,,
\end{equation}
and the proof will be complete. To show this result, take $\psi \in \mathcal{C}^{\infty}((0,T)\times \Rd)$, and suppose that $\Phi-\psi$ has a strict local maximum at a point $(t_1,x_1) \in (0,T)  \times \Rd$. Now, since $\nabla \phieps(t_{\varepsilon_j} x_{\varepsilon_j}) = \nabla {\psi}(t_{\varepsilon_j} x_{\varepsilon_j})$, $\partial_t \phieps(t_{\varepsilon_j} x_{\varepsilon_j}) = \partial_t {\psi}(t_{\varepsilon_j} x_{\varepsilon_j})$, and $-\Delta \phieps(t_{\varepsilon_j} x_{\varepsilon_j}) \geq - \Delta \psi(t_{\varepsilon_j} x_{\varepsilon_j})$, we deduce that 
\begin{equation*}
\begin{aligned}
\partial_t {\psi}(t_{\varepsilon_j} x_{\varepsilon_j}) - | \nabla {\psi}(t_{\varepsilon_j} x_{\varepsilon_j})|^2 - 2 \nabla {\psi}(t_{\varepsilon_j} x_{\varepsilon_j}) \cdot \int_0^{t_{\varepsilon_j}} \nabla R(x,I(s)) \dx s - \left| \int_0^{t_{\varepsilon_j}} \nabla R(x,I(s)) \dx s \right|^2 \\
= \partial_t \phieps(t_{\varepsilon_j} x_{\varepsilon_j}) - | \nabla \phieps(t_{\varepsilon_j} x_{\varepsilon_j})|^2 - 2 \nabla \phieps(t_{\varepsilon_j} x_{\varepsilon_j}) \cdot \int_0^{t_{\varepsilon_j}} \nabla R(x,I(s)) \dx s - \left| \int_0^{t_{\varepsilon_j}} \nabla R(x,I(s)) \dx s \right|^2\,,\\
=\varepsilon_j \Delta \phieps  + \varepsilon_j \int_0^t \Delta R(x, I(s)) \dx s\,,\\
\leq \varepsilon_j \Delta \psi  + \varepsilon_j \int_0^t \Delta R(x,I(s)) \dx s\,.
\end{aligned}
\end{equation*}
For $\varepsilon_j \rightarrow 0$, and since ${\psi}$ it smooth, it leads to 
\begin{equation*}
\partial_t {\psi}(t_1,x_1) - | \nabla {\psi}(t_1,x_1)|^2 - 2 \nabla {\psi}(t_1,x_1) \cdot \int_0^{t_1} \nabla R(x,I(s)) \dx s - \left| \int_0^{t_1} \nabla R(x,I(s)) \dx s \right|^2 \leq 0\,.
\end{equation*}
The case of a local minimum can be treated similarly. Finally, we have proved that $\Phi$ is a viscosity solution of \eqref{eq:viscosity_phi}. It follows that $u$ is a viscosity solution of \eqref{eq:viscosity_u}.

\subsection*{Constraint when $I$ is strictly positively lower bounded on $(0,T)$}
We assume that there exists $\underline{I}>0$ such that for $t\in (0,T)$, $I(t) \geq \underline{I}$. In that case, we show that 
\begin{displaymath}
\max_{x\in \Rd} u(t,x) =0\,\forall t \in (0,T)\,.
\end{displaymath}
Using the upper bound in \eqref{eq:bounds_ueps}, there exist positive constants such that 
\begin{displaymath}
\ueps(t,x) \leq -A|x| + B + Ct\,.
\end{displaymath}
Therefore, for $M$ large enough, we obtain that 
\begin{displaymath}
\lim_{\varepsilon \rightarrow 0} \int_{|x|>M} \neps(t,x) \dx x \leq \lim_{\varepsilon \rightarrow 0} \int_{|x|>M} e^{\frac{-A|x| + B + Ct}{\varepsilon}} \dx x =0 \,,
\end{displaymath}
so that we can write 
\begin{equation}\label{eq:contradiction_max_u}
\lim_{\varepsilon \rightarrow 0} \int_{|x|\leq M} \neps(t,x) \dx x \geq \frac{\underline{I}}{\psi_M}\,.
\end{equation}
Now, if for all $|x|\leq M$ $u$ is strictly negative, there exists a positive constant $a$ such that $u(t,x) <-a$ for all $|x|\leq M$. By the locally uniform convergence of $(u_\eps)_{\eps}$ to $u$ we then deduce that $\lim_{\varepsilon \rightarrow 0} \int_{|x|\leq M} \neps(t,x) \dx x=0$, which contradicts \eqref{eq:contradiction_max_u}, and the result follows.

\subsection*{Proof of $\Gamma_t \subseteq \{x\in\Rd,\, R(x,I(t))=0\}$}
Take $t$ any continuity point of $I$, and $\overline{x}$ such that $u(t,\overline x)=0$. Using the definition of a viscosity solution at this point with the null function as a test function, we obtain that 
\begin{displaymath}
R(\overline{x},I(t))\geq 0\,.
\end{displaymath} 
Next for the other inequality, integrate Equation \eqref{eq:HJ_theo} in time on $(t,t+s)$ with $s>0$ small enough, at the fixed point $\overline{x}$, and divide by $s$. We obtain that
\begin{displaymath}
0\geq \frac{ u(t+s,\overline{x})}{s} \geq \frac{u(t+s,\overline{x}) - u(t,\overline{x})}{s} \geq \frac{1}{s} \int_0^s R(\overline{x},I(t+u)) \dx u \,,
\end{displaymath}
and for $s \rightarrow 0^+$, since $I$ is continuous in $t$, we have that 
\begin{displaymath}
0 \geq  R(\overline{x},I(t)) \,,
\end{displaymath}
and that concludes the proof of Theorem \ref{theo:CV_ueps}.

\section*{Acknowledgements}
M.C. and S.M. are grateful for partial funding from the chaire Mod\'elisation Math\'ematique et Biodiversit\'e of V\'eolia Environment - \'Ecole Polytechnique - Museum National d'Histoire Naturelle - Fondation X.
S.M. is also grateful for partial funding from the
 European Research Council (ERC) under the European Union's Horizon 2020 research and innovation programme (grant agreement No 639638), held by Vincent Calvez.

%

\end{document}